\documentclass[11pt]{article}

\usepackage{mathrsfs}
\usepackage{amssymb}
\usepackage{amsmath}
\usepackage{bbm,bm}
\usepackage{multirow}
\usepackage{booktabs}
\usepackage{graphicx}
\usepackage{enumerate}
\usepackage{graphicx}
\usepackage{color}
\usepackage{hyperref}
\usepackage{stmaryrd}
\usepackage{xcolor}
\usepackage{comment}
\usepackage{epstopdf}
\usepackage{amsthm}
\usepackage{mathrsfs}
\usepackage{relsize}
\usepackage{mwe} 

\SetSymbolFont{stmry}{bold}{U}{stmry}{m}{n}

\hypersetup{
    colorlinks,
    linkcolor={red!50!black},
    citecolor={blue!50!black},
    urlcolor={blue!80!black}
}
\numberwithin{equation}{section}

\theoremstyle{plain}
\newtheorem{theorem}{Theorem}[section] 
\newtheorem{lemma}[theorem]{Lemma}     

\newtheorem{assumption}[theorem]{Assumption}
\theoremstyle{remark}
\newtheorem*{remark}{Remark}

\newcommand\rd{\mathrm{d}}
\newcommand{\boundr}{{\rm R}}

\DeclareMathAlphabet\mathbcal{OMS}{cmsy}{b}{n}

\DeclareSymbolFont{operators}   {OT1}{cmr} {m}{n}
\DeclareSymbolFont{letters}     {OML}{cmm} {m}{it}
\DeclareSymbolFont{symbols}     {OMS}{cmsy}{m}{n}
\DeclareSymbolFontAlphabet{\mathrm}    {operators}
\DeclareSymbolFontAlphabet{\mathnormal}{letters}
\DeclareSymbolFontAlphabet{\mathcal}   {symbols}
\DeclareMathAlphabet      {\mathbf}{OT1}{cmr}{bx}{n}
\DeclareMathAlphabet      {\mathsf}{OT1}{cmss}{m}{n}
\DeclareMathAlphabet      {\mathit}{OT1}{cmr}{m}{it}
\DeclareMathAlphabet      {\mathtt}{OT1}{cmtt}{m}{n}

\DeclareMathAlphabet      {\mathjc}{OT1}{dutchcal}{bx}{n}

\newcommand{\bB}[1]{\boldsymbol{#1}}
\newcommand{\poln}{p}
\newcommand{\bpoln}{\boldsymbol{p}}
\newcommand{\spacestep}{h}
\newcommand{\nsp}{n_{\bB{c}}}
\newcommand{\ssp}{n_{\bB{s}}}
\newcommand{\nrr}{n_{r}}

\newcommand{\vhs}{v_{\rm hs}}

\usepackage[labelfont=bf,margin=1cm,font=rm]{caption}

\usepackage{enumitem}

\begin{document}

\title{Invariant-region-preserving high-order schemes for a model of reactive sedimentation}

\author{
	{\sc Juan Barajas-Calonge}\thanks{GIMNAP-Departamento de Matem\'{a}tica, Universidad del B\'{\i}o-B\'{\i}o, Chile and Departamento de Ciencias Naturales y Tecnología, Universidad de Aysén, Chile, email: {\tt juan.barajas2001@alumnos.ubiobio.cl}.}
	\quad
	{\sc Julio Careaga$^{*}$}\thanks{Bernoulli Institute, University of Groningen, The Netherlands, email: {\tt j.c.careaga.solis@rug.nl}.\protect\footnotemark[1]} 
	\quad
    {\sc Luis-Miguel Villada}\thanks{GIMNAP-Departamento de Matem\'atica, Universidad del B\'io-B\'io, Chile and   CI$^{\mathrm{2}}$MA,
    	Universidad de Concepci\'{o}n,  Casilla 160-C, Concepci\'{o}n, Chile, email: {\tt lvillada@ubiobio.cl}.}
}
\date{ }

\renewcommand{\thefootnote}{\fnsymbol{footnote}}
\setcounter{footnote}{1}

\footnotetext[1]{Corresponding author}

\maketitle
\begin{abstract}

\noindent
Reactive sedimentation of mixtures composed by biological and inert solid matter including liquid substrates is modelled by a system of convection-diffusion-reaction equations. In this process, the small solid particles dispersed in the viscous fluid settle due to gravity force while simultaneous chemical reactions take place between solid and liquid components (substrates). Among the applications of reactive sedimentation processes, the principal use is in the simulation and control of secondary settling tanks (SSTs) in water resource recovery facilities (WRRFs).
The spatially 1D governing equations include a nonlinear and non-homogeneous strongly degenerate parabolic equation for the total concentration of solids, and transport equations including reaction terms for the solid components and substrates.
To accurately approximate the model equations, a high-order finite volume scheme is developed making use of polynomial reconstructions.
Two reconstruction methods are used, a second-order monotonic upstream-centered scheme (MUSCL), and a third-order central weighted essentially non-oscillatory (CWENO3) method. In addition, for the time discretization, we employ strong stability preserving Runge-Kutta (SSPRK) schemes of second- and third-order.
The CWENO3 spatial discretization is complemented with limiters, and in the case of zero  diffusion, the designed high-order schemes (MUSCL and CWENO3) are shown to preserve an invariant-region property.
This property is particularly relevant as it ensures that physically relevant numerical solutions are obtained at each time iteration. In addition, high-order approximations for the diffusion terms are proposed to approximate the complete model equations.
Simulations of denitrification in SSTs  under a continuously operated regime in WRRFs demonstrate the performance of the model and its discretization. Finally, convergence tests experimentally show the order of accuracy of the developed scheme and the invariant-region preservation.

\bigskip
{\small
}
\end{abstract}

\section{Introduction}

\subsection{Scope}

Reactive sedimentation describes the process in which a mixture of solid particles and a liquid phase, including dissolved substrates, begin to settle due to  gravity force, while reactions between the solid and liquid components take place. Motivated by applications in water resource recovery facilities (WRRFs), reactive sedimentation has been shown to be particularly relevant in the simulation and control of secondary settling tanks (SSTs) where biological and inert matter interact with substrates dissolved in water \cite{Cruz2012,FloresAlsina2012,Henze2000}. The suspension within the SST, also called activated sludge, comprises a large amount of heterotrophic bacteria, typically considered as solid particles in mathematical models, which consume dissolved substrates and produce various compounds during and after metabolism. Among the biokinetic reactions between bacteria and substrates is the denitrification process, in which bacteria use nitrate ($\rm NO_3$) or nitrite ($\rm NO_2$) to produce nitrogen ($\rm N_2$).
In an SST, the aim is to recover clean water from the overflow (effluent) located at the top of the tank, and remove the concentrated sediment from the underflow at the bottom, see Figure~\ref{fig:fig1}. Then, simulations of reactive sedimentation in SSTs have to ensure a balance between the bulk flows (feed, effluent and underflow), while all concentrations remain positive, and the total concentration of solids does not exceed a maximum packing threshold. This motivates the use of the so-called invariant-region-preserving schemes \cite{barajas2025invariant,SDIMA_MOL}, which ensure that the approximated quantities remain physically relevant.
\begin{figure}[t]
\centering
 \includegraphics[scale=0.7]{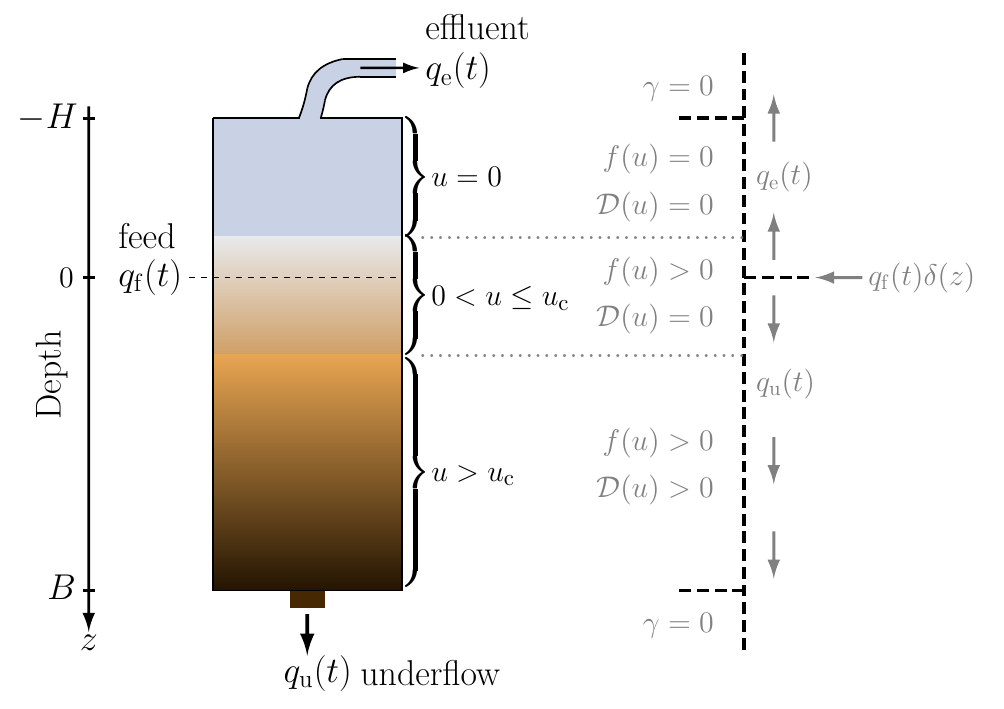}
 \caption{Schematic of a secondary settling tank (SST). The depth $z$ corresponds to the spatial coordinate, ranging from $z=-H$ to $z=B$ for the inner tank, and $z<-H$ and $z\geq B$ for the pipes. The colours in the SST represent a profile for the total solids concentration $u$. The diagram to the right of the SST shows the flux $f$, compression function $\mathcal{D}$ and bulk velocities $q_{\rm u}$ (underflow), $q_{\rm e}$ (effluent) and $q_{\rm f}$ (feed) for the particular concentration profile $u$ along the $z$-coordinate.\label{fig:fig1}}
\end{figure}
Further applications of reactive sedimentation in the wastewater treatment field can be found in the so-called sequencing batch reactors (SBRs), in which a series of settling tanks are operated in a continuous process of feed and extraction of solid and liquid matter \cite{SDAMM_SBR1,SDAMM_SBR2}.

The reactive sedimentation model in \cite{SDIMA_MOL}, which finds its basis on the underlying interaction between physical and biochemical processes, takes into account the highly nonlinear evolution of sedimentation and the interconnected reactions between solid and liquid components. In turn, the spatially 1D convection-diffusion-reaction modelling equations \cite{SDIMA_MOL}, in terms of the depth coordinate $z\in\mathbb{R}$ and time $t>0$, can be written as follows:
\begin{subequations} \label{syst:introductory:system}
\begin{align}
 & \dfrac{\partial\bB{c}}{\partial t}  + \dfrac{\partial}{\partial z} \big(\mathcal{\bB{V}}_{\bB{c}}(u,z,t)\bB{c}\big)  = \dfrac{\partial}{\partial z} \Big(\gamma(z)\dfrac{\partial}{\partial z}\mathcal{B}(u)\bB{c}\Big) +  \bB{b}_{\bB{c}}(\bB{c},\bB{s},z,t),\label{eq:introductory:c}\\
& \dfrac{\partial\bB{s}}{\partial t}  + \dfrac{\partial}{\partial z} \big(\mathcal{\bB{V}}_{\bB{s}}(u,z,t)\bB{s}\big) =
-\dfrac{\partial}{\partial z} \Big(\dfrac{ \gamma(z) \partial_z \mathcal{B}(u)u}{\rho - u}\bB{s}\Big) +\bB{b}_{\bB{s}}(\bB{c},\bB{s},z,t),\label{eq:introductory:s}\\
& \bB{c} = \big(c^{(1)},\dots, c^{(\nsp)}\big)^{\rm T},\quad \bB{s} = \big(s^{(1)},\dots, s^{(\ssp)}\big)^{\rm T},\quad u= u(z,t) = \sum_{i=1}^{n_{\bB{c}}} c^{(i)}(z,t), \label{eq:introductory:u}
\end{align}
\end{subequations}
where the unknowns are the vector of solid components $\bB{c} = \bB{c}(z,t)\in \mathbb{R}^{\nsp}$, vector of substrate components $\bB{s}=\bB{s}(z,t)\in\mathbb{R}^{\ssp}$, with $\nsp,\ssp\in\mathbb{N}$, and the total concentration of solids $u = u(z,t)$ which is a scalar quantity. Scalar functions $\mathcal{\bB{V}}_{\bB{c}}$ and $\mathcal{\bB{V}}_{\bB{s}}$ are related to convective velocities for the solids and liquid phases, respectively, while $\mathcal{B}$ accounts for sediment compressibility. Functions $\bB{b}_{\bB{c}}\in \mathbb{R}^{\nsp}$ and $\bB{b}_{\bB{s}}\in \mathbb{R}^{\ssp}$ group both reaction and feed-related terms (including a point source), $\gamma$ is a characteristic function, which equals one inside the vessel and zero otherwise, and $\rho$ is the density of the solid phase. Suitable initial conditions $\bB{c}(z,0) = \bB{c}_0(z)$, $\bB{s}(z,0) = \bB{s}_0(z)$ and $u(z,0) = u_0(z)$ satisfying \eqref{eq:introductory:u} for $z\in \mathbb{R}$ are supplemented to \eqref{syst:introductory:system}, while no boundary conditions need to be imposed.
All model ingredients, including the specific definition of functions, parameters, and constants of the system of partial differential equations (PDEs) \eqref{syst:introductory:system}, are described in detail in Section \ref{sec:model:equations}.

Although the convective terms in Equations \eqref{eq:introductory:c} and \eqref{eq:introductory:s} are linear with respect to the respective vector of unknowns, the total solids concentration $u$ defined by \eqref{eq:introductory:u} is modelled by a non-linear convection-diffusion-reaction equation which is the result of the addition of each component of \eqref{eq:introductory:c}.
Furthermore, this equation is parabolic strongly degenerate, which becomes hyperbolic for $u$ below a critical value, leading to solutions which may exhibit discontinuities, posing an additional complication to the numerical approximation.
Furthermore, the low regularity related to discontinuous solutions also affects the order of accuracy of a numerical scheme. Therefore, special techniques need to be implemented to achieve high-order such as the so-called polynomial reconstructions \cite{Liu1994,Jiang1996}. Polynomial reconstructions like the monotonic upstream-centered scheme (MUSCL) \cite{kolgan1972application,van1979towards} and weighted essentially non-oscillatory (WENO) reconstructions \cite{Levy2000} allow to maintain a high-order accuracy, while keeping a good approximation of discontinuities.

The main aim of this work is to develop a high-order numerical scheme to approximate \eqref{syst:introductory:system} making use of polynomial reconstructions for $\bB{c}$, $\bB{s}$ and $u$, and that at the same time satisfies an invariant-region-preserving property for the case $\mathcal{B} = 0$.
For this, we will focus on the second-order MUSCL and a third-order WENO reconstruction, which combines linear and quadratic polynomials. The reconstructed polynomials of each solution component will be specially designed such that the invariant-region property is achieved, including the use of a modified version of the well known Zhang and Shu limiters \cite{zhang2011positivity,zhang2010maximum}, which allow to preserve positivity and the maximum principle.
Furthermore, suitable strong stability preserving Runge-Kutta (SSPRK) methods will be employed for the time approximations.

\subsection{Related works}

Models of continuous sedimentation of monodispersed particles in the non-reactive case and in one spatial dimension can be found in \cite{Chancelier1994, Diehl1996, DeClercq2008, Burger2005}. For a review on models of non-reactive sedimentation for SSTs, including applications in water science, we refer to \cite{Li2014}. A common constitutive assumption in these models is the settling velocity which nonlinearly depends on the volume fraction of solids particles in agreement with Kynch's theory \cite{Kynch1952}. Furthermore, compression effects where sediment particles get into contact with each other and form a network for concentrations above a critical value $u_{\rm c}>0$ are incorporated as a degenerate diffusion term into the model equation \cite{Burger2005}. Another important complication, as described in \cite{Diehl1996}, is the discontinuous flux function, which arises as a result of the continuously operated regime.

The biokinetic reactions that take place in SSTs are typically described by the comprehensive framework of activated sludge models (ASM) \cite{Henze1987,Henze2000}, which characterize the reactions between the solids and substrates by means of systems of ordinary differential equations (ODEs). The first ODE-based activated sludge model is the ASM1 \cite{Henze1987} and further standardised models are ASM2, ASM2d, and ASM3 \cite{Henze2000}, that differ according to the level of biological complexity they represent \cite{FloresAlsina2012}.
Reduced models focusing on the biokinetic reactions related to nitrification and denitrification processes can be found in \cite{Cruz2012,Jeppsson1993,Ni2008,Julien1998,vanHaandel1981}. For these reduced models, the key variables are the concentrations of diluted nitrate ($\rm NO_3$), nitrite ($\rm NO_2$) and produced nitrogen ($\rm N_2$). One of the first PDE based models for reactive (batch) sedimentation was developed in \cite{SDcace_reactive}, where the governing equations correspond to a convection-diffusion-reaction system of equations written for a reduced biokinetic model of denitrification taking place in the SSTs. This model was later extended to include an arbitrary number of solids and substrate components for the case of continuous sedimentation \cite{SDIMA_MOL,SDm2an_reactive}, which corresponds to the model studied in this work. Applications to sequencing batch reactors can be found in the series of works \cite{SDAMM_SBR1,SDAMM_SBR2,Burger2023}, where the main difficulty is to correctly approximate a moving boundary. Although the numerical methods developed to simulate these reactive sedimentation models are shown to satisfy an invariant-region preserving property, they are only first-order accurate.
Furthermore, the presence of shocks and discontinuities in the solutions require a special treatment in order to accomplish high-order approximations.

For high-order numerical approximations, some of the most widely used techniques for nonlinear conservation laws are the so-called polynomial reconstruction methods \cite{LeVeque1990Green}, which compute the numerical flux by locally using polynomial representations for the left and right interface states from neighbouring solution values. The weighted essentially non-oscillatory (WENO) reconstructions were introduced in \cite{Liu1994}, which are methods designed to avoid spurious oscillations near discontinuities by adaptively combining several candidate reconstructions. Several extensions to WENO schemes have been developed in the literature \cite{Cravero2017, Li2021, zhang2010maximum,zhang2011maximum,Zhu2018}, see also the review provided in \cite{Shu2009}. In \cite{Cravero2017}, a central WENO (CWENO) scheme is introduced, whose main characteristic is that the polynomial reconstructions may include polynomials of different degrees. Multi-resolution WENO schemes are developed in \cite{Li2021, Zhu2018} designed such that high-order accuracy is achieved through adaptive reconstruction across multiple spatial scales.

Invariant-region properties, which may combine positivity and maximum-principle preservation for more than one variable in a PDE system, plays a crucial role to ensure that numerical solutions are physically relevant in many contexts. For instance, numerical schemes that preserve the positivity of density and pressure have been developed for the Euler equations of compressible gas dynamics \cite{zhang2010positivity,zhang2011positivity}, while positivity-preserving methods have also been proposed for the shallow water equations \cite{xing2013positivity,xing2010positivity} and for polydisperse sedimentation models \cite{barajas2025invariant,barajas2026second}. Finally, a comprehensive overview of property-preserving schemes for conservation laws, together with an extensive list of references, can be found in the monograph by Kuzmin and Hajduk \cite{kuzmin2024property}.

\subsection{Outline of the paper}

The paper is organised as follows. In Section~\ref{sec:model:equations}, we introduce the model of reactive sedimentation, including the governing equations, functions and parameters. The spatial approximation of the model equations is introduced in Section~\ref{sec:numscheme}. The polynomial reconstructions are presented in Section~\ref{sec:polynomial:reconstructions}, where in particular the MUSCL scheme is discussed in Section~\ref{sec:MUSCL} and the CWENO scheme is treated in Section~\ref{sec:CWENO}. Additional limiters for the CWENO reconstructions are included in Section~\ref{sec:limiters}, and the time approximations for the second- and third-order schemes are introduced in Section~\ref{sec:time:approximation}. The theoretical results related to the invariant-region preserving of the high-order scheme for the case without compression effects are included in Section~\ref{sec:invariant-region}. In Section~\ref{sec:numerical:examples}, we present numerical examples to show  the theoretical order of accuracy of the scheme, the invariant-region property and reliability of the developed high-order scheme.
Finally, concluding remarks are included in Section \ref{sec:conclusions}.

\section{Model equations} \label{sec:model:equations}

To model the one-dimensional reactive sedimentation process, we first consider that the sedimentation takes place in a tank whose depth $z$ varies from $z=-H$ (height) to $z=B$ (bottom), with $H,B\geq 0$. Furthermore, we assume that the process is primarily influenced by gravity, so horizontal variations are negligible. For the mixture, we consider $n_{\bB{c}}\in\mathbb{N}$ solid components and $n_{\bB{s}}\in \mathbb{N}$ substrates, defined in vector notation as $\bB{c} := \bigl(c^{(1)},c^{(2)},\dots,c^{(n_{\bB{c}})}\bigr)^{\texttt{t}}$ and $\bB{s} := \bigl(s^{(1)},s^{(2)},\dots,s^{(n_{\bB{s}})}\bigr)^{\texttt{t}}$, respectively, both considered vectors of unknowns.
The vector of solid concentrations $\bB{c}\in\mathbb{R}^{n_{\bB{c}}}$ typically contains concentrations of heterotrophic organisms and non-degradable matter, while the vector of substrates $\bB{s}\in\mathbb{R}^{n_{\bB{s}}}$ contains concentrations of substrates dissolved in water. The total solids' concentration $u$, defined as the sum of components of $\bB{c}$, is also considered a quantity of interest, therefore there are $n_{\bB{c}}+n_{\bB{s}}+1$ unknowns in total, all of them functions depending on the spatial coordinate $z\in \mathbb{R}$ and time $t>0$. The system of PDEs governing this process is given by \cite{SDIMA_MOL}
\begin{subequations} \label{syst:reac}
\begin{alignat}{2}
&\dfrac{\partial c^{(i)}}{\partial t} + \dfrac{\partial}{\partial z}\Big(
\big(q(z,t) + \gamma(z)\bigl(v_{\rm hs}(u) - \partial_z \mathcal{B}(u)\bigr)\big) c^{(i)} \Big)
  = b^{(i)}_{\bB{c}}(\boldsymbol{c},\bB{s},z,t)\,, \quad & \text{for }i=1,\dots,n_{\bB{c}}, \label{eq:reac:a}
\\[1ex]
 &\dfrac{\partial s^{(j)}}{\partial t} + \dfrac{\partial}{\partial z}\left( \left(q(z,t) - \dfrac{\gamma(z)(v_{\rm hs}(u) - \partial_z\mathcal{B}(u))u}{\rho-u}\right)  s^{(j)}\right)  =  b^{(j)}_{\bB{s}}(\boldsymbol{c},\bB{s},z,t)\,, \quad  &\text{for }j = 1,\dots,n_{\bB{s}}, \label{eq:reac:b}\\[1ex]
 & u(z,t)  = c^{(1)} (z,t) + \dots + c^{(n_{\boldsymbol{c}})} (z,t)\,, \quad &\label{eq:reac:c}
\end{alignat}
\end{subequations}
where $v_{\rm hs}:=v_{\rm hs}(u)$ is the hindered settling velocity, defined such that $\vhs\in C^1([0,u_{\max}])$, with $\vhs'(u)\leq 0$ and $\vhs(u_{\max}) = 0\leq \vhs(u)\leq v_0=\vhs(0)$, for $0\leq u\leq u_{\max}$, and there exists a positive constant $L_{\rm v}$ for which
\begin{align} \label{eq:Lipschitz:vhs}
 |\vhs(u)-\vhs(\widetilde{u})|\leq L_{\rm v}|u-\widetilde{u}|,\quad \text{ for all}\quad 0\leq u,\widetilde{u}\leq u_{\max},
\end{align}
which corresponds to the Lipschitz continuity condition of $\vhs$. Furthermore, $q:=q(z,t)$ is the volume-average bulk velocity defined as
\begin{align*}
 q(z,t) = \begin{cases}
           -q_{\rm e}(t) &  \text{ if } z\leq 0,\\
           q_{\rm u}(t) &  \text{ if } z> 0,
          \end{cases}
\end{align*}
where $q_{\rm u} + q_{\rm e} = q_{\rm f}$ with $q_{\rm f}$, $q_{\rm e}$ and $q_{\rm u}$ being the feed, underflow and effluent bulk flows, time-dependent functions respectively. Function $\mathcal{B}:=\mathcal{B}(u)$ describes the sediment compressibility, which is zero for $u<u_{\rm c}$ where $u_{\rm c}\in \mathbb{R}^+$ is a critical concentration above which the particles come into contact, creating a network capable of bearing stress. The constant density of the solid phase is $\rho >0$, and the source terms $\bB{b}_{\bB{c}} = \big(b_{\bB{c}}^{(1)},\dots, b_{\bB{c}}^{(n_{\bB{c}})}\big)^{\tt t}\in\mathbb{R}^{n_{\bB{c}}}$ and $\bB{b}_{\bB{s}}= \big(b_{\bB{s}}^{(1)},\dots, b_{\bB{s}}^{(n_{\bB{s}})}\big)^{\tt t}\in \mathbb{R}^{n_{\bB{s}}}$ are vector functions that involve feed and reaction terms corresponding to equations \eqref{eq:reac:a} and \eqref{eq:reac:b}, respectively defined by
\begin{align*}
 \bB{b}_{\bB{c}}(\bB{c},\bB{s},z,t) &= \delta(z)q_{\rm f}(t)\bB{c}_{\rm f}(t) + \gamma(z)\bB{\sigma}_{\bB{c}}\bB{r}(\bB{c},\bB{s})\,,\\
 \bB{b}_{\bB{s}}(\bB{c},\bB{s},z,t) &= \delta(z)q_{\rm f}(t)\bB{s}_{\rm f}(t) + \gamma(z)\bB{\sigma}_{\bB{s}}\bB{r}(\bB{c},\bB{s})\,.
\end{align*}
Here $\delta$ is the Dirac distribution, $\bB{c}_{\rm f}$ and $\bB{s}_{\rm f}$ are the vectors of feed concentration of solids and liquid components, respectively, $ 0 \leq \bB{r}(\bB{c}, \bB{s})\in \mathbb{R}^{\nrr}$ is a vector of nonlinear functions modelling the reaction processes, with $\nrr$ a positive integer, $\bB{\sigma}_{\bB{c}}\in \mathbb{R}^{ n_{\bB{c}}\times \nrr}$ and $\bB{\sigma}_{\bB{s}}\in \mathbb{R}^{n_{\bB{s}}\times \nrr}$ are constant dimensionless stoichiometric matrices. We further define $b_u$ as the sum of the components of $\bB{b}_{\bB{c}}$, that is
\begin{align*}
  b_u(\bB{c},\bB{s},z,t) := \sum_{i = 1}^{n_{\bB{c}}} b^{(i)}_{\bB{c}}(\bB{c},\bB{s},z,t) = \delta(z)q_{\rm f}(t)u_{\rm f}(t) + \gamma(z)\bB{1}^{\tt t} \bB{\sigma}_{\bB{c}}\bB{r}(\bB{c},\bB{s})\,,
\end{align*}
where $\bB{1}$ is the vector of ones in $\mathbb{R}^{n_{\bB{c}}}$.
Moreover, we define the maximum concentration of solids $u_{\max}\in\mathbb{R}^+$ and assume that $u_{\max}<\rho$. We observe that adding up \eqref{eq:reac:a} and using \eqref{eq:reac:c} one obtains the following non-linear scalar equation for the total concentration $u$:
\begin{align}
 &\dfrac{\partial u}{\partial t} + \dfrac{\partial}{\partial z}\Big( \big(q(z,t) + \gamma(z)\bigl(\vhs(u) -  \partial_z \mathcal{B}(u)\bigr)\big) u \Big) = b_{u}(\bB{c},\bB{s},z,t)\,, \label{eq:EDP:u}
\end{align}
where the product $v_{\rm hs}(u)u = f(u)$ corresponds to the nonlinear batch flux function \cite{Kynch1952}, and the compression function for the total solids concentration is defined by a function $\mathcal{D}:=\mathcal{D}(u)$ such that $\partial_z\mathcal{D}(u) = (\partial_z\mathcal{B}(u))u$. In the non-reactive case \cite{Burger2005}, function $\mathcal{D}$ is defined depending on $v_{\rm hs}$ and the effective solids stress $\sigma_{\rm e}:=\sigma_{\rm e}(u)$, which is positive only when $u>u_{\rm c}$ and $\sigma_{\rm e}(u)=0$ for $u\leq u_{\rm c}$. These compression-related functions are then defined by the integrals
\begin{align*}
 \mathcal{D}(u) :=  \displaystyle\int_{u_{\rm c}}^{u} d(s)\rd s\qquad \text{and}\qquad   \mathcal{B} (u) = \displaystyle\int_{u_{\rm c}}^{u} \dfrac{d(s)}{s}\rd s\quad\mbox{with}\quad d(u) := \dfrac{\rho}{ g (\rho-\rho_{\rm L})} v_{\rm hs}(u)\sigma_{\rm e}'(u)\,,
\end{align*}
where $0<\rho_{\rm L}<\rho$ is the density of the liquid phase. Note that $\vhs$ is typically assumed to be zero for values above $u_{\max}$. Equation~\eqref{eq:EDP:u} is then strongly degenerate parabolic due to the fact that $\mathcal{D}$ vanishes in the $u$-interval $[0,u_c]$ and it degenerates into a first-order hyperbolic equation. The final system of equations for the unknowns $u$, $\bB{c}$ and $\bB{s}$  is given by
\begin{subequations}\label{syst:final}
\begin{align}
&\dfrac{\partial u}{\partial t} + \dfrac{\partial}{\partial z}\Big( \big(q(z,t) + \gamma(z)\bigl(\vhs(u) -  \partial_z \mathcal{B}(u)\bigr)u\big) \Big) = b_{u}(\bB{c},\bB{s},z,t)\,, \label{eq:model:u}\\
  &\dfrac{\partial \bB{c}}{\partial t} + \dfrac{\partial}{\partial z}\Big(
\big(q(z,t) + \gamma(z)\bigl(\vhs(u) -  \partial_z \mathcal{B}(u)\bigr)\big) \bB{c}\Big)
  = \bB{b}_{\bB{c}}(\bB{c},\bB{s},z,t)\,, \label{eq:model:c}\\
&\dfrac{\partial \bB{s}}{\partial t} + \dfrac{\partial}{\partial z}\left( \left(q(z,t) - \gamma(z)\dfrac{(v_{\rm hs}(u) - \partial_z\mathcal{B}(u))u}{\rho-u}\right)  \bB{s}\right)  = \bB{b}_{\bB{s}}(\bB{c},\bB{s},z,t)\,. \label{eq:model:s}
\end{align}
\end{subequations}
Note that in the system above, condition~\eqref{eq:reac:c} has been used to obtain Equation~\eqref{eq:model:u}, hence \eqref{syst:reac}
implies \eqref{syst:final}. On the contrary, the opposite implication to go from \eqref{syst:final} to \eqref{syst:reac} can be understood in the following way. Adding up \eqref{eq:model:c} by components and subtracting \eqref{eq:model:u} we get:
\begin{align} \label{eq:difference:advection}
 \dfrac{\partial }{\partial t}\vartheta + \dfrac{\partial}{\partial z} \Big( \big(q(z,t) + \gamma(z)\bigl(\vhs(u) -  \partial_z \mathcal{B}(u)\bigr)\big) \vartheta \Big) = 0,
\end{align}
with $\vartheta := u - (c^{(1)} + \cdots + c^{(n_{\bB{c}})})$. The new variable $\vartheta$ satisfies then the linear advection equation \eqref{eq:difference:advection} with homogeneous initial and boundary conditions. Now, under the assumption that $u$ is a given function and that the unique solution to \eqref{syst:final} is the homogeneous solution $\vartheta = 0$, then condition \eqref{eq:reac:c} follows directly, and therefore in this case system \eqref{syst:final} would imply \eqref{syst:reac}. A proper proof requires to study the entropy solutions to \eqref{syst:reac} and regularity of $ q(z,t) + \gamma(z)\bigl(\vhs(u) -  \partial_z \mathcal{B}(u)\bigr)$, which is out of the scope of this work. In \cite{SDIMA_MOL}, it is shown that choosing an appropriate first-order numerical scheme for \eqref{syst:final}, condition \eqref{eq:reac:c} is fulfilled at every time iteration. For the high-order numerical scheme described in Section~\ref{sec:numscheme}, we are going to also show that condition \eqref{eq:reac:c} is fulfilled for the case when the compression effects are not considered.

\noindent To simplify the notation used in the spatial discretization of system \eqref{syst:final} in Section~\ref{sec:numscheme}, we introduce the convective velocity, diffusion-related fluxes and total flux of \eqref{eq:model:u}:
\begin{alignat}{2}
\mathcal{V} &:=\mathcal{V}(u,z,t) = q(z,t) + \gamma(z)\vhs(u)\,, & \label{eq:def:V} \\
\mathcal{J} &:=\mathcal{J}(u,\partial_zu,z) = \gamma(z)\partial_z\mathcal{D}(u)\,, \label{eq:def:J}\\
\mathcal{E} &:=\mathcal{E}(u,\partial_zu,z) = \gamma(z)\partial_z\mathcal{B}(u)\,, \label{eq:def:E}\\
\mathcal{F} &:=\mathcal{F}(u,z,t) = \big(\mathcal{V}(u,z,t) - \mathcal{E}(u,\partial_z u, z)\big)u \,. & \label{eq:def:F}
\end{alignat}
In addition, the vector fluxes corresponding to \eqref{eq:model:c} and \eqref{eq:model:s}, are respectively defined by
\begin{align*}
 \mathbcal{G}^{\bB{c}}(\bB{c},u,\partial_zu,z,t) &:= \big(\mathcal{V}(u,z,t) - \mathcal{E}(u,\partial_zu,z)\big)\bB{c}\,,\\
 \mathbcal{G}^{\bB{s}}(\bB{s},u,\partial_zu,z,t) &:= \bigg(\rho q(z,t) - \mathcal{V}(u,z,t) u + \mathcal{E}(u,\partial_zu,z) u \bigg)\dfrac{\bB{s}}{\rho - u}\,.
\end{align*}
The set of physically relevant solutions (invariant-region set) is defined as:
\begin{align}\label{eq:def:invariant:region:set:W}
 \mathbcal{W} := \Big\{(u,\bB{c},\bB{s})\in \mathbb{R}\times \mathbb{R}^{\nsp}\times \mathbb{R}^{\ssp}:\quad 0\leq u\leq u_{\max},\quad \bB{1}^{\tt t}\bB{c} = u,\quad \bB{c}\geq \bB{0},\quad \bB{s}\geq\bB{0}\Big\}.
\end{align}
Finally, we summarise all the hypotheses related to the reaction terms required for our model in the following assumption.

\begin{assumption}\label{asumption:reactionterms}
	Let $\bB{\xi}\in\{\bB{c},\bB{s}\}$ and consider the corresponding number of components $n_{\bB{\xi}}$ ($n_{\bB{c}}$ or $n_{\bB{s}}$ depending on $\bB{\xi}$). Given $k=1,\dots, n_{\bB{\xi}}$, let $\mathcal{J}_{\bB{\xi},k}^-$ and $\mathcal{J}_{\bB{\xi},k}^+$ be respectively the set of negative and non-negative  coefficients related to the $k$ row of the stoichiometric matrix $\sigma_{\bB{\xi}}$:
\begin{align*}
\mathcal{J}_{\bB{\xi},k}^-:=\Big\{j\in\mathbb{N}:\,\sigma_{\bB{\xi}}^{(k,j)}<0\Big\},\qquad
\mathcal{J}_{\bB{\xi},k}^+:=\Big\{j\in\mathbb{N}:\,\sigma_{\bB{\xi}}^{(k,j)}\geq 0\Big\}.
\end{align*}
The stoichiometric matrices $\bB{\sigma}_{\bB{c}}$ and $\bB{\sigma}_{\bB{s}}$, and vector of reaction rates $\bB{r}$ satisfy the following hypotheses
	\begin{enumerate} \label{asumption:reactionterms}
		\item The vector of ones belongs to ${\rm ker}(\bB{\sigma}_{\bB{c}}^{\tt t}, \bB{\sigma}_{\bB{s}}^{\tt t})\,$.
		\item For each $j=1,\dots, \nrr$, there holds $r^{(j)}\in \mathrm{C}\big(\bar{\mathbb{R}}_+^{\nsp},\bar{\mathbb{R}}_+^{\ssp}\big)$, with $r^{(j)}$ locally Lipschitz continuous.
		\item Given $k=1,\dots,n_{\bB{\xi}}$, for every $j\in \mathcal{J}_{\bB{\xi},k}^-$ there exists a bounded function $\bar{r}_{\bB{\xi}}^{(j)}\in  \mathrm{C}\big(\bar{\mathbb{R}}_+^{\nsp},\bar{\mathbb{R}}_+^{\ssp}\big)$ with $\bar{r}_{\bB{\xi}}^{(j)}(\bB{c},\bB{s})>0$ if $\bB{c}>\bB{0}$ and $\bB{s} > \bB{0}$, and $\bar{r}_{\bB{\xi}}^{(j)}(\bB{c},\bB{s})=0$ if $\bB{c} = \bB{0}$ and $\bB{s}=\bB{0}$, such that
		\begin{align*}
			r^{(j)}(\bB{c},\bB{s}) =  \bar{r}_{\bB{\xi}}^{(j)}(\bB{c},\bB{s})\xi^{(k)}.
		\end{align*}
		There exists an upper bound $\boundr>0$ such that $\bar{r}^{(j)}_{\bB{\xi}}(\bB{c},\bB{s}) \leq \boundr$ for all $j\in \mathcal{J}_{\bB{\xi},k}^-$ and all $(u,\bB{c},\bB{s})\in \mathbcal{W}$.
		\item There exists an $\boundr_{\max} > 0$ such that each component of $|\bB{\sigma}_{\bB{\xi}}\bB{r}(\bB{c},\bB{s})|$ is bounded by $\boundr_{\max} $ for all $(u,\bB{c},\bB{s}) \in \mathbcal{W}$.
		\item There exists an $\varepsilon > 0$ such that
		\begin{align*}
			\bB{\sigma}_{\bB{c}}\bB{r}(\bB{c},\bB{s}) = \bB{0}\,\quad \text{for all}\quad u = \bB{1}^{\tt t}\bB{c} =  c^{(1)}+\cdots +  c^{(n_{\bB{c}})}\, \geq\, u_{\max} - \varepsilon.
		\end{align*}
	\end{enumerate}
\end{assumption}

\begin{remark}
The first assumption is made in order to guarantee conservation of mass, while the second one is to secure that in absence of spatial variations, the resulting ODE system is well-posed. The third assumption ensures that the reaction terms tend to zero when the vectors of concentration are zero and gives a bound to control the reaction terms having negative stoichiometric indexes. Assumptions 4 and 5 are made to properly bound the reaction terms when showing the invariant-region property in Section~\ref{sec:invariant-region}.
\end{remark}

\section{Spatial approximation}\label{sec:numscheme}

We begin by discretizing the spatial domain within the tank $\Omega:=[-H,B\,]$ into $N\in\mathbb{N}$ cells of length $\spacestep = (H+B)/N$ defined by the intervals $ [z_{j-1/2}, z_{j+1/2}]=:I_j$, where the cell-limits are $z_{j+1/2} = -H+ j\spacestep$, with $j=0,\dots, N$. The cell centre of each $j$-cell is denoted by $z_j$. The unknowns at each $j$-cell are approximated by cell sliding averages of the form
\begin{align*}
 u_j(t) = \dfrac{1}{\spacestep}\int_{I_j} u(z,t){\rm d}z\,,\qquad j = 0,\dots,N\,,
\end{align*}
where $u_j(t)\approx u(z,t)$ for all $z\in I_j$ and time $t>0$. Similarly, vectors $\bB{c}_j(t)$ and $\bB{s}_{j}(t)$ are defined component-wise as cell averages approximating $\bB{c}(z,t)$ and $\bB{s}(z,t)$ for all $z\in I_j$ and time $t>0$, respectively. Then, system \eqref{syst:final} is approximated by a finite volume numerical scheme based on piece-wise constant solutions. The discretized indicator function is given by
\begin{align*}
 \gamma_{j+1/2} = \gamma(z_{j+1/2}) : = \begin{cases}
                     1 & \mbox{if }z_{j+1/2} \in (-H,B)\,,\\[1ex]
                     0 & \mbox{otherwise}\,, \\[1ex]
                    \end{cases}
\end{align*}
and the discretized Kronecker delta is
\begin{align*}
\delta_{j,j_{\rm f}}:=\begin{cases}
                       1 & \mbox{if }j=j_{\rm f}\,,\\[1ex]
                         0 & \mbox{otherwise}\,,
                      \end{cases}
\end{align*}
where $j_{\rm f} = \lceil H/\spacestep \rceil$. Furthermore, we approximate the known spatially discontinuous and time-dependent function $q$ as
\begin{align}\label{eq:def:qjhalf}
q_{j+1/2}(t) =
\begin{cases}
-q_{\rm e}(t) & \quad j\leq j_{\rm f},\\
q_{\rm u}(t) & \quad j> j_{\rm f}.
\end{cases}
\end{align}
To approximate the convective fluxes in the upcoming sections, we introduce the upwind operator
\begin{alignat}{2}\label{eq:def:upwind}
&{\rm Upw}(a;b,c) := \max\{a, 0\}b + \min\{a, 0\}c\,&&\qquad \text{for } a,b,c\in \mathbb{R}\,,
\end{alignat}
where $a$ in this case plays the role of the convective velocity, while $b$ and $c$ would correspond to the left and right states in a standard discretization, respectively.
Note that operator ${\rm Upw}$ corresponds to a two-point Lipschitz monotone flux, which is nondecreasing in the second argument and nonincreasing in the third argument. We will also use the following notation
\begin{align*}
 a^+ = \max\{a,0\},\qquad a^- = \min\{a,0\},\qquad \text{for }a\in \mathbb{R}.
\end{align*}
Note that with this notation, the upwind operator can be written as
\begin{align*}
 {\rm Upw}(a;b,c) = a^{+} b + a^{-}c.
\end{align*}

\subsection{Total concentration equation}\label{sec:total:concentration:equations}
We begin by describing the finite volume approximation of Equation~\eqref{eq:model:u} which combined with a monotonic upstream-centered scheme (MUSCL) or  weighted essentially non-oscillatory (WENO) reconstructions \cite{Shu1998,Levy2000} gives rise to a high-order scheme for the total concentration $u$. To do so, we integrate Equation~\eqref{eq:model:u} over each cell interval $I_j$ and apply integration by parts to obtain:
\begin{align}\label{fvintegral}
  \int_{I_j}\partial_t u\,{\rm d}z + \widehat{\mathcal{F}}_{j+1/2} - \widehat{\mathcal{F}}_{j-1/2}
  = \int_{I_j} b_u\big(\boldsymbol{c},\bB{s},z,t\big){\rm d}z\,,
\end{align}
for all $j = 1,\dots,N$, where $\widehat{\mathcal{F}}_{j+1/2}$ corresponds to the discretization of the total flux \eqref{eq:def:F} at the cell boundary $z= z_{j+1/2}$ and time $t>0$, respectively, whose precise definitions will be described next. To define $\smash{\widehat{\mathcal{F}}_{j+1/2}}$, we follow the approach from \cite{SDIMA_MOL} and make use of an upwind scheme with discrete velocity $\smash{\widehat{\mathcal{V}}_{j+1/2}-\mathcal{E}_{j+1/2}}$ corresponding to the approximation of $\mathcal{V}$ in \eqref{eq:def:V} and $\mathcal{E}$ in \eqref{eq:def:E}, respectively, which we then combine with polynomial reconstructions. To achieve the desired higher-order approximation, we consider a polynomial reconstruction of $u$ on each interval $I_j$, given by polynomials $\poln^u_j:=\poln^u_j(z)$ of degree $\ell-1$, with $\ell$ a positive integer. Then, for each cell interface $j+1/2$, the left and right polynomial reconstructions are respectively denoted by $\widehat{u}_{j+1/2}^{\rm L}$ and $\widehat{u}_{j+1/2}^{\rm R}$, and the numerical flux is defined by
\begin{align} \label{eq:numfluxF}
\widehat{\mathcal{F}}_{j+1/2}&
:= {\rm Upw}\Big(\widehat{\mathcal{V}}_{j+1/2} - \mathcal{E}_{j+1/2};\,\widehat{u}_{j+1/2}^{\rm L},\,\widehat{u}_{j+1/2}^{\rm R}\Big)\,,
\end{align}
where the discrete velocity $\widehat{\mathcal{V}}_{j+1/2}$ is given by
\begin{align*}
 \widehat{\mathcal{V}}_{j+1/2} := q_{j+1/2} + \gamma_{j+1/2}\vhs(\widehat{u}_{j+1/2}^{\rm R})\,.
\end{align*}
The left and right reconstructed values of $u$ at $z_{j+1/2}$ are computed using the polynomial reconstructions at each side of the interface by $\widehat{u}_{j+1/2}^{\rm L} := \poln^u_j(z_{j+1/2})$ and $\widehat{u}_{j+1/2}^{\rm R} := \poln^u_{j+1}(z_{j+1/2})$, respectively. One of the main properties that $p_j^u$ satisfies, which is going to be widely employed in the next sections, is the cell-average preservation:
\begin{align}
 u_j = \dfrac{1}{h}\int_{I_j} p_j^u (z)\,{\rm d}z,\qquad \text{for }j=1,\dots,N\,. \label{eq:average:property:reconsturctions}
\end{align}
A full description of how to determine the polynomials $\poln_j^u$ is going to be provided in Section~\ref{sec:polynomial:reconstructions}.

To obtain a fully second- and third-order discretization, the diffusive term $\mathcal{E}$ must be approximated with at least the same order of accuracy as the convective flux, which is going to be of order two for the MUSCL scheme and of order three for the WENO scheme in Section~\ref{sec:polynomial:reconstructions}. Furthermore, as there exist two functions related to the compression effects, we propose two approaches to compute $\mathcal{E}_{j+1/2}$, one based on the standard compression function $\mathcal{D}$, called ``$\mathcal{D}$-method'', and a second option using directly function $\mathcal{B}$, called ``$\mathcal{B}$-method''. For the $\mathcal{D}$-method, we first discretize $\partial_z \mathcal{D}$ through the following second and fourth order related formulas
\begin{align}
\mathcal{J}_{j+1/2}^{\rm 2nd} &:= \gamma_{j+1/2}\,\dfrac{\mathcal{D}(u_{j+1}) - \mathcal{D}(u_{j})}{h}\,, \label{eq:J:second}\\
\mathcal{J}_{j+1/2}^{\rm 4th} &:= \gamma_{j+1/2}\,\frac{-\mathcal D(\tilde u_{j+2}) + 27\,\mathcal D(\tilde u_{j+1}) -27\,\mathcal D(\tilde u_{j})+\mathcal D(\tilde u_{j-1})}{24\,h}, \label{eq:J:fourth}
\end{align}
where we use the following fourth-order auxiliary reconstruction of $u$ at the cell centre $z_j$ from the cell averages:
\begin{equation}\label{eq:u:tilde:aux}
	\tilde u_j:=u_j - \dfrac{1}{24}\bigl( u_{j+1} - 2u_j + u_{j-1} \bigr)\,.
\end{equation}
Then, given $\mathcal{J}_{j+1/2} = \mathcal{J}_{j+1/2}^{\rm 2nd}$ for the second-order scheme or $\mathcal{J}_{j+1/2} = \mathcal{J}_{j+1/2}^{\rm 4th}$ for a fourth-order approximation of $\mathcal{J}$, we discretize $\mathcal{E}$ making use of the relation $ (\partial_z\mathcal{B}(u))u = \partial_z\mathcal{D}(u)$, which connects functions $\mathcal{D}$ and $\mathcal{B}$. Therefore, we can define $\mathcal{E}_{j+1/2}$ by the following quotient:
\begin{equation}\label{eq:e-disc}
	\mathcal E_{j+1/2}=\begin{cases}
		\dfrac{\mathcal J_{j+1/2}}{\bar{u}_{j+1/2}} & \text{ if }\bar{u}_{j+1/2}:= \tfrac{1}{2}\bigl(\widehat{u}_{j+1/2}^{\rm L}+\widehat{u}_{j+1/2}^{\rm R}\bigr)\geq  u_{c},\\
		0& \text{ otherwise },
		\end{cases}
\end{equation}
where we have also used the reconstructions $\smash{\widehat{u}_{j+1/2}^{\rm L}}$ and $\smash{\widehat{u}_{j+1/2}^{\rm R}}$.

For the $\mathcal{B}$-method, we directly discretize $\partial_z \mathcal{B}$ using the second- and fourth-order formulas introduced for $\partial_z\mathcal{D}$ in  \eqref{eq:J:second} and \eqref{eq:J:fourth}, respectively, such that we get the two following approximations for $\mathcal{E}$:
\begin{align}
\mathcal{E}_{j+1/2}^{\rm 2nd} &:= \gamma_{j+1/2}\,\dfrac{\mathcal{B}(u_{j+1}) - \mathcal{B}(u_{j})}{h}\,, \label{eq:E:second}\\
\mathcal{E}_{j+1/2}^{\rm 4th} &:= \gamma_{j+1/2}\,\frac{-\mathcal{B}(\tilde u_{j+2}) + 27\,\mathcal{B}(\tilde u_{j+1}) -27\,\mathcal{B}(\tilde u_{j})+\mathcal{B}(\tilde u_{j-1})}{24\,h}, \label{eq:E:fourth}
\end{align}
where $\tilde{u}_{j}$ is defined in \eqref{eq:u:tilde:aux}. Then, depending on the order of the approximation, we choose either the second-order $\mathcal{E}_{j+1/2} = \mathcal{E}_{j+1/2}^{\rm 2nd}$  or the fourth-order approach $\mathcal{E}_{j+1/2} = \mathcal{E}_{j+1/2}^{\rm 4th}$.

For the right-hand side of \eqref{fvintegral}, we use the definition of the Dirac distribution and a quadrature rule for the integral involving the reaction terms. In turn, we approximate the integral of the reaction terms making use of the polynomial reconstructions of $\bB{c}$ and $\bB{s}$ over the cell $I_j = [z_{j-1/2},z_{j+1/2}]$, respectively denoted by $\bB{p}_{j}^{\bB{c}}$ and $\bB{p}_{j}^{\bB{s}}$, and employing a $G$-point Legendre-Gauss-Lobatto quadrature rule with weights $w_1,w_2,\dots,w_G$ for the interval $[-\tfrac{1}{2},\tfrac{1}{2}]$ such that $w_{1} + \cdots + w_{G} =1$,
and with quadrature nodes given by
\begin{equation}\label{eq:gauss-legendre}
	S_j :=  \Big\{z_{j-1/2}= z_j^{1}, z_j^{2},\dots, z_j^{G-1}, z_j^{G} = z_{j+1/2} \Big\},\quad \text{ for all } j=0,1,\dots,N.
\end{equation}
Then, we have
\begin{align*}
 \int_{I_j} b_u\big(\bB{c},\bB{s},z,t\big){\rm d}z\,\approx\, \delta_{j,j_{\rm f}}q_{\rm f}(t)u_{\rm f}(t) + \gamma_j \spacestep\, \sum_{\alpha=1}^{G} w_\alpha \bB{1}^{\tt t}\bB{\sigma}_{\bB{c}}\bB{r}\big(\bB{p}_j^{\bB{c}}(z_{j}^{\alpha}),\bB{p}_j^{\bB{s}}(z_{j}^{\alpha}) \big)\,,
\end{align*}
with $\gamma_j = \gamma(z_j)$.

\subsection{Component concentration equations} \label{sec:component:concentration:equations}
We first observe that Equations \eqref{eq:model:c} and \eqref{eq:model:s} are both linear in their respective vectors of unknowns, $\bB{c}$ and $\bB{s}$, therefore an upwind scheme is the natural choice to approximate the related vector fluxes. Furthermore, as mentioned earlier in Section~\ref{sec:model:equations}, the approximation of \eqref{eq:model:c} should be designed such that the sum over each component equation equals the approximation of Equation~\eqref{eq:model:u}. This condition is required in order to ensure that the sum of the components of the discrete solution match the approximation of the total concentration of solids.
To establish the higher-order accuracy, we proceed as for the total concentration of solids and make use of polynomial reconstructions on each component of the vectors $\bB{c}$ and $\bB{s}$, respectively. Then, given the vector functions $\bpoln^{\bB{c}}_j$ and $\bpoln^{\bB{s}}_j$, each one composed of $n_{\bB{c}}$ and $n_{\bB{s}}$ polynomials of degree $\ell-1$, respectively, we denote the polynomial reconstructions of $\bB{c}$ and $\bB{s}$ respectively by $\widehat{\bB{c}}_{j+1/2}^{\rm L}$ (resp.~$\widehat{\bB{c}}_{j+1/2}^{\rm R}$) and $\widehat{\bB{s}}_{j+1/2}^{\rm L}$ (resp.~ $\widehat{\bB{s}}_{j+1/2}^{\rm R}$). Hence, the approximation of the (convective) fluxes are given by
\begin{align}
& \widehat{\mathbcal{G}}^{\bB{c}}_{j+1/2} :=
{\rm Upw}\Bigl(\widehat{\mathcal{V}}_{j+1/2}-{\mathcal{E}}_{j+1/2};\,\widehat{\bB{c}}_{j+1/2}^{\rm L},\,\widehat{\bB{c}}_{j+1/2}^{\rm R} \Bigr)\,, \label{eq:numfluxP}\\
& \widehat{\mathbcal{G}}^{\bB{s}}_{j+1/2} := {\rm Upw}
 \Biggl( \rho q_{j+1/2} - \widehat{\mathcal{F}}_{j+1/2} ;\,
 \dfrac{\widehat{\bB{s}}_{j+1/2}^{\rm L}}{\rho - \widehat{u}_{j+1/2}^{\rm L}},\,
 \dfrac{\widehat{\bB{s}}_{j+1/2}^{\rm R}}{\rho - \widehat{u}_{j+1/2}^{\rm R}}\Biggr)\,, \label{eq:numfluxS}
\end{align}
where $\widehat{\mathcal{V}}_{j+1/2}$, $\mathcal{E}_{j+1/2}$ and $\widehat{\mathcal{F}}_{j+1/2}$ are defined in Section~\ref{sec:total:concentration:equations} and the polynomial reconstructions at the cell boundaries are given, as in the case of $u$, by
\begin{align*}
 \widehat{\bB{c}}_{j+1/2}^{\rm L}  = \bpoln^{\bB{c}}_j(z_{j+1/2})\,, \quad  \widehat{\bB{c}}_{j+1/2}^{\rm R}  = \bpoln^{\bB{c}}_{j+1}(z_{j+1/2})\,,\quad
 \widehat{\bB{s}}_{j+1/2}^{\rm L}  = \bpoln^{\bB{s}}_j(z_{j+1/2}) \quad \mbox{and}\quad \widehat{\bB{s}}_{j+1/2}^{\rm R}  = \bpoln^{\bB{s}}_{j+1}(z_{j+1/2})\,.
\end{align*}
The right-hand sides of Equations \eqref{eq:model:c} and \eqref{eq:model:s}, which involve feed and reaction terms, are approximated in the same way as for the total concentration equation by
\begin{align*}
\int_{I_j} \bB{b}_{\bB{c}}\big(\bB{c},\bB{s},z,t\big){\rm d}z
& \,\approx\, \delta_{j,j_{\rm f}}q_{\rm f}(t)\bB{c}_{\rm f}(t) + \gamma_j\, \spacestep\sum_{\alpha=1}^G  w_{\alpha} \bB{\sigma}_{\bB{c}}\bB{r}\big(\bB{p}_j^{\bB{c}}(z_j^\alpha),\bB{p}_j^{\bB{s}}(z_j^\alpha)\big),\\
 \int_{I_j} \bB{b}_{\bB{s}}\big(\bB{c},\bB{s},z,t\big){\rm d}z &
 \,\approx\, \delta_{j,j_{\rm f}}q_{\rm f}(t)\bB{s}_{\rm f}(t) + \gamma_j\,\spacestep
  \sum_{\alpha=1}^G  w_{\alpha}\bB{\sigma}_{\bB{s}}\bB{r}\big(\bB{p}_j^{\bB{c}}(z_j^\alpha),\bB{p}_j^{\bB{s}}(z_j^\alpha)\big),
\end{align*}
where the quadrature weights $ w_{\alpha}$ and nodes $z_j^\alpha$ for all $\alpha=1,\dots,G$ have been defined in Section~\ref{sec:total:concentration:equations}.

\subsection{Semi-discrete scheme} \label{sec:semi-discrete:scheme}

Collecting the spatial approximations described in Sections~\ref{sec:total:concentration:equations} and \ref{sec:component:concentration:equations} for the equations in the reactive sedimentation model \eqref{syst:final}, we can write the spatially discretized (semi-discrete) scheme as the following ODE system for each $j=0,1,\dots,N+1$, and  $0<t\leq T$:
\begin{subequations}\label{syst:edos}
\begin{align}
\dfrac{\rd u_j}{\rd t}(t) & = - \dfrac{1}{\spacestep}(\widehat{\mathcal{F}}_{j+1/2}-\widehat{\mathcal{F}}_{j-1/2}) + \dfrac{1}{\spacestep}\delta_{j,j_{\rm f}}q_{\rm f}u_{\rm f}
 + \, \gamma_j\sum_{\alpha=1}^G  w_{\alpha} \bB{1}^{\tt t} \bB{\sigma}_{\bB{c}}\bB{r}\big(\bB{p}_j^{\bB{c}}(z_j^\alpha),\bB{p}_j^{\bB{s}}(z_j^\alpha)\big)
\,, \label{eq:edo:u} \\
\dfrac{\rd \bB{c}_j}{\rd t} (t) & = - \dfrac{1}{\spacestep}\big(\widehat{\mathbcal{G}}^{\bB{c}}_{j+1/2} - \widehat{\mathbcal{G}}^{\bB{c}}_{j-1/2}\big)  + \dfrac{1}{\spacestep}\delta_{j,j_{\rm f}}q_{\rm f}\bB{c}_{\rm f} + \, \gamma_j\sum_{\alpha=1}^G  w_{\alpha} \bB{\sigma}_{\bB{c}}\bB{r}\big(\bB{p}_j^{\bB{c}}(z_j^\alpha),\bB{p}_j^{\bB{s}}(z_j^\alpha)\big)\,, \label{eq:edo:p}\\[1ex]
\dfrac{\rd \bB{s}_j}{\rd t} (t) & = - \dfrac{1}{\spacestep}\big(\widehat{\mathbcal{G}}^{\bB{s}}_{j+1/2} - \widehat{\mathbcal{G}}^{\bB{s}}_{j-1/2}\big)  + \dfrac{1}{\spacestep}\delta_{j,j_{\rm f}}q_{\rm f}\bB{s}_{\rm f} + \,
\, \gamma_j\sum_{\alpha=1}^G  w_{\alpha} \bB{\sigma}_{\bB{s}}\bB{r}\big(\bB{p}_j^{\bB{c}}(z_j^\alpha),\bB{p}_j^{\bB{s}}(z_j^\alpha)\big)
\,, \label{eq:edo:s}
\end{align}
\end{subequations}
where the ghost cells $j=0$ and $j=N+1$ are included to properly compute the fluxes in the numerical scheme. Note that in principle, as $\gamma$ controls the inner and outer part of the sedimentation tank, the calculation at the ghost cells is straightforward from the scheme. This method-of-lines formulation, is going to be particularly helpful for the implementation of a higher-order time approximation method in Section~\ref{sec:time:approximation}. In addition, to simplify the notation in the upcoming sections, we introduce the vectors containing the cell values related to the stencil of the scheme of the discretized unknowns $\mathbf{u}_{{\tt st},j}(t)\in\mathbb{R}^{3}$, $\mathbf{c}_{{\tt st},j}(t)\in \mathbb{R}^{n_{\bB{c}}\times 3}$ and $\mathbf{s}_{{\tt st},j}(t)\in \mathbb{R}^{n_{\bB{s}}\times 3}$ for $t>0$, such that
\begin{align}\label{eq:stencil:vectors}
\mathbf{u}_{{\tt st},j}(t)&:= \begin{pmatrix}
                     u_{j-1}(t),
                     u_j(t),
                     u_{j+1}(t)
                    \end{pmatrix}\,,\\
\mathbf{c}_{{\tt st},j}(t)&:= \begin{pmatrix}
                     \bB{c}_{j-1}(t),
                     \bB{c}_j(t),
                     \bB{c}_{j+1}(t)
                    \end{pmatrix}\,,\\
\mathbf{s}_{{\tt st},j}(t)&:= \begin{pmatrix}
                     \bB{s}_{j-1}(t),
                     \bB{s}_j(t),
                     \bB{s}_{j+1}(t)
                    \end{pmatrix}\,,
\end{align}
for all $j=0,1,\dots,N+1$. Moreover, we define the functions corresponding to the discretized differential operators and source terms of each equation of system~\eqref{syst:edos}  respectively as
\begin{align*}
 \mathcal{H}^u_j\big(\mathbf{u}_{{\tt st},j}^{n},\mathbf{c}_{{\tt st},j}^{n},\mathbf{s}_{{\tt st},j}^{n}\big)&:= - \dfrac{1}{\spacestep}(\widehat{\mathcal{F}}_{j+1/2}-\widehat{\mathcal{F}}_{j-1/2})  + \dfrac{1}{\spacestep}\delta_{j,j_{\rm f}}q_{\rm f}u_{\rm f}
  + \, \gamma_j\sum_{\alpha=1}^G  w_{\alpha} \bB{1}^{\tt t} \bB{\sigma}_{\bB{c}}\bB{r}\big(\bB{p}_j^{\bB{c}}(z_j^\alpha),\bB{p}_j^{\bB{s}}(z_j^\alpha)\big)
\,,\\
 \mathbcal{H}^{\bB{c}}_j\big(\mathbf{u}_{{\tt st},j}^{n},\mathbf{c}_{{\tt st},j}^{n},\mathbf{s}_{{\tt st},j}^{n}\big)  &:= - \dfrac{1}{\spacestep}\big(\widehat{\mathbcal{G}}^{\bB{c}}_{j+1/2} - \widehat{\mathbcal{G}}^{\bB{c}}_{j-1/2}\big)  +  \dfrac{1}{\spacestep}\delta_{j,j_{\rm f}}q_{\rm f}\bB{c}_{\rm f} + \, \gamma_j\sum_{\alpha=1}^G  w_{\alpha} \bB{\sigma}_{\bB{c}}\bB{r}\big(\bB{p}_j^{\bB{c}}(z_j^\alpha),\bB{p}_j^{\bB{s}}(z_j^\alpha)\big)\,,\\
 \mathbcal{H}^{\bB{s}}_j\big(\mathbf{u}_{{\tt st},j}^{n},\mathbf{c}_{{\tt st},j}^{n},\mathbf{s}_{{\tt st},j}^{n}\big)  &:= - \dfrac{1}{\spacestep}\big(\widehat{\mathbcal{G}}^{\bB{s}}_{j+1/2} - \widehat{\mathbcal{G}}^{\bB{s}}_{j-1/2}\big)  + \dfrac{1}{\spacestep}\delta_{j,j_{\rm f}}q_{\rm f}\bB{s}_{\rm f} + \, \gamma_j \sum_{\alpha=1}^G  w_{\alpha} \bB{\sigma}_{\bB{s}}\bB{r}\big(\bB{p}_j^{\bB{c}}(z_j^\alpha),\bB{p}_j^{\bB{s}}(z_j^\alpha)\big)\,.
\end{align*}

\section{Polynomial reconstructions} \label{sec:polynomial:reconstructions}

In this section, we address the polynomial reconstructions appearing at the numerical fluxes of the space discretization in system \eqref{syst:edos}. We describe two cases, a second-order reconstruction based on a MUSCL scheme \cite{kolgan1972application,van1979towards}, and a third-order central WENO scheme (CWENO3) proposed in \cite{Levy2000}.
Moreover, for all $j=1,\dots,N$, the vectors of polynomials $\bpoln_j^{\bB{c}}$ and $\bpoln_j^{\bB{s}}$ associated with \eqref{eq:numfluxP} and \eqref{eq:numfluxS}, respectively, are such that
\begin{align*}
\bpoln_j^{\bB{c}}(z):=\big(\poln_j^{\bB{c},(1)}(z),\poln_j^{\bB{c},(2)}(z),\dots,\poln_j^{\bB{c},(n_{\bB{c}})}(z)\big)^{\tt t}\,,\qquad
\bpoln_j^{\bB{s}}(z):=\big(\poln_j^{\bB{s},(1)}(z),\poln_j^{\bB{s},(2)}(z),\dots,\poln_j^{\bB{s},(n_{\bB{s}})}(z)\big)^{\tt t}\,,
\end{align*}
where $\smash{\poln_j^{\bB{c},(i)}}$ is related to the $\smash{c_j^{(i)}}$ concentration component for all $\smash{i\in\{1,\dots,n_{\bB{c}}\}}$, and analogously, $\smash{\poln_j^{\bB{s},(k)}(z)}$ is related to the $\smash{s_j^{(k)}}$ substrate component for all $k\in\{1,\dots,n_{\bB{s}}\}$. For notational convenience, the time dependence is suppressed in the polynomials above and will likewise be omitted throughout the remainder of this section.

Every polynomial $\poln_j^{\bB{c},(i)}$ and $\poln_j^{\bB{s},(k)}$ is assumed to have the same polynomial degree, and will be computed employing a similar procedure as for the polynomial $\poln_j^{u}(z)$ associated with \eqref{eq:numfluxF}. Therefore, in what follows, we will concentrate the description on the polynomial reconstructions corresponding to $\poln_j^{u}(z)$ and provide the key ingredients for the construction of the other polynomials.

\subsection{Second-order reconstruction}\label{sec:MUSCL}

The second-order accuracy, for $\ell=2$, is obtained by employing a linear reconstruction of $u$ over each interval $I_j$. To do so, we use a standard MUSCL scheme \cite{van1979towards} combined with a minmod slope limiter so that the linear polynomial reconstruction at each $j$-cell is given by
\begin{align}
 \poln_j^u(z) : = u_{j} + \sigma_j^u(z-z_j)\,,\qquad\mbox{for } z\in I_j\,, \label{eq:muscl:pu}
\end{align}
where the slope $\sigma_j^u$ is  computed by
\begin{align*}
 \sigma_j^u :=\dfrac{1}{h} {\rm minmod}\Big(\alpha_{\rm M} (u_{j} - u_{j-1}),\tfrac{1}{2}(u_{j+1}-u_{j-1}),\alpha_{\rm M} (u_{j+1} - u_{j})\Big)\,,
\end{align*}
with parameter $\alpha_{\rm M}\in(1,2)$ and $\rm minmod$ function defined as
\begin{align*}
{\rm minmod}(a,b,c) =
\begin{cases}
   {\rm sign}(a)\min\big\{|a|,|b|,|c|\big\} & \mbox{if }{\rm sign}(a) = {\rm sign}(b) = {\rm sign}(c)\,,\\
   0 & \mbox{otherwise}\,,
  \end{cases}
\end{align*}
for all $a,b,c\in \mathbb{R}$, where ${\rm sign}(a)$ returns the sign of the real number $a$. For the vector of solid and substrate components, we define the MUSCL polynomial reconstructions as follows
\begin{alignat}{2}
 p_j^{\bB{c},(i)}(z) & = c_{j}^{(i)} + \sigma_j^{\bB{c}, i}(z-z_j),&\quad &\text{for } i=1,\dots,\nsp, \label{eq:muscl:pc}\\
 p_j^{\bB{s},(k)}(z) & = s_{j}^{(k)} + \sigma_j^{\bB{s}, k}(z-z_j),&\quad &\text{for } k=1,\dots,\ssp, \label{eq:muscl:ps}
\end{alignat}
for all $z\in I_j$, where
\begin{align*}
\sigma_j^{\bB{c},i} &:=
\dfrac{1}{h}
\begin{cases}
\alpha_{\rm M} \big(c_{j}^{(i)} - c_{j-1}^{(i)}\big) &\text{ if } \sigma^u_j = \alpha_{\rm M}(u_{j} - u_{j-1}),\\[1ex]
\tfrac{1}{2}\big(c_{j+1}^{(i)}-c_{j-1}^{(i)}\big) & \text{ if } \sigma_j^u = \tfrac{1}{2}(u_{j+1}-u_{j-1}),\\[1ex]
\alpha_{\rm M} \big(c_{j+1}^{(i)} - c_{j}^{(i)}\big) & \text{ if } \sigma_j^u = \alpha_{\rm M} (u_{j+1} - u_{j}),\\[1ex]
0 & \text{ otherwise},
\end{cases}\\
\sigma_j^{\bB{s},k} &:=\dfrac{1}{h} {\rm minmod}\Big(\alpha_{\rm M} (s_{j}^{(k)} - s_{j-1}^{(k)}),\tfrac{1}{2}(s_{j+1}^{(k)}-s_{j-1}^{(k)}),\alpha_{\rm M} (s_{j+1}^{(k)} - s_{j}^{(k)})\Big)\,.
\end{align*}
From the definition of $p_j^u$, $\bB{p}_j^{\bB{c}}$, and $\bB{p}_j^{\bB{s}}$, it is straightforward to show that the average-preserving property \eqref{eq:average:property:reconsturctions} is fulfilled for each polynomial, that is
\begin{align*}
 u_j = \dfrac{1}{h}\int_{I_j} p_j^u(z)\,{\rm d}z,\quad
 \bB{c}_j = \dfrac{1}{h}\int_{I_j} \bB{p}_j^{\bB{c}}(z)\,{\rm d}z,\quad
 \bB{s}_j = \dfrac{1}{h}\int_{I_j} \bB{p}_j^{\bB{s}}(z)\,{\rm d}z, \quad \text{for all }j.
\end{align*}
The following lemma states that the MUSCL reconstruction satisfies an invariant-region property.
\begin{lemma}\label{lem:IRP:poly:muscl:u:c:s}
 Consider the polynomials $p_j^u$, $\smash{\bB{p}_j^{\bB{c}}}$ and $\smash{\bB{p}_j^{\bB{s}}}$ defined in \eqref{eq:muscl:pu}, \eqref{eq:muscl:pc} and \eqref{eq:muscl:ps}, respectively.
 Assume that $(u_j,\bB{c}_j,\bB{s}_j)\in \mathbcal{W}$. Then, there holds
\begin{align*}
 \big(p_j^u(z), \bB{p}_j^{\bB{c}}(z), \bB{p}_j^{\bB{s}}(z)\big)\in \mathbcal{W}, \qquad  \text{for all } z\in I_j.
\end{align*}
\end{lemma}
\begin{proof}
We first assume that $\sigma_j^u > 0$, from which we directly have
 \begin{align*}
  \sigma_j^u \leq \dfrac{1}{h}\alpha_M (u_{j} - u_{j-1}),\qquad
  \sigma_j^u \leq \dfrac{1}{2h}(u_{j+1}-u_{j-1}),\qquad
  \sigma_j^u \leq \dfrac{1}{h}\alpha_M (u_{j+1} - u_{j}).
 \end{align*}
Then, we can conclude that $ 0\leq  u_{j-1}<u_j<u_{j+1} \leq u_{\rm max}$. Now, since $p_j^u$ is a linear and increasing polynomial with respect to $z$, we only need to show the bounds at the limits $z_{j\pm 1/2}$, that is
\begin{align*}
 p_j^u(z_{j-1/2}) &= u_j - \dfrac{h}{2}\sigma_j^u \geq (1-\widetilde{\alpha}_M)u_j + \widetilde{\alpha}_M u_{j-1}\geq 0,\\
 p_j^u(z_{j+1/2}) &= u_j + \dfrac{h}{2}\sigma_j^u \leq
 (1-\widetilde{\alpha}_M)u_j + \widetilde{\alpha}_M u_{j+1}\leq u_{\max},
\end{align*}
where $\widetilde{\alpha}_M = \alpha_M/2$. The case $\sigma_j^u <0$ is analogous, and $\sigma_j^u =0$ is trivial. The bounds $p_j^{\bB{c},(i)}\geq0$ and $p_j^{\bB{s},(k)}\geq0$ can be shown in the same way as for $p_j^u$ for all $i=1,\dots,\nsp$ and $k=1,\dots,\ssp$, and thanks to the definition of the slope $\sigma_j^{\bB{c},i}$, we have
\begin{align*}
 \bB{1}^{\tt t}\bB{p}^{\bB{c}}_j(z) = \sum_{i=1}^{\nsp} c_j^{(i)} + \sigma_j^{\bB{c}, i}(z-z_j) = u_j + \sum_{i=1}^{\nsp} \sigma_j^{\bB{c},i} (z-z_j) = u_j + \sigma_j^u (z-z_j) = p_j^u(z),
\end{align*}
which concludes the proof.
\end{proof}

\subsection{Third-order reconstruction} \label{sec:CWENO}
To achieve third-order accuracy, we set $\ell = 3$ and consider the compact central WENO scheme introduced in \cite{Levy2000}, in which the polynomial reconstruction combines linear and quadratic polynomials.
We will refer to this third-order scheme as CWENO3. For ease of presentation, we drop the sub-index $j$ when describing the linear and quadratic polynomials involved in this reconstruction. In turn, we define the left and right linear interpolation polynomials as
\begin{align}\label{def:polLR}
 \poln_{{\rm L}}^u(z) := u_{j-1} + \dfrac{u_j-u_{j-1}}{h}(z-z_{j-1})\,,\qquad
 \poln_{{\rm R}}^u(z) := u_{j} + \dfrac{u_{j+1}-u_{j}}{h}(z-z_{j})\,,
\end{align}
respectively, and the so-called \emph{central polynomial} denoted by $\poln_{\rm C}^u$, is defined as a quadratic polynomial computed with the values of $u_{j+1}$, $u_{j}$ and $u_{j-1}$, see Figure \ref{fig:fig2}. The polynomial reconstruction of $u$ at the cell $I_j$ is then determined by the combination of these three polynomials in the following way
\begin{equation}\label{eq:polREC}
 \poln^{u}_j(z)\,:=\,w_{\rm L}\,\poln_{\rm L}^u(z) + w_{\rm C}\,\poln_{\rm C}^u(z) + w_{\rm R}\,\poln_{\rm R}^u(z)\,,
\end{equation}
where $w_{\rm L}$, $w_{\rm C}$ and $w_{\rm R}$ are positive weights that add up to one. The central polynomial $\poln_{\rm C}^{u}$ depends on $\poln_{\rm L}^{u}$ and $\poln_{\rm R}^{u}$, and the quadratic optimal polynomial:
\begin{align*}
 \poln_{\rm opt}^u(z) &:= u_j - \dfrac{u_{j+1} - 2u_j + u_{j-1}}{24} + \dfrac{u_{j+1}-u_{j-1}}{2h}(z-z_j) + \dfrac{u_{j+1}-2u_j+u_{j-1}}{2h^2}(z-z_j)^2\,,
\end{align*}
such that
\begin{align*}
 \poln_{\rm C}^u(z) :=&\, 2\poln_{\rm opt}^u(z) - \dfrac{1}{2}\Big(\poln_{\rm L}^u(z)+\poln_{\rm R}^u(z)\Big) \\
 =&\, u_j - \dfrac{u_{j+1} - 2u_j+u_{j-1}}{12} + \dfrac{u_{j+1}-u_{j-1}}{2h}(z-z_j) + \dfrac{u_{j+1}-2u_j+u_{j-1}}{h^2}(z-z_j)^2 \,.
\end{align*}
The above formula is given for the specific choice of constants $C_{\rm L} = C_{\rm R} = 1/4$ and $C_{\rm C} = 1/2$, which appear in the following definition of the weights:
\begin{align} \label{eq:weights}
 w_K := \dfrac{\alpha_K}{\alpha_{\rm L}+\alpha_{\rm C}+\alpha_{\rm R}}\,,\qquad\alpha_K = \dfrac{C_K}{(\varepsilon + \mathcal{S}_{K}^u)^2}\qquad\text{for } K\in\{\rm L,C,R\}\,,
\end{align}
where the indicators of smoothness of $\poln_{\rm L}^u$, $\poln_{\rm R}^u$ and $\poln_{\rm C}^u$ are respectively given by
\begin{subequations} \label{eq:definition:SL:SR:SC}
\begin{align}
 &\mathcal{S}_{{\rm L}}^u
 := (u_j-u_{j-1})^2,\\
 &\mathcal{S}_{{\rm R}}^u
 := (u_{j+1}-u_j)^2\,,\\
 &\mathcal{S}_{{\rm C}}^u
 := \dfrac{13}{12}\bigl(u_{j+1}-2u_j+u_{j-1}\bigr)^2+\dfrac{1}{4}\bigl(u_{j+1}-u_{j-1}\bigr)^2.
\end{align}
\end{subequations}
The constant $\varepsilon>0$ in the definition of the weights \eqref{eq:weights} is taken to prevent the denominator from vanishing.
In addition, the values of $\varepsilon$ should be such that $\varepsilon + h^2 u^2_x\gg |\mathcal{S}_K^u - h^2u_x^2|$ in smooth regions, and that $\varepsilon\ll \mathcal{S}_K^u$ near discontinuities for $K\in \{{\rm L},{\rm C},{\rm R}\}$.
\begin{figure}[t!]
\centering
 \includegraphics[scale=0.8]{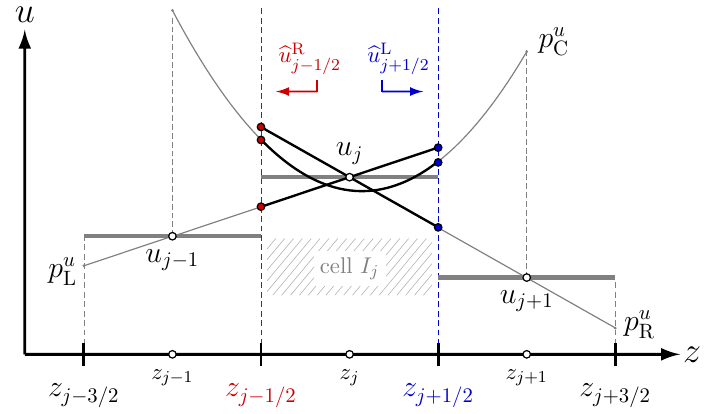}
 \caption{Illustration of the polynomials $\poln_{\rm L}^u$, $\poln_{\rm R}^u$ and $\poln_{\rm C}^u$ related to the reconstruction $\poln_{j}^{u}$ defined in \eqref{eq:polREC} at the cell $I_j = [z_{j-1/2},z_{j+1/2}]$ for the CWENO3 scheme from Section~\ref{sec:CWENO}. The coloured dots are used in the construction of $\widehat{u}_{j-1/2}^{\rm R}$ (red) and $\widehat{u}_{j+1/2}^{\rm L}$ (blue).\label{fig:fig2}}
\end{figure}

To highlight the functional dependencies in the calculation of the polynomial reconstructions, we introduce $\mathcal{P}_j$ as the function that given $u_{j-1}$, $u_j$ and $u_{j+1}$ defines the quadratic polynomial \eqref{eq:polREC} described above, that is
\begin{align}
p_j^u(z) =:\mathcal{P}_j\Big(z,u_{j-1},u_j,u_{j+1};\,\mathcal{S}_{\rm L}^u,\mathcal{S}_{\rm C}^u,\mathcal{S}_{\rm R}^u\Big)\,. \label{eq:puj}
\end{align}
Then, we can directly introduce the polynomial reconstructions for the vectors of solid concentrations and substrates as follows
\begin{alignat}{2}
 p_j^{\bB{c},(i)} &= \mathcal{P}_j\Big(z,c_{j-1}^{(i)},c_j^{(i)},c_{j+1}^{(i)};\,
 \mathcal{S}_{\rm L}^u,\mathcal{S}_{\rm C}^u,\mathcal{S}_{\rm R}^u\Big),\quad &\text{ for }&i=1,\dots,\nsp, \label{eq:pcj}\\
  p_j^{\bB{s},(k)}& = \mathcal{P}_j\Big(z,s_{j-1}^{(k)},s_j^{(k)},s_{j+1}^{(k)};\,
  \mathcal{S}_{\rm L}^{\bB{s},k},\mathcal{S}_{\rm C}^{\bB{s},k},\mathcal{S}_{\rm R}^{\bB{s},k}\Big),
  \quad &\text{ for }&k=1,\dots,\ssp, \label{eq:psj}
\end{alignat}
with $\mathcal{S}_{\rm L}^{\bB{s},k}$, $\mathcal{S}_{\rm C}^{\bB{s},k}$ and $\mathcal{S}_{\rm R}^{\bB{s},k}$ defined as in \eqref{eq:definition:SL:SR:SC} by
\begin{align*}
 &\mathcal{S}_{{\rm L}}^{\bB{s},k}
 := (s^{(k)}_j-s^{(k)}_{j-1})^2,\\
 &\mathcal{S}_{{\rm R}}^{\bB{s},k}
 := (s^{(k)}_{j+1}-s^{(k)}_j)^2\,,\\
 &\mathcal{S}_{{\rm C}}^{\bB{s},k}
 := \dfrac{13}{12}\bigl(s^{(k)}_{j+1}-2s^{(k)}_j+s^{(k)}_{j-1}\bigr)^2+\dfrac{1}{4}\bigl(s^{(k)}_{j+1}-s^{(k)}_{j-1}\bigr)^2,
\end{align*}
where we have also omitted index $j$ as in the previous description for the total concentration. Since by construction the left, right and central interpolation polynomials are linear in terms of the discretized values of the unknowns, and since for $\bB{p}_j^{\bB{c}}$ we have evaluated $\mathcal{S}_{\rm L}$, $\mathcal{S}_{\rm C}$, and $\mathcal{S}_{\rm R}$ on terms depending on the total concentration, we can show that
\begin{align}\label{eq:sum:property:pu:pc}
 p_j^u = \sum_{i=1}^{\nsp} p_j^{\bB{c},(i)} = \bB{1}^{\tt t}\bB{p}_j^{\bB{c}}.
\end{align}

\subsection{Positivity preserving scaling limiters}\label{sec:limiters}
It is well known that the reconstruction procedure described above leads to  schemes that resolve discontinuities sharply but in some cases fail to preserve the positivity of solutions. This shortcoming motivated the development of the linear scaling limiter by Zhang and Shu \cite{zhang2010maximum,zhang2011maximum}.
We begin describing this limiter for the polynomial reconstructions of $u_j$ and $\bB{c}_j$,  respectively. To implement this limiter, we replace the polynomials in \eqref{eq:puj} and \eqref{eq:pcj} by
\begin{align}
	\widetilde{p}_j^{u}(z) &:= \theta_j^{\rm solid} \bigl(p_j^{u}(z)-u_{j} \bigr)+u_j,\label{eq:ptilde:u}\\[1ex]
	\widetilde{p}_j^{\bB{c},(i)}(z) &:= \theta_j^{\rm solid} \bigl(p_j^{\bB{c},(i)}(z)-c^{(i)}_{j} \bigr)+c^{(i)}_j,\quad \text{ for }i=1,\dots,\nsp,\label{eq:ptilde:c}
\end{align}
where $\smash{\theta_j^{\rm solid}}$ is defined such that $\smash{0\leq \widetilde{p}_j^{u}\leq u_{\max}}$, and in addition, due to the connection between vector $\bB{c}$ and $u$, we impose the stronger condition that $\smash{p_j^{\bB{c},(i)}\geq 0}$ for all $i$ in the definition of $\theta_j^{\rm solid}$, which thanks to Equation \eqref{eq:sum:property:pu:pc} also implies $p_j^u\geq 0$. Then, we set
\begin{align} \label{eq:theta:c}
\theta_j^{\rm solid} &= \min\left\{1,\ \frac{u_{\max}-u_j}{M^u_j-u_j},\,	\min_{1\leq i\leq\nsp}\frac{c^{(i)}_j}{c^{(i)}_j-m^{\bB{c},i}_j}\right\},
\end{align}
with
\begin{align}\label{eq:mju}
m^{\bB{c},i}_j:=\min_{z\in I_j} p^{\bB{c},(i)}_j (z),\qquad M^u_j:=\max_{z\in I_j}p^u_j(z).
\end{align}
Then, using the modified polynomial including the limiter, the reconstructed variables in \eqref{eq:numfluxF} are now computed as $\widehat{u}_{j+1/2}^{\rm L} := \widetilde{\poln}^u_j(z_{j+1/2})$ and $\widehat{u}_{j+1/2}^{\rm R} := \widetilde{\poln}^u_{j+1}(z_{j+1/2})$, and analogously for the reconstruction of vector $\bB{c}_j$. Observe that from the definition of $\theta_j^{\rm solid}$ and Equation~\eqref{eq:sum:property:pu:pc}, we can show that
\begin{align} \label{eq:sum:tilde:pc:tilde:pu}
 \bB{1}^{\tt t}\widetilde{\bB{p}}^{\bB{c}}_j(z) = \sum_{i=1}^{\nsp} p^{\bB{c},(i)}_j(z)
 = \sum_{i=1}^{\nsp}\Big(\theta_j^{\rm solid} \big(p_j^{\bB{c},(i)}(z)-c^{(i)}_{j} \big)+c^{(i)}_j\Big) = \theta_j^{\rm solid}\sum_{i=1}^{\nsp} \big(p_j^{\bB{c},(i)}(z)-c^{(i)}_{j} \big)+u_j
 = \widetilde{p}_j^u(z).
\end{align}
For the polynomial reconstruction related to the vector of substrates, unlike the case of the solids related variables, we only require non-negativity of each component, therefore we define
\begin{align}
 \widetilde{p}_j^{\bB{s},(k)}(z)& := \theta_{k,j}^{\rm fluid} \bigl(p_j^{\bB{s},(k)}(z)-s_{j}^{(k)} \bigr)+s_j^{(k)},\quad \text{ for all }k=1,\dots,\ssp,\label{eq:ptilde:s}
\end{align}
where
\begin{align}\label{eq:theta:s}
\theta_{k,j}^{\rm fluid} & =\min\left\{1,\,	\dfrac{s^{(k)}_j}{s^{(k)}_j-m^{\bB{s},k}_j}\right\}, \quad \text{with } m^{\bB{s},k}_j=\min_{z\in I_j} p^{\bB{s},(k)}_j (z).
\end{align}
Next, we recall some useful inequalities for the minimum and maximum values related to the polynomial reconstructions involved in the definition of the slopes in \eqref{eq:theta:c} and \eqref{eq:theta:s}, respectively. The polynomial reconstruction of the total solids concentration satisfies
\begin{align}\label{eq:m:M:u}
\min_{z\in I_j}p^u_j(z) =:m_j^u \leq u_j = \dfrac{1}{\spacestep}\int_{I_j}p_j^u(z)\,{\rm d}z\leq M_j^u\,,
\end{align}
Furthermore, for the minimum values of the reconstructions for the vectors of solids and substrates components, we have
\begin{alignat}{2}
m_j^{\bB{c},i} \leq c_j^{(i)} \leq u_j &\leq M_j^u, &\qquad
m_j^{\bB{c},i}&\leq p_j^{\bB{c},(i)},\qquad \text{for }i=1,\dots,\nsp\,, \label{eq:m:min:c}\\
 m_j^{\bB{s},k} &\leq s_j^{(k)},&\qquad m_j^{\bB{s},k}&\leq p_j^{\bB{s},(k)},\qquad \text{for }k=1,\dots,\ssp\,,\label{eq:m:min:s}
\end{alignat}
where from the second equation in \eqref{eq:m:min:c}, we conclude that
\begin{align}\label{eq:sum:m:pu}
 \sum_{i=1}^{\nsp} m_j^{\bB{c},i}\leq \bB{1}^{\tt t}\bB{p}_j^{\bB{c}} = p_j^{u}.
\end{align}
The above inequalities allow us to state that $0\leq \theta_j^{\rm solid},\theta_{j,k}^{\rm fluid}\leq 1$ for all $k$, which implies that \eqref{eq:ptilde:u}, \eqref{eq:ptilde:c} and \eqref{eq:ptilde:s} are in fact, convex combinations between the averaged values and the reconstructed polynomials. The next lemma shows that the limiters preserve the invariant-region property.

\begin{lemma}\label{lem:IRP:poly:weno:u:c:s}
 Consider the polynomial reconstructions $\widetilde{p}_j^u$, $\smash{\widetilde{\bB{p}}_j^{\bB{c}}}$ and $\smash{\widetilde{\bB{p}}_j^{\bB{s}}}$ for the third-order CWENO3 scheme defined in \eqref{eq:ptilde:u}, \eqref{eq:ptilde:c} and \eqref{eq:ptilde:s}, respectively. Assume that $(u_j,\bB{c}_j,\bB{s}_j)\in \mathbcal{W}$ for all $j$. Then, there holds
\begin{align*}
 \big(\widetilde{p}_j^u(z), \widetilde{\bB{p}}_j^{\bB{c}}(z),\widetilde{\bB{p}}_j^{\bB{s}}(z)\big)\in \mathbcal{W}, \qquad  \text{for all } z\in I_j.
\end{align*}
\end{lemma}
\begin{proof}
Given $z\in I_j$, if $p_j^u(z)-u_j <0$, then $p_j^u(z) < u_j \leq u_{\max}$ and therefore $\widetilde{p}_j^u(z) \leq u_{\max}$, as $\widetilde{p}_j^u$ is a convex combination between $u_j$ and $p_j^u$. In addition, if $p_j^u(z)-u_j \geq 0$, we have
\begin{align*}
 \widetilde{p}_j^u(z)\leq \dfrac{u_{\max}-u_j}{M_j^u - u_j} (p_j^u(z) - u_j) + u_j\leq u_{\max}.
\end{align*}
For the lower bounds, we observe that for all $1\leq i\leq \nsp$, $\theta_j^{\rm solid}\leq c_j^{(i)}/(c_j^{(i)}-m_j^{\bB{c},i})$, which is equivalent to $\theta_j^{\rm solid}(c_j^{(i)}-m_j^{\bB{c},i})\leq c_j^{(i)}$ as $c_j^{(i)}-m_j^{\bB{c},i}\geq 0$ thanks to \eqref{eq:m:min:c}. Then, adding over $i$ and using \eqref{eq:sum:m:pu}, we get
\begin{align*}
 0 &\leq
(1-\theta_j^{\rm solid})u_j+\theta_j^{\rm solid}\sum_{i=1}^{\nsp}m_j^{\bB{c},i}
 \leq
(1-\theta_j^{\rm solid})u_j+\theta_j^{\rm solid}p_j^u = \widetilde{p}_j^u.
\end{align*}
Similarly, for the vector of solids and substrates we have
\begin{align*}
 0 \leq
(1-\theta_j^{\rm solid})c_j^{(i)} + \theta_j^{\rm solid}m_j^{\bB{c},i}
 & \leq
(1-\theta_j^{\rm solid})c_j^{(i)}+\theta_j^{\rm solid}p_j^{\bB{c},(i)} = \widetilde{p}_j^{\bB{c},(i)},\quad \text{for }i=1,\dots,\nsp,\\
 0 \leq
(1-\theta_{k,j}^{\rm fluid})s_j^{(k)} + \theta_{k,j}^{\rm fluid}m_j^{\bB{s},k}
 & \leq
(1-\theta_{k,j}^{\rm fluid})s_j^{(k)} + \theta_{k,j}^{\rm fluid}p_j^{\bB{s},(k)} = \widetilde{p}_j^{\bB{s},(k)}, \quad \text{for }k=1,\dots,\ssp,
\end{align*}
and finally, $\widetilde{p}_j^{\bB{c},(1)} + \widetilde{p}_j^{\bB{c},(2)} + \cdots + \widetilde{p}_j^{\bB{c},(\nsp)} = \widetilde{p}_j^u$ is fulfilled thanks to Equation \eqref{eq:sum:tilde:pc:tilde:pu}, concluding in this way the proof.
\end{proof}

The minimum and maximum values $\smash{m_j^{\bB{c},i}}$, $\smash{m_j^{\bB{s},k}}$ and $\smash{M_j^u}$ in \eqref{eq:theta:c} and \eqref{eq:theta:s} require evaluating the extrema of polynomials on each cell.
This inconvenience can be avoided if we define simplified limiters (as in \cite{zhang2010maximum}) that reduce the evaluations of each polynomial to a finite number of nodes of a $G$-point Legendre-Gauss-Lobatto quadrature rule on the interval $I_j = [z_{j-1/2},z_{j+1/2}]$.
This formula is exact for the integral of polynomials of degree up to $2G-3$, such as the quadrature rule described in Section~\ref{sec:component:concentration:equations}, and the main gain is that the local extrema can be computed using the polynomials at the limits $z_{j\pm1/2}$, average on $I_j$, and weights of the quadrature rule. To explain the derivation of this compact formula, we focus on $p_j^u$, and compute its minimum, denoted by $m_j^u$, and its maximum $M_j^u$. We consider the weights $ w_\alpha$ and nodes $z_j^{\alpha}$ for $\alpha=1,\dots,G$ introduced in Section~\ref{sec:component:concentration:equations} to approximate the reaction terms, and further define $\nu_{\alpha} := p_j^{u}(z_j^\alpha)$ with which we have
\begin{align}\label{eq:media-gl}
{u}_{j} \,=\, \dfrac{1}{h} \int_{I_j} p_j^{u} (z) \, \mathrm{d}z
= \sum_{\alpha=1}^{G} w_{\alpha} \nu_{\alpha}
= w_{1}	\nu_{1} + w_{G}\nu_G + \sum_{\alpha=2}^{G-1} w_{\alpha} \nu_\alpha,
\end{align}
where thanks to the mean value theorem, there exists $\xi\in I_j$ such that the last sum above can be written as
\begin{align} \label{eq:mean:value}
	\sum_{\alpha=2}^{G-1} w_{\alpha} \nu_j^\alpha =
	(1- w_1 - w_G)\sum_{\alpha=2}^{G-1} \dfrac{w_{\alpha}}{1-w_1 - w_G} p_j^{u}(z_j^\alpha) = (1- w_1 - w_G) p_j^{u} \bigl(\xi \bigr),
\end{align}
where we have used the fact the coefficients in the middle sum are such that
\begin{align*}
\sum_{\alpha=2}^{G-1}\dfrac{w_{\alpha}}{1- w_1 - w_G} = 1.
\end{align*}
Then, replacing \eqref{eq:mean:value} into \eqref{eq:media-gl}, we have
\begin{align*}
{u_{j} = w_1  p_j^{u}(z_{j-1/2}) + w_G p_j^{u}(z_{j+1/2})+ (1-w_1 - w_G)p_j^{u}(\xi)},
\end{align*}
from which we can determine the value of $p_j^{u}(\xi)$ as follows:
\begin{align*}
p_j^{u}(\xi)= \frac{1}{1 - w_1 - w_G}\Big(u_{j}-w_1  p_j^{u}(z_{j-1/2})- w_G p_j^{u}(z_{j+1/2})\Big) =: \eta_{j}^u.
\end{align*}
Note that the formula above can be computed even though $\xi$ is a priori unknown. Thus the minimum and maximum values of $p_j^{u}(z)$ over $I_j$ can be computed through the compact formulas:
\begin{align}
	m_{j}^{u}&=\min\big\{ p_j^{u}(z_{j-1/2}),\, \eta_{j}^u,\, p_j^{u}(z_{j+1/2})\big\},\qquad
	M_{j}^{u} =\max\big\{ p_j^{u}(z_{j-1/2}),\, \eta_{j}^u,\, p_j^{u}(z_{j+1/2})\big\}. \label{eq:new-min}
\end{align}
In particular, for a third-order reconstruction we employ a three-point ($G=3$) quadrature rule with weights  $w_1=1/6,  w_2=2/3,$ and $w_3 = 1/6$. In the remainder of the paper, unless otherwise specified, we will only consider the modified polynomial reconstructions including the limiters in \eqref{eq:ptilde:u}, \eqref{eq:ptilde:c} and \eqref{eq:ptilde:s}, for which we will drop the tilde, and simply denote them as $p_j^u$, $\bB{p}_j^{\bB{c}}$ and $\bB{p}_j^{\bB{s}}$ instead.

\section{Time approximation} \label{sec:time:approximation}

As pointed out in \cite{Shu1988}, a suitable time discretization for higher order approximations is a $\ell$-th total variation diminishing (TVD) Runge-Kutta method, also known as strong stability preserving Runge-Kutta (SSPRK) methods \cite{gottlieb1998total,gottlieb2001strong}. Due to convexity in the intermediate stages, in particular for the second and third-order SSPRK methods, these time discretizations preserve the positivity of numerical solutions. To begin describing the time approximation, we consider a time step $\Delta t>0$, discrete time points $t^n:=n\Delta t$ with $n\in \mathbb{N}$ and each time dependent variable and unknown evaluated at $t^n$ will be denoted with the superscript $n$, for example $u_j(t^n) = u_j^n $, etc. We employ second- and third-order SSPRK schemes for the MUSCL and CWENO3 spatial reconstructions, respectively.

\subsection{Second order time approximation}
The second-order explicit SSPRK method is given by the following Butcher tableau:
\[
\begin{array}{c|cc}
	0      &     &  \\[1ex]
	1& 1  &         \\[1ex]
	\hline\\[-2ex]
	& 1/2  & 1/2
\end{array}
\]
Then, we sequentially solve the three equations in \eqref{syst:edos} in two stages.
\begin{itemize}
	\item {\bf Stage 1:} Solve for the first intermediate variables $\mathbf{u}_{{\tt st},j}^{(1)}$, $\mathbf{c}_{{\tt st},j}^{(1)}$ and $\mathbf{s}_{{\tt st},j}^{(1)}$ for all $j=1,\dots,N$:
	\begin{align*}
		u_j^{(1)} \,&=\, u_j^{n} + \Delta t \mathcal{H}^u_j\big(\mathbf{u}_{{\tt st},j}^{n},\mathbf{c}_{{\tt st},j}^{n},\mathbf{s}_{{\tt st},j}^{n}\big)\,, \\[1ex]
		\bB{c}_j^{(1)} \,&=\, \bB{c}_j^{n} + \Delta t \mathbcal{H}^{\bB{c}}_j\big(\mathbf{u}_{{\tt st},j}^{n},\mathbf{c}_{{\tt st},j}^{n},\mathbf{s}_{{\tt st},j}^{n}\big) \,,\\[1ex]
		\bB{s}_j^{(1)} \,&=\, \bB{s}_j^{n} + \Delta t \mathbcal{H}^{\bB{s}}_j\big(\mathbf{u}_{{\tt st},j}^{n},\mathbf{c}_{{\tt st},j}^{n},\mathbf{s}_{{\tt st},j}^{n}\big) \,.
	\end{align*}

	 \item{\bf Stage 2:} Update the solution at time $t^{n+1}$ for all $j=1,\dots,N$:
	\begin{align*}
		u_j^{n+1} \,&=\, \tfrac{1}{2}u_j^{n} + \tfrac{1}{2}u_j^{(1)} + \tfrac{1}{2}\Delta t \mathcal{H}_j^{u}\big(\mathbf{u}_{{\tt st},j}^{(1)},\mathbf{c}_{{\tt st},j}^{(1)},\mathbf{s}_{{\tt st},j}^{(1)}\big)\,, \\[1ex]
		\bB{c}_j^{n+1} \,&=\, \tfrac{1}{2}\bB{c}_j^{n} + \tfrac{1}{2}\bB{c}_j^{(1)} + \tfrac{1}{2}\Delta t \mathbcal{H}_j^{\bB{c}}\big(\mathbf{u}_{{\tt st},j}^{(1)},\mathbf{c}_{{\tt st},j}^{(1)},\mathbf{s}_{{\tt st},j}^{(1)}\big)\,, \\[1ex]
		\bB{s}_j^{n+1} \,&=\, \tfrac{1}{2}\bB{s}_j^{n} + \tfrac{1}{2}\bB{s}_j^{(1)} + \tfrac{1}{2}\Delta t \mathbcal{H}_j^{\bB{s}}\big(\mathbf{u}_{{\tt st},j}^{(1)},\mathbf{c}_{{\tt st},j}^{(1)},\mathbf{s}_{{\tt st},j}^{(1)}\big)\,.
	\end{align*}
\end{itemize}
\subsection{Third order time approximation}
The third-order explicit SSPRK method is given by the following Butcher tableau:
\[
\begin{array}{c|ccc}
0      &\quad        &\quad        &\quad        \\[1ex]
1      & 1      &        &        \\[1ex]
1/2& 1/4 & 1/4 &        \\[1ex]
\hline\\[-2ex]
       & 1/6 & 1/6 & 2/3
\end{array}
\]
Then, we sequentially solve the three equations in \eqref{syst:edos} in three stages.
 \begin{itemize}
  \item {\bf Stage 1:} Solve for the first intermediate variables $\mathbf{u}_{{\tt st},j}^{(1)}$, $\mathbf{c}_{{\tt st},j}^{(1)}$ and $\mathbf{s}_{{\tt st},j}^{(1)}$ for all $j=1,\dots,N$:
\begin{align*}
 u_j^{(1)} \,&=\, u_j^{n} + \Delta t \mathcal{H}^u_j\big(\mathbf{u}_{{\tt st},j}^{n},\mathbf{c}_{{\tt st},j}^{n},\mathbf{s}_{{\tt st},j}^{n}\big)\,,\\[1ex]
  \bB{c}_j^{(1)} \,&=\, \bB{c}_j^{n} + \Delta t \mathbcal{H}^{\bB{c}}_j\big(\mathbf{u}_{{\tt st},j}^{n},\mathbf{c}_{{\tt st},j}^{n},\mathbf{s}_{{\tt st},j}^{n}\big) \,,\\[1ex]
   \bB{s}_j^{(1)} \,&=\, \bB{s}_j^{n} + \Delta t \mathbcal{H}^{\bB{s}}_j\big(\mathbf{u}_{{\tt st},j}^{n},\mathbf{c}_{{\tt st},j}^{n},\mathbf{s}_{{\tt st},j}^{n}\big) \,.
 \end{align*}
\item{\bf Stage 2:} Solve for the second intermediate variables $\mathbf{u}_{{\tt st},j}^{(2)}$, $\mathbf{c}_{{\tt st},j}^{(2)}$ and $\mathbf{s}_{{\tt st},j}^{(2)}$ for all $j=1,\dots,N$:
\begin{align*}
u_j^{(2)} \,&=\, \tfrac{3}{4}u_j^{n} + \tfrac{1}{4}u_j^{(1)} + \tfrac{1}{4}\Delta t \mathcal{H}_j^{u}\big(\mathbf{u}_{{\tt st},j}^{(1)},\mathbf{c}_{{\tt st},j}^{(1)},\mathbf{s}_{{\tt st},j}^{(1)}\big)\,,\\[1ex]
\bB{c}_j^{(2)} \,&=\, \tfrac{3}{4}\bB{c}_j^{n} + \tfrac{1}{4}\bB{c}_j^{(1)} + \tfrac{1}{4}\Delta t \mathbcal{H}_j^{\bB{c}}\big(\mathbf{u}_{{\tt st},j}^{(1)},\mathbf{c}_{{\tt st},j}^{(1)},\mathbf{s}_{{\tt st},j}^{(1)}\big)\,,\\[1ex]
\bB{s}_j^{(2)} \,&=\, \tfrac{3}{4}\bB{s}_j^{n} + \tfrac{1}{4}\bB{s}_j^{(1)} + \tfrac{1}{4}\Delta t \mathbcal{H}_j^{\bB{s}}\big(\mathbf{u}_{{\tt st},j}^{(1)},\mathbf{c}_{{\tt st},j}^{(1)},\mathbf{s}_{{\tt st},j}^{(1)}\big)\,.
\end{align*}
\item{\bf Stage 3:} Update the solution at time $t^{n+1}$ for all $j=1,\dots,N$:
\begin{align*}
 u_j^{n+1} \,&=\, \tfrac{1}{3}u_j^{n} + \tfrac{2}{3}u_j^{(1)} + \tfrac{2}{3}\Delta t \mathcal{H}_j^{u}\big(\mathbf{u}_{{\tt st},j}^{(2)},\mathbf{c}_{{\tt st},j}^{(2)},\mathbf{s}_{{\tt st},j}^{(2)}\big)\,,\\[1ex]
 \bB{c}_j^{n+1} \,&=\, \tfrac{1}{3}\bB{c}_j^{n} + \tfrac{2}{3}\bB{c}_j^{(1)} + \tfrac{2}{3}\Delta t \mathbcal{H}_j^{\bB{c}}\big(\mathbf{u}_{{\tt st},j}^{(2)},\mathbf{c}_{{\tt st},j}^{(2)},\mathbf{s}_{{\tt st},j}^{(2)}\big)\,,\\[1ex]
 \bB{s}_j^{n+1} \,&=\, \tfrac{1}{3}\bB{s}_j^{n} + \tfrac{2}{3}\bB{s}_j^{(1)} + \tfrac{2}{3}\Delta t \mathbcal{H}_j^{\bB{s}}\big(\mathbf{u}_{{\tt st},j}^{(2)},\mathbf{c}_{{\tt st},j}^{(2)},\mathbf{s}_{{\tt st},j}^{(2)}\big)\,.
\end{align*}
 \end{itemize}

\section{Invariant region property} \label{sec:invariant-region}

In this section we are interested in showing that the spatially high-order scheme \eqref{syst:edos} with MUSCL and CWENO3 reconstructions, for the case  without diffusion, namely $\mathcal{D} =\mathcal{B} = 0$, combined with a suitable time discretization as described in Section~\ref{sec:time:approximation}, satisfies an invariant-region-preserving property. For the invariant-region property, one first assumes that at a given time $t^n$, the solution $(u_j^n, \bB{c}_j^n,\bB{s}_j^n)$ belongs to $\mathbcal{W}$, and then the aim is to show that $(u_j^{n+1}, \bB{c}_j^{n+1},\bB{s}_j^{n+1})$ also belongs to $\mathbcal{W}$. The standard approach followed in the literature \cite{barajas2025invariant,barajas2026second,zhang2011maximum} is to establish an invariant-region property for the first order scheme using forward Euler approximation in time, and then using this result, state the invariant-region property for the high-order scheme.

Before studying the first-order scheme, we collect some useful bounds related to the reaction terms.
Given $(u_j^n,\bB{c}_j^n,\bB{s}_j^n)\in \mathbcal{W}$, the following inequality can be obtained as a direct consequence of Assumption~\ref{asumption:reactionterms}:
\begin{align}
\bB{1}^{\tt t}\bB{\sigma}_{\bB{c}}\bB{r}_j^n & = \sum_{k=1}^{\nsp} \sum_{i = 1}^{\nrr}\sigma_{\bB{c}}^{(k,i)} r_j^{(i),n}   \geq - \sum_{k=1}^{\nsp}
 \sum_{i \in \mathcal{J}_{\bB{c},k}^-}^{\nrr}\big(-\sigma_{\bB{c}}^{(k,i)}\big) \bar{r}_{\bB{c},j}^{(i)}c^{(k),n}_j \geq
 - \Lambda_{\bB{c},1} \mathrm{R} u^{n}_j, \label{eq:R:lowerbound:solids}\\
( \bB{\sigma}_{\bB{s}}\bB{r}_j^n)^{(k)} & = \sum_{i = 1}^{\nrr}\sigma_{\bB{s}}^{(k,i)} r_j^{(i),n}   \geq -
 \sum_{i \in \mathcal{J}_{\bB{s},k}^-}^{\nrr}\big(-\sigma_{\bB{s}}^{(k,i)}\big) \bar{r}_{\bB{s},j}^{(i)}s^{(k),n}_j \geq
 - \Lambda_{\bB{s},1} \mathrm{R} s^{(k),n}_j, \label{eq:R:lowerbound:substrates}
\end{align}
where $\Lambda_{\bB{\xi},1}\geq 0$ is defined as
\begin{align*}
 \Lambda_{\bB{\xi},1} :=\max_{k} \sum_{i\in \mathcal{J}_{\bB{\xi},k}^-} (-\sigma_{\bB{\xi}}^{(k,i)})\quad \text{ for } \bB{\xi}\in \{\bB{c},\bB{s}\}.
\end{align*}
For an upper bound of the reaction term for the solid components, we let $L_{\rm R}>0$ be the constant related to the Lipschitz continuity of each function $r^{(i)}$ for $i=1,\dots,\nrr$, see Assumption~\ref{asumption:reactionterms}. Then, using the Lipschitz continuity condition of $r^{(i)}$ and the fact that $r^{(i)}(\bB{c},\bB{s})=\bB{0}$ for $\bB{c}$ vectors such that $\bB{1}^{\tt t}\bB{c}=u_{\max}$, and  defining $\bar{\bB{c}} = \bB{c}_j^n + \tfrac{1}{\nsp}(u_{\max} - u_j^n) \bB{1}$, we get
\begin{align*}
\big|\,r^{(i)}\big(\bar{\bB{c}},\,\bB{s}_j^n\big) -  r_j^{(i),n} \,\big| = r_j^{(i),n} \leq L_{\rm R} \sum_{i=1}^{\nsp} \big|c_j^{(i),n} + \tfrac{1}{\nsp}(u_{\max} - u_j^n) - c_j^{(i),n}\big| = L_{\rm R} (u_{\max} - u_j^n),
\end{align*}
which can be used to obtain the following bound
\begin{align}\label{eq:R:upperbound}
\bB{1}^{\tt t}\bB{\sigma}_{\bB{c}}\bB{r}_j^n &
\leq  \sum_{k=1}^{\nsp}
 \sum_{i \in \mathcal{J}_{\bB{c},k}^+}^{\nrr}\sigma_{\bB{c}}^{(k,i)} r_j^{(i),n}
 \leq L_{\rm R} (u_{\max} - u_j^{n})  \sum_{k=1}^{\nsp}
 \sum_{i \in \mathcal{J}_{\bB{c},k}^+}^{\nrr}\sigma_{\bB{c}}^{(k,i)}
 \leq L_{\rm R} \Lambda_{\bB{c},2}(u_{\max} - u_j^{n}),
\end{align}
where $\Lambda_{\bB{c},2}\geq 0$ is defined as
\begin{align*}
  \Lambda_{\bB{c},2} := \sum_{k=1}^{\nsp}
 \sum_{i \in \mathcal{J}_{\bB{c},k}^+}^{\nrr}\sigma_{\bB{c}}^{(k,i)}\geq 0.
\end{align*}
In what follows, we are going to consider the Courant–Friedrichs–Lewy (CFL) condition
\begin{align}\label{eq:CFL:convective}
	\Delta t \leq \bigg(\dfrac{1}{\spacestep}\left(\dfrac{\rho\|q_{\rm f}\|_{\infty} }{\rho-u_{\max}}  + \max\big\{v_0,L_{\rm v}\big\}\right)  +  \max\big\{\Lambda_{\bB{c},1} \mathrm{R}, \Lambda_{\bB{s},1} \mathrm{R}, \Lambda_{\bB{c},2} L_{\rm R} \big\}\bigg)^{-1},
\end{align}
where we have that the ratio $\lambda:= \Delta t /h$ is bounded. Observe that to properly cover the diffusive case, with $\mathcal{D}\neq 0$, an additional term of the type $\|d\|_{\infty}/h^2$ must be included inside the parentheses on the right-hand side of \eqref{eq:CFL:convective}.

To show the invariant-region property of the first order scheme \cite{SDIMA_MOL}, we take advantage of the structure of the PDE system \eqref{syst:final}, which can be solved sequentially. Therefore, we first establish the bound $0\leq u_j^{n+1}\leq u_{\max}$, then $\bB{c}_j^{n+1}\geq \bB{0}$ and $\bB{1}^{\tt t}\bB{c}_j^{n+1} = u_j^{n+1}$, and lastly $\bB{s}_j^{n+1}$ respectively in the three lemmas that follow.
\begin{lemma} \label{lem:bounds:u}
Assume $(u_j^n,\bB{c}_j^n,\bB{s}_j^n)\in \mathbcal{W}$ for all $j=1,\dots,N$, and consider the following three points first-order scheme for the total solids concentration
\begin{align}\label{eq:first:order:u}
 u_j^{n+1} & = u_j^n -\lambda ( {\mathcal{F}}_{j+1/2}^n- {\mathcal{F}}_{j-1/2}^n)  + \lambda \delta_{j,j_{\rm f}}q_{\rm f}^n u_{\rm f}^n + \, \Delta t \,\bB{1}^{\tt t}\bB{\sigma}_{\bB{c}}\bB{r}_j^n =: \Phi_j^u(u_{j-1}^n, u_j^n,u_{j+1}^n)\,,
\end{align}
where
\begin{align} \label{eq:F:V:convective}
\mathcal{F}_{j+1/2}^n = \mathcal{V}_{j+1/2}^{n,+} {u}_j^n  + \mathcal{V}_{j+1/2}^{n,-} {u}_{j+1}^n,\quad \text{and}\quad \mathcal{V}_{j+1/2}^n = q_{j+1/2}^n + \gamma_{j+1/2} \vhs(u_{j+1}^n).
\end{align}
Then, under CFL condition \eqref{eq:CFL:convective}, there holds $0\leq u_j^{n+1}\leq u_{\max}$ for all $j=1,\dots,N$.
\end{lemma}

\begin{proof}
For the lower bound, we use the definition of $\mathcal{V}_{j\pm 1/2}^n$ in \eqref{eq:F:V:convective}, including the definition of $q_{j\pm 1/2}^n$ in \eqref{eq:def:qjhalf} and the fact that $0\leq \vhs(u) \leq v_0$ for all $0\leq u\leq u_{\max}$, the lower bound \eqref{eq:R:lowerbound:solids} and CFL condition \eqref{eq:CFL:convective} to obtain
\begin{align*}
u_j^{n+1} & = u_j^n  - \lambda\mathcal{V}_{j+1/2}^{n,+} {u}_j^n  +  \lambda\mathcal{V}_{j-1/2}^{n,-} {u}_{j}^n  - \lambda\mathcal{V}_{j+1/2}^{n,-} {u}_{j+1}^n  +  \lambda\mathcal{V}_{j-1/2}^{n,+} {u}_{j-1}^n
+ \lambda \delta_{j,j_{\rm f}}q_{\rm f}^nu_{\rm f}^n + \, \Delta t \,\bB{1}^{\tt t}\bB{\sigma}_{\bB{c}}\bB{r}_j^n\\[1ex]
& \geq u_j^n  - \lambda \big(\max\{q_{\rm u}(t)\} + v_0 + \max\{q_{\rm e}(t)\}\big) {u}_{j}^n - \, \Delta t \Lambda_{\bB{c}} \mathrm{R} u^{n}_j + \lambda \delta_{j,j_{\rm f}}q_{\rm f}^nu_{\rm f}^n \\
&\geq  \Big(1 - \lambda \big(\|q_{\rm f}\|_{\infty} + v_0\big)  - \, \Delta t \Lambda_{\bB{c}} \mathrm{R}\Big) u^{n}_j\geq 0.
\end{align*}
For the upper bound, we proceed as follows. First, we subtract $u_j^{n+1}$ to $u_{\max}$ such that we bound the following equation instead
\begin{align*}
 u_{\max} - u_j^{n+1} & = u_{\max} -  u_j^n + \lambda \big(  \mathcal{V}_{j+1/2}^{n,+} {u}_j^n - \mathcal{V}_{j-1/2}^{n,-} {u}_{j}^n  + \mathcal{V}_{j+1/2}^{n,-} {u}_{j+1}^n  -  \mathcal{V}_{j-1/2}^{n,+} {u}_{j-1}^n\big) \\ &\quad - \lambda \delta_{j,j_{\rm f}}q_{\rm f}^n u_{\rm f}^n - \, \Delta t \,\bB{1}^{\tt t}\bB{\sigma}_{\bB{c}}\bB{r}_j^n.
\end{align*}
Then, the convective terms can be lower-bounded by
\begin{align*}
&  \mathcal{V}_{j+1/2}^{n,+} {u}_j^n - \mathcal{V}_{j-1/2}^{n,-} {u}_{j}^n  + \mathcal{V}_{j+1/2}^{n,-} {u}_{j+1}^n  -  \mathcal{V}_{j-1/2}^{n,+} {u}_{j-1}^n - \delta_{j,j_{\rm f}}q_{\rm f}^n u_{\rm f}^n\\
& \qquad  \geq (q_{j+1/2}^{n,+}- q_{j-1/2}^{n,-}) u_j^n  -  (q_{j-1/2}^{n,+} - q_{j+1/2}^{n,-} + \delta_{j,j_{\rm f}} q_{\rm f}^n)u_{\max} - \vhs(u_j^n) u_{\max}
\end{align*}
where we have used that $u_{\rm f}^n \leq u_{\max}$, and the following estimates
\begin{align*}
 & \mathcal{V}_{j+1/2}^{n,+} {u}_j^n   \geq q_{j+1/2}^{n,+} u_j^n,\quad
- \mathcal{V}_{j-1/2}^{n,-} {u}_{j}^n  \geq - q_{j-1/2}^{n,-} {u}_{j}^n,\quad
 \mathcal{V}_{j+1/2}^{n,-} {u}_{j+1}^n
  \geq q_{j+1/2}^{n,-}u_{\max},\\
-&\mathcal{V}_{j-1/2}^{n,+} {u}_{j-1}^n
 \geq -q_{j-1/2}^{n,+} u_{\max} - \big(\gamma_{j-1/2}\vhs(u_j^n)\big)^{+}u_{\max}.
\end{align*}
The term related to $\vhs$ can be bounded using its Lipschitz continuity at $u= u_{\max}$ in~\eqref{eq:Lipschitz:vhs}, i.e., $\vhs(u_j^n)\leq L_{\rm v} (u_{\max} - u_j^n)$, together with $u_j^n - u_{\max} \leq 0$. Furthermore, using the following identity
\begin{align*}
\|q_{\rm f}\|_{\infty}\geq  q_{j+1/2}^{n,+}-q_{j-1/2}^{n,-} = q_{j-1/2}^{n,+}-q_{j+1/2}^{n,-} + \delta_{j,j_{\rm f}} q_{\rm f}^n
 =
 \begin{cases}
 q_{\rm e}^n &\text{ if } j < j_{\rm f} - 1,\\
 q_{\rm f}^n &\text{ if } j = j_{\rm f},\\
 q_{\rm u}^n &\text{ if } j > j_{\rm f} + 1,
 \end{cases}
\end{align*}
we have
\begin{align*}
&  \mathcal{V}_{j+1/2}^{n,+} {u}_j^n - \mathcal{V}_{j-1/2}^{n,-} {u}_{j}^n  + \mathcal{V}_{j+1/2}^{n,-} {u}_{j+1}^n  -  \mathcal{V}_{j-1/2}^{n,+} {u}_{j-1}^n - \delta_{j,j_{\rm f}}q_{\rm f}^n u_{\rm f}^n\\
& \qquad  \geq \|q_{\rm f}\|_{\infty}(u_j^n  - u_{\max}) -L_{\rm v} (u_{\max} - u_j^n).
\end{align*}
Finally, combining the above inequality with \eqref{eq:R:upperbound}, and CFL condition \eqref{eq:CFL:convective}, we get
\begin{align*}
 u_{\max} - u_j^{n+1} \geq \Big(1 - \lambda\big(\|q_{\rm f}\|_{\infty} + L_{\rm v}\big) -  \Delta t L_{\rm R} \Lambda_{\bB{c},2}\Big)(u_{\max} - u_j^{n})\geq 0,
\end{align*}
concluding with this the proof.
\end{proof}

\begin{remark}
Following the proof of Lemma~\ref{lem:bounds:u}, we can straightforwardly show that under the CFL condition \eqref{eq:CFL:convective}, there holds
\begin{align}\label{eq:bound:u:and:R}
 0\leq u_j^n  + \, \Delta t \,\bB{1}^{\tt t}\bB{\sigma}_{\bB{c}}\bB{r}_j^n\leq u_{\max}.
\end{align}
\end{remark}

\begin{lemma}\label{lem:bounds:c}
Assume $(u_j^n,\bB{c}_j^n,\bB{s}_j^n)\in \mathbcal{W}$ for all $j=1,\dots,N$, and consider the following three points first-order scheme for the vector of solid concentrations
\begin{align}\label{eq:first:order:c}
\bB{c}_j^{n+1} & =\bB{c}_j^{n} - \lambda\big(\mathbcal{G}^{\bB{c},n}_{j+1/2} - \mathbcal{G}^{\bB{c},n}_{j-1/2}\big)  + \lambda\delta_{j,j_{\rm f}}q_{\rm f}^n\bB{c}_{\rm f}^n + \Delta t\, \bB{\sigma}_{\bB{c}}\bB{r}_j^n\,=:\bB{\Phi}_j^{\bB{c}}(\bB{c}_{j-1}^n,\bB{c}_j^n,\bB{c}_{j+1}^n),
\end{align}
where
\begin{align*}
 \mathbcal{G}^{\bB{c},n}_{j+1/2} :={\rm Upw}\big(\mathcal{V}_{j+1/2}^n;\,\bB{c}_{j}^n,\,\bB{c}_{j+1}^n\big)\,.
\end{align*}
Then, under CFL condition \eqref{eq:CFL:convective}, there holds $\bB{c}^{n+1}_j\geq \bB{0}$ and
\begin{align}\label{eq:sum:c:equals:u:discrete}
 \bB{1}^{\tt t}\bB{c}_j^{n+1} = c^{(1),n+1}_j +c^{(2),n+1}_j+\cdots + c^{(\nsp),n+1}_j = u_j^{n+1},\quad\text{ for all } j =1,\dots,N.
\end{align}

\end{lemma}

\begin{proof}
For the lower bound, we follow the same bounds as in the proof of Lemma~\ref{lem:bounds:u}, such that for $i=1,\dots,\nsp$, we have
\begin{align*}
c_j^{(i),n+1} & = c_j^{(i),n}  - \lambda\mathcal{V}_{j+1/2}^{n,+} {c}_j^{(i),n}  +  \lambda\mathcal{V}_{j-1/2}^{n,-} {c}_{j}^{(i),n}  - \lambda\mathcal{V}_{j+1/2}^{n,-} {c}_{j+1}^{(i),n}  +  \lambda\mathcal{V}_{j-1/2}^{n,+} {c}_{j-1}^{(i),n}\\
&\quad + \lambda \delta_{j,j_{\rm f}}q_{\rm f}^n c_{\rm f}^{(i),n} + \, \Delta t \,(\bB{\sigma}_{\bB{c}}\bB{r}_j^n)^{(i)}\\[1ex]
&\geq  \Big(1 - \lambda \big(\|q_{\rm f}\|_{\infty} + v_0\big)  - \, \Delta t \Lambda_{\bB{c}} \mathrm{R}\Big) c^{(i),n}_j\geq 0,
\end{align*}
where the last inequality is a consequence of the CFL condition \eqref{eq:CFL:convective}.
Note that the bound for the reaction term can be relaxed, but we keep $\Lambda_{\bB{c}}$ as this is already needed in the CFL condition due to Lemma~\ref{lem:bounds:u}. Finally, Equation \eqref{eq:sum:c:equals:u:discrete} is trivially fulfilled as each term on $\bB{\Phi}_j^{\bB{c}}$ adds up the corresponding term in \eqref{eq:first:order:u}.
\end{proof}

\begin{lemma} \label{lem:bounds:s}
Assume $(u_j^n,\bB{c}_j^n,\bB{s}_j^n)\in \mathbcal{W}$ for all $j=1,\dots,N$, and consider the following three points first-order scheme for the vector of substrates
\begin{align}\label{eq:first:order:s}
\bB{s}_j^{n+1} & =\bB{s}_j^{n} - \lambda\big(\mathbcal{G}^{\bB{s},n}_{j+1/2} - \mathbcal{G}^{\bB{s},n}_{j-1/2}\big)  + \lambda \delta_{j,j_{\rm f}}q_{\rm f}^n\bB{s}_{\rm f}^n + \Delta t\, \bB{\sigma}_{\bB{s}}\bB{r}_j^n\,=:\bB{\Phi}_j^{\bB{s}}(\bB{s}_{j-1}^n,\bB{s}_j^n,\bB{s}_{j+1}^n),
\end{align}
where
\begin{align*}
 \mathbcal{G}^{\bB{s},n}_{j+1/2} = {\rm Upw}
 \Biggl( \rho q_{j+1/2}^n - \mathcal{F}_{j+1/2}^n;\,
 \dfrac{\bB{s}_{j}^n}{\rho - u_{j}^n},\,
 \dfrac{\bB{s}_{j+1}^n}{\rho - u_{j+1}^n}\Biggr)\,.
\end{align*}
Then, under CFL condition \eqref{eq:CFL:convective}, there holds $\bB{s}_j^{n+1}\geq \bB{0}$ for all $j=1,\dots,N$.
\end{lemma}
\begin{proof}
Taking into account the positive and negative parts from the definition of $\mathbcal{G}^{\bB{s},n}_{j\pm 1/2}$, for $i=1,\dots,\ssp$, we have
\begin{align*}
s_j^{(i),n+1} & \geq  s_j^{(i),n}
- \lambda\dfrac{{s}_j^{(i),n}}{\rho - u_j^n}\Big( \big( \rho q_{j+1/2}^n - \mathcal{F}_{j+1/2}^n\big)^{+}  -  \lambda\big( \rho q_{j-1/2}^n - \mathcal{F}_{j-1/2}^n\big)^{-} \Big) + \, \Delta t \,(\bB{\sigma}_{\bB{s}}\bB{r}_j^n)^{(i)},
\end{align*}
where each convective term $\big( \rho q_{j\pm 1/2}^n - \mathcal{F}_{j\pm1/2}^n\big)^{\pm}$ can be lower-bounded using $u_j^n\leq u_{\max}\leq \rho$ and $0\leq \vhs(u_j^n)\leq v_0$ as follows
\begin{align*}
 \big( \rho q_{j+1/2}^n - \mathcal{F}_{j+1/2}^n\big)^{+} & = \Big(\rho q_{j+1/2}^n - \mathcal{V}_{j+1/2}^{n,+} {u}_j^n  - \mathcal{V}_{j+1/2}^{n,-} {u}_{j+1}^n\Big)^{+} \leq \Big(\rho q_{j+1/2}^n - \mathcal{V}_{j+1/2}^{n,-} {u}_{j+1}^n\Big)^{+} \\
  & \leq \big(\rho q_{j+1/2}^n - q_{j+1/2}^{n,-} {u}_{j+1}^n)\big)^{+} \leq \big(q_{j+1/2}^{n,-}(\rho- {u}_{j+1}^n) + \rho q_{j+1/2}^{n,+}\big)^{+}\\
  &\leq \rho q_{j+1/2}^{n,+},
\end{align*}
and
\begin{align*}
 \big( \rho q_{j-1/2}^n - \mathcal{F}_{j-1/2}^n\big)^{-} & = \Big(\rho q_{j-1/2}^n - \mathcal{V}_{j-1/2}^{n,+} {u}_{j-1}^n  - \mathcal{V}_{j-1/2}^{n,-} {u}_{j}^n\Big)^{-} \geq \Big(\rho q_{j-1/2}^n - \mathcal{V}_{j-1/2}^{n,+} {u}_{j-1}^n\Big)^{+} \\
  & \geq \big(\rho q_{j-1/2}^n - q_{j-1/2}^{n,+} {u}_{j-1}^n - \vhs(u_j^n) {u}_{j-1}^n)\big)^{-} \\
  & \geq \big(q_{j-1/2}^{n,+}(\rho- {u}_{j-1}^n) + \rho q_{j-1/2}^{n,-}- \vhs(u_j^n) {u}_{j-1}^n\big)^{-}\\
  &\geq \rho q_{j-1/2}^{n,-}- \vhs(u_j^n) {u}_{j-1}^n \geq \rho( q_{j-1/2}^{n,-}- v_0 ).
\end{align*}
Finally, using $ \rho/(\rho-u_{\max})\geq  \rho/(\rho-u_j^n)$ and the lower estimate of the reaction term \eqref{eq:R:lowerbound:substrates}, we have
\begin{align*}
 s_j^{(i),n+1} & \geq  s_j^{(i),n}
- \lambda\dfrac{\rho{s}_j^{(i),n}}{\rho - u_j^n}\Big(q_{j+1/2}^{n,+} - q_{j-1/2}^{n,-} + v_0\Big)  - \, \Delta t \Lambda_{\bB{s}} \mathrm{R} s^{(i),n}_j\\
 & \geq  s_j^{(i),n}
- \lambda {s}_j^{(i),n}\left(\dfrac{\rho\|q_{\rm f}\|_{\infty} }{\rho-u_{\max}}  + v_0\right)  - \, \Delta t \Lambda_{\bB{s}} \mathrm{R} s^{(i),n}_j\\
&= \left(1- \lambda\left(\dfrac{\rho\|q_{\rm f}\|_{\infty} }{\rho-u_{\max}} + v_0\right)  - \, \Delta t \Lambda_{\bB{s}} \mathrm{R}\right) s^{(i),n}_j,
\end{align*}
which together with the CFL condition \eqref{eq:CFL:convective}, implies the desired result $s_j^{(i),n+1}\geq 0$.
\end{proof}

Combining Lemmas~\ref{lem:bounds:u}, \ref{lem:bounds:c} and \ref{lem:bounds:s}, we can conclude that the first-order scheme, given by equations \eqref{eq:first:order:u}, \eqref{eq:first:order:c} and \eqref{eq:first:order:s}, is invariant-region-preserving over the set $\mathbcal{W}$ for every time iteration $t^n$ provided the CFL condition \eqref{eq:CFL:convective} holds. The next theorem states the invariant-region property for the MUSCL scheme (Section~\ref{sec:MUSCL}) and CWENO3 including limiters (Section~\ref{sec:CWENO}), respectively, combined with forward Euler time approximations.

\begin{theorem} \label{thm:invariant-region:WENO}
Assume $(u_j^n,\bB{c}_j^n,\bB{s}_j^n)\in \mathbcal{W}$ for all $j=1,\dots,N$, and that the CFL condition \eqref{eq:CFL:convective} holds with $\lambda/w_1$ instead of $\lambda$. Consider the full discretization of \eqref{syst:edos}, with MUSCL or CWENO3 reconstructions including limiters, and forward Euler in time, given by
\begin{subequations}
\begin{align}
 u_j^{n+1} & = u_j^{n} + \Delta t  \mathcal{H}^{u}_j(\mathbf{u}_{{\tt st},j}^{n},\mathbf{c}_{{\tt st},j}^{n},\mathbf{s}_{{\tt st},j}^{n}),  \label{eq:update:weno:u}\\
 \bB{c}_j^{n+1} & = \bB{c}_j^{n} + \Delta t \mathcal{H}^{\bB{c}}_j(\mathbf{u}_{{\tt st},j}^{n},\mathbf{c}_{{\tt st},j}^{n},\mathbf{s}_{{\tt st},j}^{n}),\label{eq:update:weno:c}\\
 \bB{s}_j^{n+1} & = \bB{s}_j^{n} +  \Delta t \mathcal{H}^{\bB{s}}_j(\mathbf{u}_{{\tt st},j}^{n},\mathbf{c}_{{\tt st},j}^{n},\mathbf{s}_{{\tt st},j}^{n}), \label{eq:update:weno:s}
 \end{align}
 \end{subequations}
 for $j=1,\dots,N$, where $\mathbf{u}_{{\tt st},j}^{n}$, $\mathbf{c}_{{\tt st},j}^{n}$ and $\mathbf{s}_{{\tt st},j}^{n}$ are the vectors defined in \eqref{eq:stencil:vectors}. Then, there holds $(u_j^{n+1},\bB{c}_j^{n+1},\bB{s}_j^{n+1})\in \mathbcal{W}$.
\end{theorem}

\begin{proof}
We begin by recalling that since $(u_j^n,\bB{c}_j^n,\bB{s}_j^n)\in \mathbcal{W}$, thanks to Lemma~\ref{lem:IRP:poly:muscl:u:c:s} in the case of MUSCL scheme or Lemma~\ref{lem:IRP:poly:weno:u:c:s} in the case of CWENO3 including limiters, the polynomial reconstructions $(p_j^u,\bB{p}_j^{\bB{c}},\bB{p}^{\bB{s}}_j)$ belong to $\mathbcal{W}$ for all $j=1,\dots,N$. Furthermore, the following identities hold
\begin{align}\label{eq:property:sums:p:alpha}
u_j^n = \sum_{\alpha =1}^{G} w_\alpha p_j^{u,n} (z_j^\alpha),\quad \bB{c}_j^n = \sum_{\alpha =1}^{G} w_\alpha \bB{p}_j^{\bB{c},n}(z_j^\alpha), \quad  \bB{s}_j^n = \sum_{\alpha = 1}^{G} w_\alpha \bB{p}_j^{\bB{s},n}(z_j^\alpha),
\end{align}
with $z_j^\alpha$ and $ w_j^\alpha$, the nodes and quadrature weights defined in Section~\ref{sec:total:concentration:equations}. Without loss of generality, we assume $w_1 = w_G$, and introduce the following short-hand notation for the reaction terms
\begin{align*}
 \widehat{\bB{r}}_{j-1/2}^{n,{\rm R}} =  \bB{\sigma}_{\bB{c}}\bB{r}\big(\widehat{\bB{c}}_{j-1/2}^{n,R},\widehat{\bB{s}}_{j-1/2}^{n,R}\big),\quad
 \widehat{\bB{r}}_{j+1/2}^{n,{\rm L}} = \bB{\sigma}_{\bB{c}}\bB{r}\big(\widehat{\bB{c}}_{j+1/2}^{n,L},\widehat{\bB{s}}_{j+1/2}^{n,L}\big),\quad
 \bB{r}_{j}^{n,\alpha} = \bB{r}\big(\bB{p}^{\bB{c},n}_j(z_j^\alpha),\bB{p}^{\bB{s},n}_j(z_j^\alpha)\big),
\end{align*}
for all $\alpha=1,\dots,G$. Then, we first show the bounds related to $u_j^{n+1}$, for which we use \eqref{eq:property:sums:p:alpha},  add and subtract $\lambda\mathcal{F}\big(\widehat{u}_{j-1/2}^{n,R},\widehat{u}_{j+1/2}^{n,L}\big) $ and reorganise the resulting terms in a convex combination as follows:
\begin{align*}
u_j^{n+1}
& =\sum_{\alpha=2}^{G-1} w_\alpha p_j^{u,n} (z_j^\alpha) + w_1\widehat{u}_{j-1/2}^{n,\rm R} + w_1 \widehat{u}_{j+1/2}^{n,\rm L} - \lambda\Big(\mathcal{F}(\widehat{u}_{j+1/2}^{n,L},\widehat{u}_{j+1/2}^{n,R}) - \mathcal{F}(\widehat{u}_{j-1/2}^{n,L},\widehat{u}_{j-1/2}^{n,R})\Big)\\
& \quad
+ \lambda \delta_{j,j_{\rm f}}q_{\rm f}^nu_{\rm f}^n
+ \Delta tw_1 \,\bB{1}^{\tt t}\bB{\sigma}_{\bB{c}}\widehat{\bB{r}}_{j-1/2}^{n,R}
+ \Delta tw_1 \,\bB{1}^{\tt t}\bB{\sigma}_{\bB{c}}\widehat{\bB{r}}_{j+1/2}^{n,L}
+ \Delta t\sum_{\alpha=2}^{G-1}\, w_{\alpha} \,\bB{1}^{\tt t}\bB{\sigma}_{\bB{c}}\bB{r}_{j}^{n,\alpha}\\
& = w_1 \Big(\widehat{u}_{j+1/2}^{n,\rm L}
 - \dfrac{\lambda}{w_1}\big(
 \mathcal{F}(\widehat{u}_{j+1/2}^{n,L},\widehat{u}_{j+1/2}^{n,R}) -
 \mathcal{F}(\widehat{u}_{j-1/2}^{n,R},\widehat{u}_{j+1/2}^{n,L})\big)
 + \dfrac{\lambda}{w_1} \delta_{j,j_{\rm f}}q_{\rm f}^n(\tfrac{1}{2}u_{\rm f}^n)
 + \Delta t\,\bB{1}^{\tt t}\bB{\sigma}_{\bB{c}}\widehat{\bB{r}}_{j+1/2}^{n,L}
\\
&\quad +
\widehat{u}_{j-1/2}^{n,\rm R} - \dfrac{\lambda}{w_1}\big(
\mathcal{F}(\widehat{u}_{j-1/2}^{n,R},\widehat{u}_{j+1/2}^{n,L}) -
\mathcal{F}(\widehat{u}_{j-1/2}^{n,L},\widehat{u}_{j-1/2}^{n,R})\big)
+\dfrac{\lambda}{w_1} \delta_{j,j_{\rm f}}q_{\rm f}^n(\tfrac{1}{2}u_{\rm f}^n)
+ \Delta t\,\bB{1}^{\tt t}\bB{\sigma}_{\bB{c}}\widehat{\bB{r}}_{j-1/2}^{n,R} \Big)\\
 &\quad
+ \sum_{\alpha=2}^{G-1} w_\alpha \big( p_j^{u,n} (z_j^\alpha) + \Delta t \,\bB{1}^{\tt t}\bB{\sigma}_{\bB{c}}\bB{r}_{j}^{n,\alpha}\big)\\
&=
w_1 \Phi_j^u\big(\widehat{u}_{j-1/2}^{n,R},\widehat{u}_{j+1/2}^{n,L},\widehat{u}_{j+1/2}^{n,R}\big) +
\sum_{\alpha=2}^{G-1} w_\alpha \big( p_j^{u,n} (z_j^\alpha) + \Delta t \,\bB{1}^{\tt t}\bB{\sigma}_{\bB{c}}\bB{r}_{j}^{n,\alpha}\big)\\
&\quad + w_G \Phi_j^u\big(\widehat{u}_{j-1/2}^{n,L},\widehat{u}_{j-1/2}^{n,R},\widehat{u}_{j+1/2}^{n,L}\big),
\end{align*}
which from Lemma~\ref{lem:bounds:u}, and Equation \eqref{eq:bound:u:and:R} as $0\leq p_j^u \leq u_{\max}$, with CFL condition \eqref{eq:CFL:convective} using $\lambda/w_1$ instead of $\lambda$, we conclude that this is in fact a convex combination of quantities bounded between $0$ and $u_{\max}$, therefore $0\leq u^{n+1}_j\leq u_{\max}$. Similarly, for the vector of solid concentrations and substrates, we have
\begin{align*}
 \bB{c}_j^{n+1} &=
w_1 \bB{\Phi}_j^{\bB{c}}\big(\widehat{\bB{c}}_{j-1/2}^{n,R},\widehat{\bB{c}}_{j+1/2}^{n,L},\widehat{\bB{c}}_{j+1/2}^{n,R}\big) +
\sum_{\alpha=2}^{G-1} w_\alpha \big( \bB{p}_j^{\bB{c},n} (z_j^\alpha) + \Delta t \,\bB{\sigma}_{\bB{c}}\bB{r}_{j}^{n,\alpha}\big)
+ w_G \bB{\Phi}_j^{\bB{c}}\big(\widehat{\bB{c}}_{j-1/2}^{n,L},\widehat{\bB{c}}_{j-1/2}^{n,R},\widehat{\bB{c}}_{j+1/2}^{n,L}\big),\\
\bB{s}_j^{n+1} &=
w_1 \bB{\Phi}_j^{\bB{s}}\big(\widehat{\bB{s}}_{j-1/2}^{n,R},\widehat{\bB{s}}_{j+1/2}^{n,L},\widehat{\bB{s}}_{j+1/2}^{n,R}\big) +
\sum_{\alpha=2}^{G-1} w_\alpha \big( \bB{p}_j^{\bB{s},n} (z_j^\alpha) + \Delta t \,\bB{\sigma}_{\bB{s}}\bB{r}_{j}^{n,\alpha}\big)
+ w_G \bB{\Phi}_j^{\bB{s}}\big(\widehat{\bB{s}}_{j-1/2}^{n,L},\widehat{\bB{s}}_{j-1/2}^{n,R},\widehat{\bB{s}}_{j+1/2}^{n,L}\big),
\end{align*}
which thanks to Lemmas~\ref{lem:bounds:c} and~\ref{lem:bounds:s}, we can conclude that $\bB{c}_j^{n+1}\geq \bB{0}$, $\bB{s}_j^{n+1}\geq \bB{0}$ and $\bB{1}^{\tt t}\bB{c}_{j}^{n+1} = u_{j}^{n+1}$, hence $(u_j^{n+1},\bB{c}_j^{n+1},\bB{s}_j^{n+1})\in\mathbcal{W}$.

\end{proof}

In the next theorem, we show that discretizations of \eqref{syst:final} with the MUSCL and CWENO3 reconstructions including limiters, combined with SSPRK approximations of order two and three respectively, satisfy the invariant-region property on $\mathbcal{W}$.

\begin{theorem}\label{thm:invariant-region:time}
Let $\Delta t$ satisfy the CFL condition of Theorem 6.4 and assume
that the polynomial reconstructions (MUSCL or CWENO3 including limiters) are recomputed, at
every Runge--Kutta stage, from the cell averages of that stage. Then the
second- and third-order SSPRK methods of Section~5 are
invariant-region preserving, that is, given $(u^n_j,\boldsymbol{c}^n_j,\boldsymbol{s}^n_j\bigr)\in\mathbcal{W}$ then $\bigl(u^{n+1}_j,\boldsymbol{c}^{n+1}_j,\boldsymbol{s}^{n+1}_j\bigr)\in \mathbcal{W}$ for all $j$ and $n\in\mathbb{N}$.
\end{theorem}

\begin{proof}
Following the ideas from \cite{gottlieb1998total,gottlieb2001strong}, we show that the updated solution of the multi step scheme can be written as a convex combination of terms belonging to $\mathbcal{W}$. To do so, we define the vector of solutions $\vec{\mathbf{u}}_j=(u_j,\boldsymbol{c}_j,\boldsymbol{s}_j)$ and $\vec{\mathbf{u}}_{\text{st},j} = (u_{\text{st},j},\boldsymbol{c}_{\text{st},j},\boldsymbol{s}_{\text{st},j})$, and write the SSPRK methods of Section~\ref{sec:time:approximation} as
\begin{equation}\label{eq:shu-osher}
\vec{\mathbf{u}}^{(0)}_j = \vec{\mathbf{u}}^{n}_j,\qquad
\vec{\mathbf{u}}^{(i)}_j = \sum_{k=0}^{i-1}
\Bigl(\alpha_{ik}\,\vec{\mathbf{u}}^{(k)}_j
 + \Delta t\,\beta_{ik}\,\vec{\mathbcal{H}}_j
\bigl(\vec{\mathbf{u}}^{(k)}_{\text{st},j}\bigr)\Bigr),\quad \text{ for }i\leq m,
\qquad \vec{\mathbf{u}}^{n+1}_j = \vec{\mathbf{u}}^{(m)}_j,
\end{equation}
for all $j$, where $\vec{\mathbcal{H}}_j := (\mathcal{H}_j^u,\mathbcal{H}_j^{\bB{c}},\mathbcal{H}_j^{\bB{s}})$, and $\alpha_{ik},\beta_{ik}\ge0$ are coefficients satisfying
\begin{align*}
\sum_{k=0}^{i-1}\alpha_{ik}=1\quad \text{and}\quad \beta_{ik}\leq \alpha_{ik}, \qquad  \text{for all } i=1,\dots,m.
\end{align*}
Now, setting $\tau_{ik}:=(\beta_{ik}/\alpha_{ik})\Delta t$ whenever $\alpha_{ik}>0$, we observe that the $i$-th stage of
\eqref{eq:shu-osher} can be rewritten as the convex combination
\begin{equation}\label{eq:convex-stage}
\vec{\mathbf{u}}^{(i)}_j
=\sum_{k=0}^{i-1}\alpha_{ik}\,\mathbcal{E}_{ikj}^{(k)},
\quad \text{with}\quad  \mathbcal{E}_{ikj}^{(k)}
:=\vec{\mathbf{u}}^{(k)}_j + \tau_{ik}\,\vec{\mathbcal{H}}_j
\bigl(\vec{\mathbf{u}}_{\mathtt{st},j}^{(k)}\bigr),\qquad \mathbcal{E}^{(0)}_{ikj}:=\vec{\mathbf{u}}^{n}_j,
\end{equation}
where $\mathbcal{E}_{ikj}^{(k)}$ denotes the forward Euler operator with time step
$\tau_{ik}>0$ associated with the semi-discrete scheme \eqref{syst:edos}. Since $\tau_{ik}\leq \Delta t$, then $\tau_{ik}$ fulfils the CFL condition \eqref{eq:CFL:convective} and therefore thanks to Theorem~\ref{thm:invariant-region:WENO}, if $\vec{\mathbf{u}}^{(k)}_j\in \mathbcal{W}$, we have $\mathbcal{E}_{\tau_{ik}}^{(k)}\in\mathbcal{W}$ for $k=0,\dots,i-1$ and $i\leq m$. The rest of the proof follows straightforwardly by induction over index $i$, starting with $\vec{\mathbf{u}}^{(0)}\in\mathbcal{W}$, and using the fact that for each $i$, $\vec{\mathbf{u}}^{(i)}$ is given by a convex combination of $\mathbcal{E}_{ikj}^{(k)}$ from $k=1,\dots, i-1$.
\end{proof}
\begin{remark}
Note that condition $\beta_{ik}\leq \alpha_{ik}$ in the proof above holds for the two schemes presented in Section~\ref{sec:time:approximation}, but does not hold automatically for all SSPRK methods.
\end{remark}

\begin{remark}
The semi-discrete scheme \eqref{eq:edo:u}--\eqref{eq:edo:s} is advanced in time only for the interior cells indexed by $j\in\{1,\dots,N\}$. The $n_{\mathrm g}$ layers of ghost cells on either side of the computational domain, indexed by $j\in\{1-n_{\mathrm g},\dots,0\}$ and $j\in\{N+1,\dots,N+n_{\mathrm g}\}$, are not evolved in time. Instead, their values are reset after each stage of the time integrator described in Section~\ref{sec:time:approximation} according to
	\begin{equation}\label{eq:ghost}
		u_{1-j}:=u_1,\qquad u_{N+j}:=u_N,\qquad\text{for}\quad j=1,\dots,n_{\mathrm g},
	\end{equation}
	with analogous component-wise assignments for $\boldsymbol c$ and $\boldsymbol s$. For the reconstructions described in Section~\ref{sec:polynomial:reconstructions} and the two-point approximations $\mathcal{J}_{j+1/2}^{\rm 2nd}$ \eqref{eq:J:second} and $\mathcal{E}_{j+1/2}^{\rm 2nd}$ \eqref{eq:E:second} of the diffusive-related terms, we set $n_{\mathrm g}=1$, whereas the fourth-order approximations \eqref{eq:J:fourth} and \eqref{eq:E:fourth} require $n_{\mathrm g}=2$.

 In addition, we observe that at the top and bottom of the tank one has $\gamma_{1/2}=\gamma_{N+1/2}=0$, so that $\mathcal J_{1/2}=\mathcal J_{N+1/2}=0$, and consequently $\smash{\mathcal{E}_{1/2}=\mathcal E_{N+1/2}=0}$. Hence $\smash{\widehat{\mathcal V}_{1/2}=q_{1/2}=-q_{\mathrm e}\le0}$ by \eqref{eq:e-disc}, and from the definition of upwind operator we have $\smash{\widehat{\mathcal F}_{1/2}=-q_{\mathrm e}\,\widehat{u}^{\mathrm R}_{1/2}}$ and $\smash{\widehat{\boldsymbol{\mathcal G}}^{\bB{c}}_{1/2}=-q_{\mathrm e}\,\widehat{\boldsymbol c}^{\mathrm R}_{1/2}}$. Furthermore, since $\widehat{u}^{\mathrm R}_{1/2}\le u_{\max}<\rho$, then $\rho q_{1/2}-\widehat{\mathcal F}_{1/2}
=-q_{\mathrm e}\bigl(\rho-\widehat{u}^{\mathrm R}_{1/2}\bigr)\le 0$, so
\begin{equation*}
\widehat{\boldsymbol{\mathcal G}}^{\bB{s}}_{1/2}
	=-q_{\mathrm e}\bigl(\rho-\widehat{u}^{\mathrm R}_{1/2}\bigr)\,
	\frac{\widehat{\boldsymbol s}^{\mathrm R}_{1/2}}{\rho-\widehat{u}^{\mathrm R}_{1/2}}
	=-q_{\mathrm e}\,\widehat{\boldsymbol s}^{\mathrm R}_{1/2}.
\end{equation*}
In the same way we obtain that
$\widehat{\mathcal V}_{N+1/2}=q_{\mathrm u}\ge0$ and
\begin{equation*}
	\widehat{\mathcal F}_{N+1/2}=q_{\mathrm u}\,\widehat{u}^{\mathrm L}_{N+1/2},\qquad
	\widehat{\boldsymbol{\mathcal G}}^{\bB{c}}_{N+1/2}=q_{\mathrm u}\,\widehat{\boldsymbol c}^{\mathrm L}_{N+1/2},
	\qquad
	\widehat{\boldsymbol{\mathcal G}}^{\bB{s}}_{N+1/2}=q_{\mathrm u}\,\widehat{\boldsymbol s}^{\mathrm L}_{N+1/2}.
\end{equation*}
Therefore, we observe that  traces $\widehat{u}^{\mathrm L}_{1/2}$, $\widehat{\boldsymbol c}^{\mathrm L}_{1/2}$, $\widehat{\boldsymbol s}^{\mathrm L}_{1/2}$ and their counterparts at $z_{N+1/2}$ are multiplied by $(-q_{\mathrm e})^{+}=0$ and by $q_{\mathrm u}^{-}=0$, respectively. Thus, no polynomial reconstruction is therefore needed in the ghost cells. Besides, the ghost values enter the scheme solely through the reconstruction stencils of the two
extreme cells $I_1$ and $I_N$. Since \eqref{eq:ghost} is a copy of an interior state which belongs to $\mathbcal{W}$, then hypothesis of Lemmas \ref{lem:bounds:u} to \ref{lem:bounds:s} and of Theorem \ref{thm:invariant-region:WENO} hold at the ghost indices with no additional requirement. Moreover, since \eqref{eq:ghost} is enforced after each stage of \eqref{eq:shu-osher}, every forward Euler operator $ \mathbcal{E}^{(k)}_{ikj}$ appearing in the proof of Theorem \ref{thm:invariant-region:time} acts on a state whose ghost values already belong to $\mathbcal{W}$. The induction over the stage index is therefore unaffected.
\end{remark}

\begin{table}[t]
\caption{Example 1: $L^1$ spatial errors \eqref{eq:errors} and order of convergence \eqref{eq:order} computed with the First-order, MUSCL and IRP-CWENO3 schemes at $t=0.1\, {\rm h}$. The reference solution is computed using the IRP-CWENO3 scheme with $N_{\rm ref} = 10240$ cells.\label{tab:errors:nondiff}}
\setlength{\cmidrulewidth}{0.7pt}
\centering
{   \addtolength{\tabcolsep}{-2pt}{\small
\begin{tabular}{r|cc|cccc|cccccc}
\multicolumn{13}{c}{First-order}\\[0.5ex]
\toprule
$N$ &
$e^{u}_N$ & $\mathcal{R}^{u}_N$ & $e^{\bB{c},1}_N$ & $\mathcal{R}^{\bB{c},1}_N$& $e^{\bB{c},2}_N$ & $\mathcal{R}^{\bB{c},2}_N$ & $e^{\bB{s},1}_N$ & $\mathcal{R}^{\bB{s},1}_N$ & $e^{\bB{s},2}_N$ & $\mathcal{R}^{\bB{s},2}_N$& $e^{\bB{s},3}_N$ & $\mathcal{R}^{\bB{s},3}_N$\\[0.5ex]
\cmidrule(lr){1-1}\cmidrule(lr){2-3}\cmidrule(lr){4-5} \cmidrule(lr){6-7}\cmidrule(lr){8-9} \cmidrule(lr){10-11} \cmidrule(lr){12-13}
160  & 9.86e-2 & --   & 7.08e-2 & --   & 2.78e-2 & --   & 1.76e-4 & --   & 1.14e-3 & --   & 1.22e-4 & --   \\
320  & 5.14e-2 & 0.94 & 3.69e-2 & 0.94 & 1.45e-2 & 0.94 & 9.08e-5 & 0.95 & 5.90e-4 & 0.96 & 6.30e-5 & 0.96 \\
640  & 2.63e-2 & 0.97 & 1.89e-2 & 0.97 & 7.41e-3 & 0.97 & 4.62e-5 & 0.97 & 3.00e-4 & 0.98 & 3.20e-5 & 0.98 \\
1280 & 1.33e-2 & 0.98 & 9.55e-3 & 0.98 & 3.75e-3 & 0.98 & 2.33e-5 & 0.99 & 1.51e-4 & 0.99 & 1.61e-5 & 0.99 \\
2560 & 6.69e-3 & 0.99 & 4.80e-3 & 0.99 & 1.89e-3 & 0.99 & 1.17e-5 & 0.99 & 7.59e-5 & 0.99 & 8.09e-6 & 0.99 \\
\midrule
\multicolumn{11}{c}{} \\[-2.0ex]
\multicolumn{13}{c}{MUSCL}\\[0.5ex]
\toprule
$N$ &
$e^{u}_N$& $\mathcal{R}^{u}_N$ & $e^{\bB{c},1}_N$ & $\mathcal{R}^{\bB{c},1}_N$& $e^{\bB{c},2}_N$ & $\mathcal{R}^{\bB{c},2}_N$ &
$e^{\bB{s},1}_N$& $\mathcal{R}^{\bB{s},1}_N$ & $e^{\bB{s},2}_N$ & $\mathcal{R}^{\bB{s},2}_N$& $e^{\bB{s},3}_N$ & $\mathcal{R}^{\bB{s},3}_N$\\[0.5ex]
\cmidrule(lr){1-1}\cmidrule(lr){2-3}\cmidrule(lr){4-5} \cmidrule(lr){6-7}\cmidrule(lr){8-9} \cmidrule(lr){10-11} \cmidrule(lr){12-13}
160  & 3.09e-3 & --   & 2.22e-3 & --   & 8.72e-4 & --   & 3.47e-6 & --   & 2.10e-5 & --   & 2.40e-6 & --   \\
320  & 7.06e-4 & 2.13 & 5.07e-4 & 2.13 & 1.99e-4 & 2.13 & 7.69e-7 & 2.17 & 4.50e-6 & 2.22 & 4.95e-7 & 2.28 \\
640  & 1.66e-4 & 2.09 & 1.19e-4 & 2.09 & 4.69e-5 & 2.09 & 1.79e-7 & 2.11 & 1.02e-6 & 2.14 & 1.12e-7 & 2.15 \\
1280 & 3.87e-5 & 2.10 & 2.78e-5 & 2.10 & 1.09e-5 & 2.10 & 4.31e-8 & 2.05 & 2.44e-7 & 2.07 & 2.64e-8 & 2.08 \\
2560 & 9.11e-6 & 2.09 & 6.54e-6 & 2.09 & 2.57e-6 & 2.09 & 1.05e-8 & 2.04 & 5.95e-8 & 2.03 & 6.43e-9 & 2.04 \\
\midrule
\multicolumn{11}{c}{} \\[-2.0ex]
\multicolumn{13}{c}{IRP-CWENO3}\\[0.5ex]
\toprule
$N$ &
$e^{u}_N$& $\mathcal{R}^{u}_N$ & $e^{\bB{c},1}_N$ & $\mathcal{R}^{\bB{c},1}_N$& $e^{\bB{c},2}_N$ & $\mathcal{R}^{\bB{c},2}_N$ &
$e^{\bB{s},1}_N$& $\mathcal{R}^{\bB{s},1}_N$ & $e^{\bB{s},2}_N$ & $\mathcal{R}^{\bB{s},2}_N$& $e^{\bB{s},3}_N$ & $\mathcal{R}^{\bB{s},3}_N$\\[0.5ex]
\cmidrule(lr){1-1}\cmidrule(lr){2-3}\cmidrule(lr){4-5} \cmidrule(lr){6-7}\cmidrule(lr){8-9} \cmidrule(lr){10-11} \cmidrule(lr){12-13}
160  & 1.99e-3 & --   & 1.43e-3 & --   & 5.62e-4 & --   & 2.52e-6  & --   & 1.63e-5 & --   & 1.80e-6  & --   \\
320  & 2.49e-4 & 3.00 & 1.79e-4 & 3.00 & 7.02e-5 & 3.00 & 3.03e-7  & 3.06 & 1.96e-6 & 3.06 & 2.16e-7  & 3.05 \\
640  & 3.03e-5 & 3.04 & 2.17e-5 & 3.04 & 8.54e-6 & 3.04 & 3.75e-8  & 3.01 & 2.42e-7 & 3.02 & 2.67e-8  & 3.02 \\
1280 & 3.72e-6 & 3.02 & 2.67e-6 & 3.02 & 1.05e-6 & 3.02 & 4.65e-9  & 3.01 & 3.00e-8 & 3.01 & 3.32e-9  & 3.01 \\
2560 & 4.49e-7 & 3.05 & 3.22e-7 & 3.05 & 1.27e-7 & 3.05 & 5.62e-10 & 3.05 & 3.63e-9 & 3.05 & 4.01e-10 & 3.05 \\
\bottomrule
\end{tabular}}}
\end{table}

\section{Numerical examples}\label{sec:numerical:examples}
We consider a reduced biokinetic model, the denitrification process of converting nitrate ($\rm NO_3$) into nitrogen gas ($\rm N_2$) presented in \cite{SDcace_reactive}. In this model the vector of solids concentration is given by $\boldsymbol{c} = (c^{(1)}, c^{(2)})$, where $c^{(1)}$ and $c^{(2)}$ correspond to the concentration of ordinary heterotrophic organisms $X_{\rm OHO}$, and non-degradable organics $X_{\rm U}$, respectively. The vector of substrates is $\boldsymbol{s} = (s^{(1)},s^{(2)},s^{(3)})$ with $s^{(1)}$, $s^{(2)}$ and $s^{(3)}$ corresponding to the concentration of nitrate $S_{\rm NO_3}$, readily biodegradable substrate $S_{\rm S}$, and nitrogen $S_{\rm N_2}$, respectively. The vector of nonlinear reactions and stoichiometric matrices are given by
\begin{align*}
 \bB{r}(\bB{c},\bB{s}) = \Psi(u)\, c^{(1)}\begin{pmatrix}
           \mu(s^{(1)},s^{(2)})\\
           b
          \end{pmatrix}\,,\quad
  \boldsymbol{\sigma}_{\boldsymbol{c}} &=
\begin{pmatrix}
 1 & -1\\
 0 & f_{\rm p}
\end{pmatrix},\qquad
\bB{\sigma}_{\bB{s}} = \begin{pmatrix}
-\bar{Y} & -1/Y & \bar{Y}\\
 0 & 0.8 & 0
\end{pmatrix}\,,
\end{align*}
where $b=6.94\times10^{-6}{\rm s^{-1}}$ is the decay rate of heterotrophic organisms, $f_{\rm p} = 0.2$ is the portion of these that decays to non-degradable organics, and the parameters $Y=0.67$ and $\bar{Y} = 0.172216$ are (dimensionless) yield factors.
Function $\Psi := \Psi(u)$ is included such that the fourth condition in Assumption~\ref{asumption:reactionterms} is fulfilled, which is assumed to be $\Psi(u) = 1$ for $u < u_{\max}-\widetilde{\varepsilon}$ for a small $0<\widetilde{\varepsilon}<\varepsilon$, and $\Psi(u)=0$ for $u\geq u_{\max}-\varepsilon$. As in all examples shown in this section, the total concentration of solids remains considerably smaller than $u_{\max}$, we simply set $\Psi \equiv 1$.
The specific growth rate function is defined by
\begin{align*}
 \mu(s^{(1)},s^{(2)}) = \mu_{\max} \dfrac{s^{(1)}}{\kappa_1+s^{(1)}}\dfrac{s^{(2)}}{\kappa_2+s^{(2)}}\,,
\end{align*}
with $\mu_{\max} = 5.56\times 10^{-5}\,{\rm s}^{-1}$ being the maximum growth rate, and $\kappa_1=5\times 10^{-4}\,{\rm kg/m^3}$ and $\kappa_2=0.02\,{\rm kg/m^3}$ are saturation constants. Then, the reaction terms read as follows
\begin{align*}
 \boldsymbol{\sigma}_{\boldsymbol{c}}\bB{r}
 &=
 c^{(1)}\begin{bmatrix}
 \mu(s^{(1)},s^{(2)}) - b\\
 f_{\rm p}b
\end{bmatrix}\,,\qquad
 \boldsymbol{\sigma}_{\boldsymbol{s}}\bB{r}
 = c^{(1)}\begin{bmatrix}
  -\bar{Y}\mu(s^{(1)},s^{(2)})\\
  -\mu(s^{(1)},s^{(2)})/Y + 0.8b\\
  \bar{Y}\mu(s^{(1)},s^{(2)})\,
\end{bmatrix}\,,
\end{align*}
The maximal total solids concentration is set to $u_{\max} = 30\, \rm kg/m^3$. The constitutive functions used in all simulations are
\begin{align*}
 v_{\rm hs}(u) = \dfrac{v_0}{1+\left( u/\check{u}\right)^{\eta}} - \dfrac{v_0}{1+\left( u_{\max}/\check{u}\right)^{\eta}}\,, \quad
\sigma_{\rm e}(u) :=\beta\chi_{\{u\geq u_{\rm c}\}}(u-u_{\rm c})\,,
\end{align*}
with $\chi_{{\{u\geq u_{\rm c}\}}}$ defined as the indicator function over the set ${\{u\geq u_{\rm c}\}}$, and constants $v_0 = 1.76 \times 10^{-3}\,\rm m/s$, $\check{u} = 3.87\, \rm kg/m^3$, $u_{\rm c}=5\, \rm kg/m^3$, $\eta = 3.58$ and $\beta = 0.2\, \rm m^2/s^2$. Other parameters are $\rho = 1050\,\rm kg/m^3$, $\rho_{\rm L} = 998\,\rm kg/m^3$, $g = 9.81\,\rm m/s^2$, cross-sectional area $A = 400\,\rm m^2$, $H =1 \,\rm m$ and $B = 3\,\rm m$.  The parameters employed in reconstructions are $\alpha_M=1.4$ and $\varepsilon=h^2$. Furthermore, in what follows, given $J\in \mathbb{N}$, we are going to use the following short-hand notation  $\mathbb{Z}_J := \{1,2,\dots, J\}$.

The numerical scheme and all simulations presented in this work have been implemented in Matlab, and the current solver can be found in the following Github repository:
\[
\text{\href{https://github.com/ReacSedim/mainsolver}{https://github.com/ReacSedim/mainsolver}}.
\]
In what follows, the first-order scheme combined with forward Euler time discretization used in Lemma~\ref{lem:bounds:u}, \ref{lem:bounds:c} and \ref{lem:bounds:s}, which has also been developed in \cite{SDIMA_MOL} is going to be denoted as ``First-order''.  The second-order scheme combining MUSCL reconstruction (Section~\ref{sec:MUSCL}) and SSPRK2 in time will be simply referred to as ``MUSCL''. The third-order WENO scheme developed in Section~\ref{sec:CWENO} including the limiters from Section~\ref{sec:limiters} and SSPRK3 in time will be termed ``IRP-CWENO3'', while the same scheme without limiters will be called  ``CWENO3''.

\begin{figure}[t]
	\centering
	\includegraphics[scale=0.35]{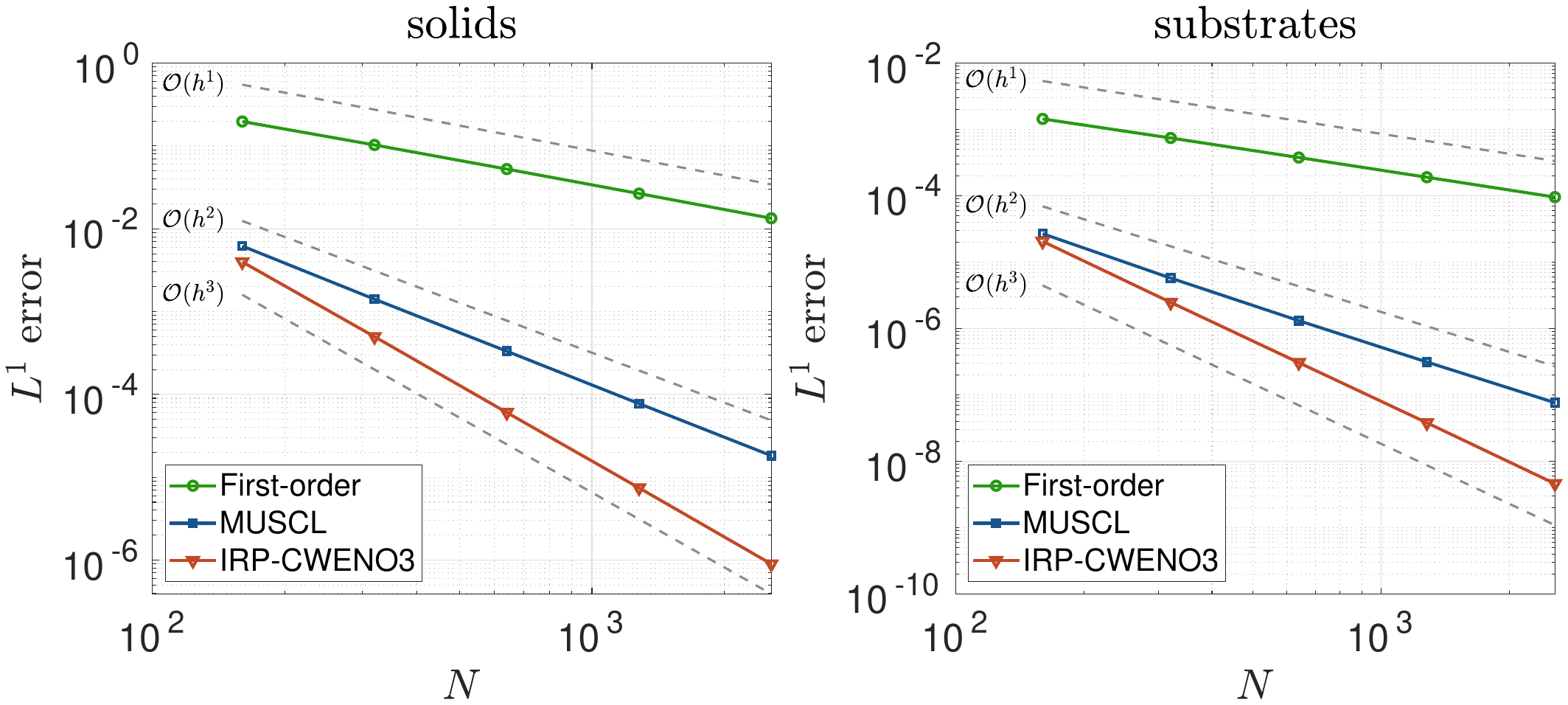}
	\caption{Example 1: combined $L^1$ spatial error of the solids $e_N^{u}+e_N^{\bB{c},1} + e_N^{\bB{c},2}$ (left) and substrates $e_N^{\bB{s},1} +  e_N^{\bB{s},2} + e_N^{\bB{s},3}$ (right) computed with the First-order, MUSCL and IRP-CWENO3 schemes varying with respect to the number of cells $N$ at $t=0.1\,\rm h$. The dashed lines are the corresponding exact order of convergence curve in each case.
	The reference solution is computed using the IRP-CWENO3 scheme with $N_{\rm ref} = 10240$ cells.\label{fig:errorplots}}
\end{figure}

\subsection{Example 1: Accuracy tests}

Since no closed-form solution of the proposed model is available in the literature, we compute the numerical errors using a reference solution obtained from the here developed IRP-CWENO3 scheme on a fine mesh of $N_{\mathrm{ref}}$ cells. To this end, we let $(u^{\mathrm{ref}}_l)_{l\in\mathbb{Z}_{N_{\mathrm{ref}}}}$ be the reference solution at the time $T$ under consideration, associated with the cells
$I^{\mathrm{ref}}_l = [z^{\mathrm{ref}}_{l-1/2},z^{\mathrm{ref}}_{l+1/2}]$, for all $l\in \mathbb{Z}_{N_{\rm ref}}$. Then, given a discretization with $N$ number of cells, being $N$ a divisor of $N_{\rm ref}$, and a specific cell $I_j$, we define $\mathcal{K}_j$ as the set of $M := N_{\mathrm{ref}}/N\in\mathbb{N}$ fine cells contained in the coarse $I_j$:
\begin{align*}
 \mathcal{K}_j:=\bigl\{\,l\in\mathbb{Z}_{N_{\mathrm{ref}}}:
	I^{\mathrm{ref}}_l\subseteq I_j\,\bigr\}.
\end{align*}
Then, the cell average of the reference solution over $I_j$ is given by
\begin{equation}\label{eq:restriction}
	\tilde u^{\,\mathrm{ref}}_j(T):=\dfrac{1}{M}\sum_{l\in\mathcal{K}_j}u^{\mathrm{ref}}_l(T),
	\qquad\text{for all}\quad j\in\mathbb{Z}_N,
\end{equation}
and the average of the solid and substrate components $\smash{\tilde{c}^{\,\mathrm{ref},(i)}_j(T)}$ and $\smash{\tilde{s}^{\,\mathrm{ref},(k)}_j(T)}$ are defined analogously for each component $i=1,\dots,n_{\bB{c}}$ and $k = 1,\dots,n_{\bB{s}}$. Furthermore, for a fixed number of cells $N$, we denote by $u^N_j(T)$, $c^{N,(i)}_j(T)$ and $s^{N,(k)}_j(T)$ the numerical solution at time $T$ computed on the mesh of $N$ cells. The approximate $L^1$ errors of each
individual unknown are then defined by
\begin{equation}\label{eq:errors}
	\begin{aligned}
		e^{u}_N(T)&:=h\sum_{j\in\mathbb{Z}_N}\bigl|\tilde u^{\,\mathrm{ref}}_j(T)-u^N_j(T)\bigr|, \\
		e^{\bB{c},i}_N(T)&:=h\sum_{j\in\mathbb{Z}_N}
		\bigl|\tilde c^{\,\mathrm{ref},(i)}_j(T)-c^{N,(i)}_j(T)\bigr|, &\quad i=1,\dots,\nsp, \\
		e^{\bB{s},k}_N(T)&:=h\sum_{j\in\mathbb{Z}_N}
		\bigl|\tilde s^{\,\mathrm{ref},(k)}_j(T)-s^{N,(k)}_j(T)\bigr|, &\quad k=1,\dots,\ssp,
	\end{aligned}
\end{equation}
where $h=(H+B)/N$. We also measure the convergence rates by
\begin{equation}\label{eq:order}
\begin{aligned}
\mathcal{R}^{u}_N&:=\log_2\left(\frac{e^{u}_N}{e^{u}_{2N}}\right),\quad
\mathcal{R}^{\bB{c},i}_N:=\log_2\left(\frac{e^{\bB{c},i}_N}{e^{\bB{c},i}_{2N}}\right),\quad
\mathcal{R}^{\bB{s},k}_N:=\log_2\left(\frac{e^{\bB{s},k}_N}{e^{\bB{s},k}_{2N}}\right),
\end{aligned}
\end{equation}
for all $i=1,\dots,\nsp$ and $k=1,\dots,\ssp$.

For this test, we consider a sequence of spatial discretizations with $N_{l}=10\cdot 2^{l}$, for $l=4,\ldots,8$, and compute a reference solution using $N_{\mathrm{ref}}=10240$ grid points. We set $q_{\mathrm{f}}(t)=q_{\mathrm{u}}(t)=q_{\mathrm{e}}(t)=0$, which implies that $u_{\mathrm{f}}(t)=0$, $\bB{c}_{\mathrm{f}}(t)=\bB{0}$, and $\bB{s}_{\mathrm{f}}(t)=\bB{0}$. Furthermore, we prescribe the following smooth initial condition:
\begin{equation*}
u_0(z)      = 5+2 \psi(z) \,{\rm kg/m^3}, \qquad
\bB{c}_0(z) = u_0(z)\begin{pmatrix}5/7\\ 2/7\end{pmatrix}, \qquad
\bB{s}_0(z) = \begin{pmatrix} 0.5+0.05 \psi(z) \\ 1.0+0.1 \psi(z)\\ 0.1+0.01 \psi(z) \end{pmatrix}{\rm kg/m^3},
\end{equation*}
where $\psi(z) :=\varphi(z-1)$, and $\varphi$ is the function defined by
\begin{equation*}
\varphi(\zeta):=
\begin{cases}
(1-\zeta^{2})^{6} &\text{ if } |\zeta|<1,\\[2pt]
0&\text{ if } |\zeta|\geq 1.
\end{cases}
\end{equation*}
The approximate $L^1$ spatial errors defined by \eqref{eq:errors} and their corresponding orders of convergence given by \eqref{eq:order} are displayed in Table \ref{tab:errors:nondiff} at time $t =0.1\, {\rm h}$ for the three schemes, First-order, MUSCL and IRP-CWENO3.
We observe that the numerical errors effectively decrease upon mesh refinement for all components of the solution, and the order of convergence in all unknowns confirm first-, second- and third-order for the First-order, MUSCL, and IRP-CWENO3 schemes, respectively. In Figure~\ref{fig:errorplots} we show the error plots for the combined $L^1$ errors of the solid components given by  $e_N^{u}+e_N^{\bB{c},1} + e_N^{\bB{c},2}$, and the combined errors for the substrates $e_N^{\bB{s},1} +  e_N^{\bB{s},2} + e_N^{\bB{s},3}$ for the three already mentioned schemes. In all methods, we observe that the error curves are approximately parallel to the actual curves related to the exact orders of convergence $\mathcal{O}(h^1)$,  $\mathcal{O}(h^2)$ and  $\mathcal{O}(h^3)$, respectively.

\begin{table}[t]
\caption{Example 2: invariant-region-preserving without diffusion for the CWENO3 and IRP-CWENO3 schemes, respectively.
The minimum and maximum values related to set $\mathbcal{W}$ at the specified times $t$, denoted by underlining and overlining, respectively, are defined in \eqref{eq:IRP:quantities}. \label{tab:inv-region:CWENO3:IRP-CWENO3}}
\setlength{\cmidrulewidth}{0.7pt}
\centering
{\addtolength{\tabcolsep}{1pt}{\small
\begin{tabular}{r|ccccc|ccccc}
\multicolumn{1}{c}{$t=3\, \mathrm{h}$} & \multicolumn{5}{c}{CWENO3} & \multicolumn{5}{c}{IRP-CWENO3}\\
\toprule
$N$ &
$\underline{u}$& $\overline{u}$  & $\underline{\bB{c}}$ & $\underline{\bB{s}}$ & $\overline{|\,u-\bB{1}^{\tt t}\bB{c}\,|}$  &
$\underline{u}$& $\overline{u}$  & $\underline{\bB{c}}$ & $\underline{\bB{s}}$ & $\overline{|\,u-\bB{1}^{\tt t}\bB{c}\,|}$  \\
\cmidrule(lr){1-1} \cmidrule(lr){2-6} \cmidrule(lr){7-11}
40  & -1.08e-2 & 19.35 & -7.72e-3 & -5.93e-3 & 4.44e-12 & 0.0 & 19.35 & 0.0 & 0.0 &  4.38e-12\\
80  & -8.48e-3 & 21.18 & -6.05e-3 & -3.15e-3 & 5.30e-12 & 0.0 & 21.18 & 0.0 & 0.0 &  5.24e-12\\
160 & -4.45e-3 & 22.34 & -3.18e-3 & -8.31e-4 & 5.81e-12 & 0.0 & 22.34 & 0.0 & 0.0 &  5.62e-12\\
320 & -2.77e-3 & 22.82 & -1.97e-3 & -7.68e-4 & 5.91e-12 & 0.0 & 22.82 & 0.0 & 0.0 &  5.92e-12\\
640 & -1.68e-3 & 23.00 & -1.20e-3 & -5.80e-4 & 7.25e-12 & 0.0 & 23.00 & 0.0 & 0.0 &  7.25e-12\\
\midrule
\multicolumn{11}{c}{} \\[-2.0ex]
\multicolumn{1}{c}{$t=6\, \mathrm{h}$} & \multicolumn{5}{c}{CWENO3} & \multicolumn{5}{c}{IRP-CWENO3}\\[0.5ex]
\toprule
 $N$ &
$\underline{u}$& $\overline{u}$  & $\underline{\bB{c}}$ & $\underline{\bB{s}}$ & $\overline{|\,u-\bB{1}^{\tt t}\bB{c}\,|}$ &
$\underline{u}$& $\overline{u}$  & $\underline{\bB{c}}$ & $\underline{\bB{s}}$ & $\overline{|\,u-\bB{1}^{\tt t}\bB{c}\,|}$\\
\cmidrule(lr){1-1} \cmidrule(lr){2-6} \cmidrule(lr){7-11}
40  & -1.08e-2   &  19.91 & -7.72e-3 & -5.93e-3 &  5.30e-12 & 0.0  &  19.91 &  0.0  & 0.0 &  4.93e-12 \\
80  & -8.48e-3   &  21.18 & -6.05e-3 & -3.15e-3 &  6.47e-12 & 0.0  &  21.18 &  0.0  & 0.0 &  6.43e-12 \\
160 & -4.45e-3   &  22.34 & -3.18e-3 & -8.31e-4 &  7.76e-12 & 0.0  &  22.34 &  0.0  & 0.0 &  7.65e-12 \\
320 & -2.77e-3   &  22.82 & -1.97e-3 & -7.68e-4 &  9.83e-12 & 0.0  &  22.82 &  0.0  & 0.0 &  9.76e-12 \\
640 & -1.68e-3   &  23.00 & -1.20e-3 & -5.80e-4 &  1.30e-11 & 0.0  &  23.00 &  0.0  & 0.0 &  1.31e-11 \\
\bottomrule
\end{tabular}}}
\end{table}

\subsection{Example 2: Invariant-region-preservation without diffusion}

To show that the invariant-region results from Section~\ref{sec:invariant-region} hold for the IRP-CWENO3 scheme, i.e., the computed $(u_j^n,\bB{c}_j^n,\bB{s}_j^n)$ belongs to $\mathbcal{W}$ \eqref{eq:def:invariant:region:set:W} for all cells $I_j$ and time $t^n$, we consider a standard example of reactive sedimentation from \cite{SDIMA_MOL,SDcace_reactive}. For this purpose, we specify the following piecewise-constant-in-time feed concentrations:
\begin{align*}
\bB{s}_{\rm f} (t) &=
\begin{pmatrix} 0.006\\ 0.0009\\ 0 \end{pmatrix}{\rm kg/m^3},\qquad \bB{c}_{\rm f}(t) = u_{\rm f}(t)
\begin{pmatrix}
5/7 \\
2/7
\end{pmatrix}\,,\qquad u_{\rm f}(t) =
\begin{cases}
1.0\,\rm kg/m^3 &\mbox{ if } 0 \,{\rm h}  \leq t < 2 \,{\rm h} \,,\\
0.5\,\rm kg/m^3 &\mbox{ if } 2 \,{\rm h}  \leq t < 4 \,{\rm h} \,,\\
3.0\,\rm kg/m^3 &\mbox{ if } 4 \,{\rm h}  \leq t < 7 \,{\rm h} \,,\\
4.0\,\rm kg/m^3 &\mbox{ if } t \geq 7\, {\rm h}\,,
\end{cases}
\end{align*}
and the feed and underflow bulk velocities
\begin{align*}
q_{\rm f}(t) & =
\begin{cases}
1.125\,{\rm m/h}  &\mbox{ if } 0 \,{\rm h} \leq t < 2\,{\rm h}\,,\\
0.325\,{\rm m/h}  &\mbox{ if } 2 \,{\rm h}  \leq t < 4 \,{\rm h}\,,\\
1.625\,{\rm m/h} &\mbox{ if } t \geq 4\,{\rm h}\,,
\end{cases}\qquad
q_{\rm u}(t)  =
\begin{cases}
0.075\,{\rm m/h}   &\mbox{ if } 0 \,{\rm h} \leq t < 2 \,{\rm h}\,,\\
0.25\,{\rm m/h}    &\mbox{ if } 2 \,{\rm h} \leq t < 4 \,{\rm h}\,,\\
0.0875\,{\rm m/h}  &\mbox{ if } 4 \,{\rm h} \leq t < 7 \,{\rm h}\,,\\
0.125\,{\rm m/h}   &\mbox{ if } t \geq 7 \,{\rm h}\,,
\end{cases}
\end{align*}
where the effluent velocity is given by $q_{\rm e}(t) = q_{\rm f}(t) - q_{\rm u}(t)$. In addition, we also set the following initial functions:
\begin{align*}
\bB{c}_{0}(z) & = u_{0}(z)
\begin{pmatrix}
5/7 \\
2/7
\end{pmatrix}\,,\qquad
u_0(z) =
\begin{cases}
0 \,{\rm kg/m^3} &\mbox{ if } z < 0.5 \,{\rm m},\\
3.8z + 1.6 \,{\rm kg/m^3} &\mbox{ if }  z \geq 0.5 \,{\rm m},
\end{cases}\\
\bB{s}_0(z) &=
\begin{cases}
(0.006, 0, 0)^{\tt t} \,{\rm kg/m^3} &\mbox{ if } z < 0.5 \,{\rm m},\\
(0, 0.12(z-0.5), 0.006)^{\tt t}\,{\rm kg/m^3} &\mbox{ if }  z \geq 0.5 \,{\rm m}.
\end{cases}
\end{align*}
 In order to assess the effect of the scaling limiters described in section \ref{sec:limiters}, we monitor the minimum and maximum values related to set $\mathbcal{W}$ attained by each discrete unknown over the whole space-time computational grid, that is, we compute
\begin{equation}\label{eq:IRP:quantities}
	\begin{aligned}
		&\underline{u}:=\min_{(j,n)\in\mathbb{Z}_N\times\mathbb{Z}_{N_T}}\bigl\{u^n_j\bigr\},\quad	\overline{u}:=\max_{(j,n)\in\mathbb{Z}_N\times\mathbb{Z}_{N_T}}\bigl\{u^n_j\bigr\},\quad \underline{\bB{c}}:=\min_{\substack{(j,n)\in\mathbb{Z}_N\times\mathbb{Z}_{N_T}\\i\in \{1,\dots,\nsp\}}}\bigl\{c^{n,(i)}_j\bigr\},\\
		&\underline{\bB{s}}:=\min_{\substack{(j,n)\in\mathbb{Z}_N\times\mathbb{Z}_{N_T}\\l\in \{1,\dots,\ssp\}}}\bigl\{s^{n,(l)}_j\bigr\},  \quad \overline{|u-\bB{1}^{\tt t}\bB{c}|}:=\max_{(j,n)\in\mathbb{Z}_N\times\mathbb{Z}_{N_T}}\bigl\{|u_j^n-\bB{1}^{\tt t}\bB{c}_j^n|\bigr\},
	\end{aligned}
\end{equation}
where both $\underline{\bB{c}}$ and $\underline{\bB{s}}$ are scalars and $N_T$ stands for the number of time steps. The quantities defined in \eqref{eq:IRP:quantities}, computed with the CWENO3 scheme without limiters, and with the invariant-region-preserving scheme IRP-CWENO3, are reported in Table~\ref{tab:inv-region:CWENO3:IRP-CWENO3} for different mesh sizes at two simulation times, namely, $t=3\,\mathrm{h}$ and $t=6\,\mathrm{h}$. For the schemes without limiters, negative values of $\underline{u}, \underline{\bB{c}}$, and $\underline{\bB{s}}$ occur in the numerical solutions for all mesh sizes and at both simulation times. By contrast, when the limiters are applied, the numerical solutions remain within $\mathbcal{W}$, as shown in Section~\ref{sec:invariant-region}. In addition, in Figure \ref{fig:fig4}, we show the numerical solution obtained with the IRP-CWENO3 and observe that the scheme accurately reproduce the test case from \cite{SDIMA_MOL}. Figure~\ref{fig:fig5} compares the first-, second- and third-order schemes on a fixed grid with $N=160$ cells against a reference solution computed with the IRP-CWENO3 scheme on a finer mesh (with $N_{\rm ref} = 640$) and we observe that the numerical solution approaches the reference one as the order of the scheme increases.

\begin{figure}[t]
	\centering
	\includegraphics[width=\textwidth]{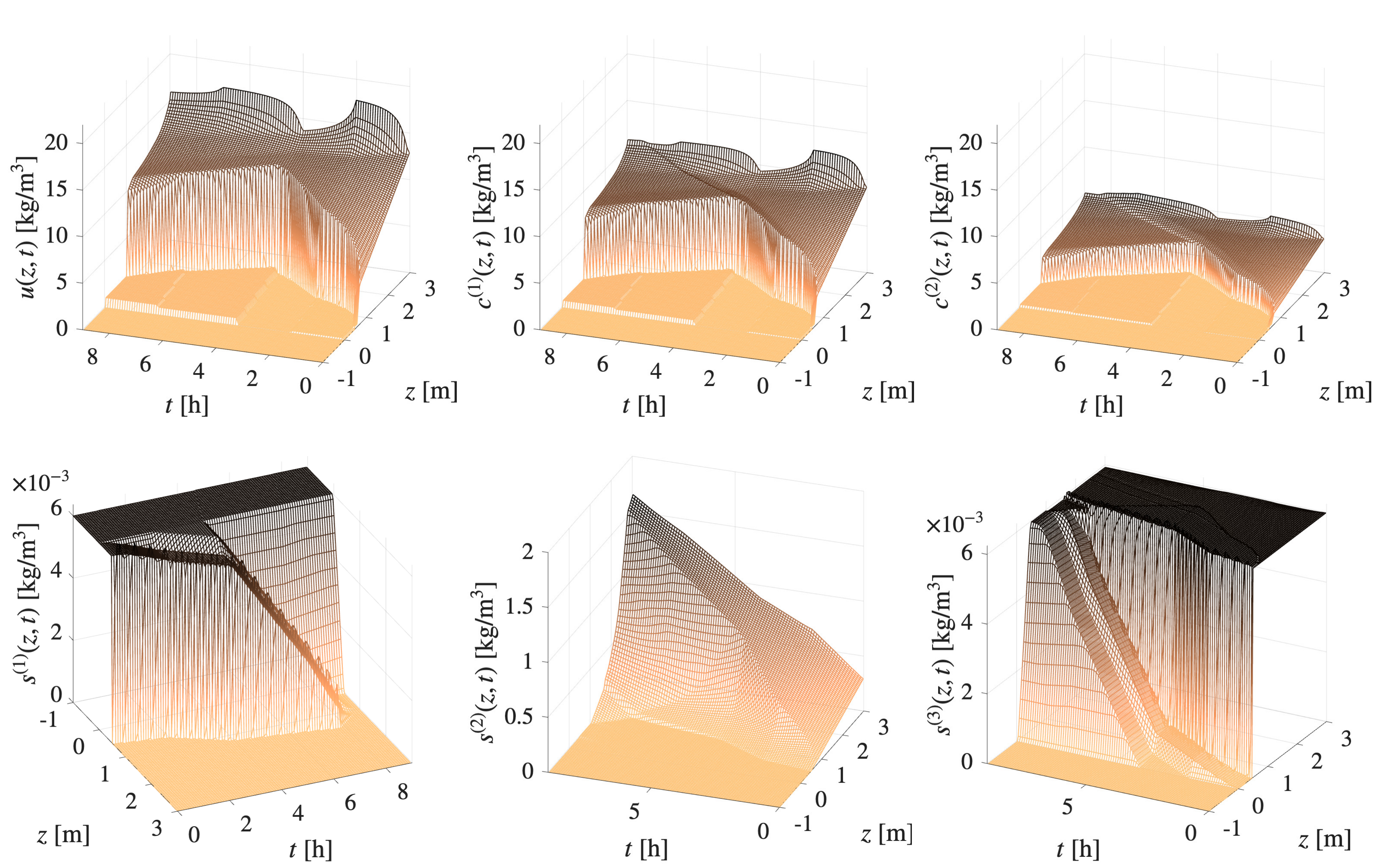}
	\caption{Example 2: numerical solution with $N = 640$ produced with IRP-CWENO3 scheme at $t = 9\, \mathrm{h}$. The solution has been projected onto a coarse grid.\label{fig:fig4}}
\end{figure}
\begin{figure}[t]
	\centering
	\includegraphics[width=0.8\textwidth]{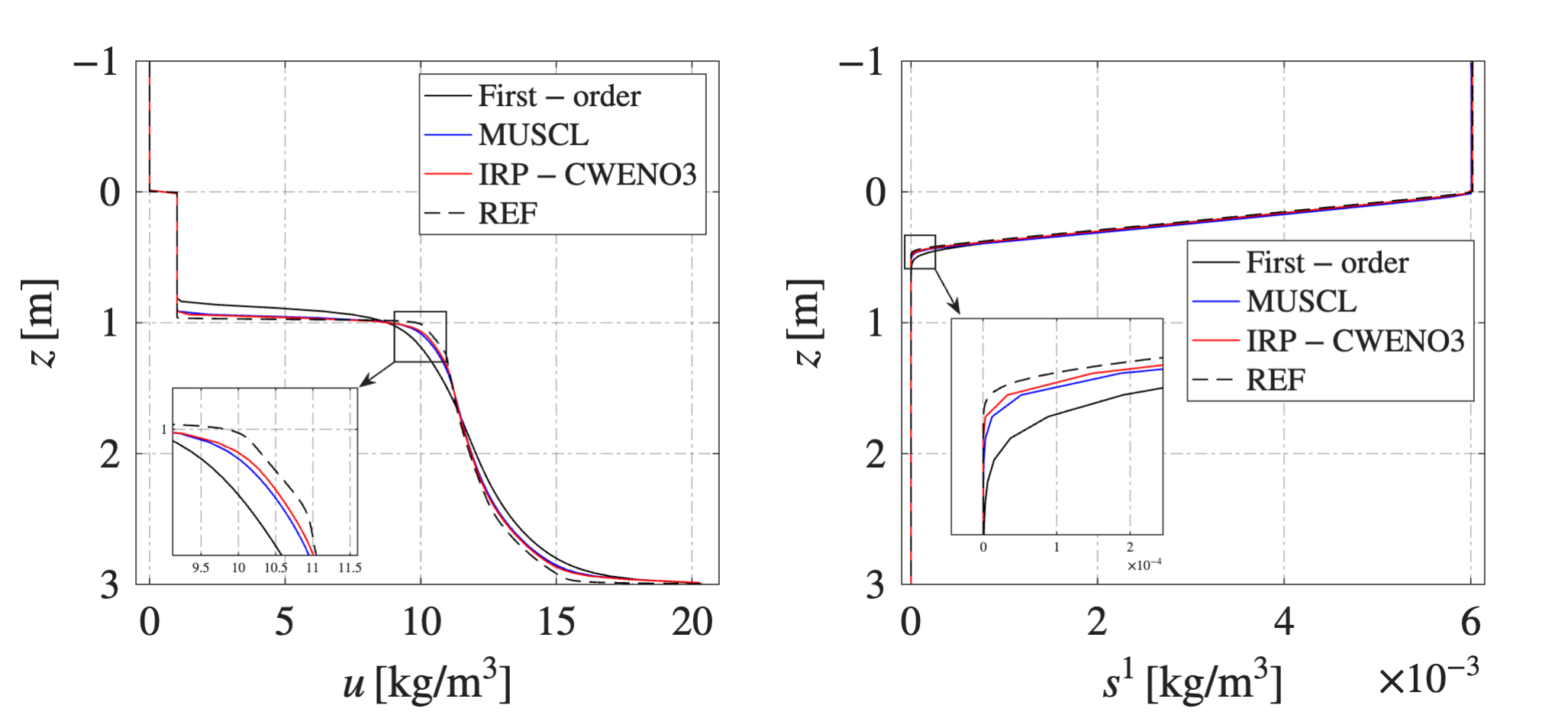}
	\caption{Example 2: numerical solutions with $N = 160$ produced with first-order, MUSCL and IRP-CWENO3 schemes at $t = 9\, \mathrm{h}$. The reference solution is computed with $N_{\mathrm{ref}} = 640$ using IRP-CWENO3 scheme.\label{fig:fig5}}
\end{figure}

\subsection{Example 3: Diffusive case with compression effects}
In this example, we consider the inclusion of the diffusion-related term $\partial_z\mathcal{B}$ (resp.~$\partial_z\mathcal{D}$) in Equation~\eqref{eq:model:u}, which accounts for the compression effects above a critical concentration $u_{\rm c}$, and its corresponding flux term $\mathcal{E}$ (resp.~$\mathcal{J}$). We study the MUSCL and IRP-CWENO3 schemes combined with approximations of the diffusion-related operator given by $\mathcal{E}_{j+1/2}$ following the $\mathcal{D}$-method and $\mathcal{B}$-method from Section~\ref{sec:numscheme}.
For this we consider the second- and fourth-order approximations of $\mathcal{E}$ (resp.~$\mathcal{J}$) defined in \eqref{eq:E:second} and \eqref{eq:E:fourth} (resp.~\eqref{eq:J:second} and \eqref{eq:J:fourth}).

First, we use the functions and initial conditions from Example 1 to assess the convergence rates of both schemes in combination with the two proposed methods to approximate the diffusion terms. We then compute a reference solution with the IRP-CWENO3 scheme using the $\mathcal{D}$-method with $N_{\rm ref} = 2560$ and $\mathcal{J}_{j+1/2} = \mathcal{J}_{j+1/2}^{\rm 4th}$ until time $t=0.1\,\rm h$. The spatial numerical errors are reported in Table~\ref{tab:errors:diff} and show that both schemes, MUSCL and IRP-CWENO3, achieve their expected orders of convergence with both approaches, $\mathcal{D}$-method and $\mathcal{B}$-method, respectively. Next, we use the functions and initial conditions from Example 2 and simulate the continuous reactive sedimentation problem up to $t=9\, \mathrm{h}$ with $N=640$. The corresponding numerical solutions are presented in Figure \ref{fig:fig6}, where we also observe the evolution of the expected discontinuity produced due to the boundary and initial conditions, and the corresponding smoothing effect in the plot of $u$ for concentrations above $u_{\rm c}$. In addition, using this setting, we investigate the invariant-region-preserving property of the MUSCL and IRP-CWENO3 schemes when combined with the proposed discretizations of the diffusion-related terms. The results reported in Table~\ref{tab:inv-region:diff:MUSCL:IRP-CWENO3} show that the invariant region is preserved, that is, $\underline{u}\approx 0$, $\underline{\bB{c}}\approx 0$, $\underline{\bB{s}}\approx 0$, $\underline{u}<u_{\max}$ and $\overline{|u-\bB{1}^{\tt t}\bB{c}|}\approx 0$, in both approaches, with the $\mathcal{D}$-method and $\mathcal{B}$-method, respectively.

\begin{table}[t]
\caption{Example 3: numerical errors \eqref{eq:errors} and approximate order of convergence \eqref{eq:order} at time $t=0.1\,\rm h$ computed with the MUSCL and IRP-CWENO3 schemes, respectively. The diffusion function is approximated using \eqref{eq:J:second}.
The reference solution is computed using the IRP-CWENO3 scheme with $N_{\rm ref} = 2560$ cells.\label{tab:errors:diff}}
\setlength{\cmidrulewidth}{0.7pt}
\centering
{\addtolength{\tabcolsep}{-2pt}{\small
\begin{tabular}{r|cc|cccc|cccccc}
\multicolumn{13}{c}{MUSCL \, \& \,  $\mathcal{D}$-method with $\mathcal{J}_{j+1/2}^{\rm 2nd}$}\\[0.5ex]
\toprule
$N$ &
$e^{u}_N$& $\mathcal{R}^{u}_N$ & $e^{\bB{c},1}_N$ & $\mathcal{R}^{\bB{c},1}_N$& $e^{\bB{c},2}_N$ & $\mathcal{R}^{\bB{c},2}_N$ &
$e^{\bB{s},1}_N$& $\mathcal{R}^{\bB{s},1}_N$ & $e^{\bB{s},2}_N$ & $\mathcal{R}^{\bB{s},2}_N$& $e^{\bB{s},3}_N$ & $\mathcal{R}^{\bB{s},3}_N$\\[0.5ex]
\cmidrule(lr){1-1}\cmidrule(lr){2-3}\cmidrule(lr){4-5} \cmidrule(lr){6-7}\cmidrule(lr){8-9} \cmidrule(lr){10-11} \cmidrule(lr){12-13}
80  & 1.42e-3 & --   & 1.02e-3 & --   & 4.00e-4 & --   & 3.43e-6 & --   & 2.21e-5 & --   & 2.68e-6 & --   \\
160 & 3.99e-4 & 1.83 & 2.87e-4 & 1.83 & 1.13e-4 & 1.83 & 1.00e-6 & 1.78 & 6.31e-6 & 1.81 & 7.62e-7 & 1.81 \\
320 & 1.13e-4 & 1.82 & 8.10e-5 & 1.82 & 3.18e-5 & 1.82 & 2.73e-7 & 1.88 & 1.80e-6 & 1.81 & 2.14e-7 & 1.84 \\
640 & 2.99e-5 & 1.92 & 2.15e-5 & 1.92 & 8.44e-6 & 1.92 & 7.17e-8 & 1.93 & 4.78e-7 & 1.91 & 5.62e-8 & 1.93 \\
\midrule
\multicolumn{13}{c}{MUSCL \, \& \,  $\mathcal{B}$-method with $\mathcal{E}_{j+1/2}^{\rm 2nd}$}\\[0.5ex]
\toprule
80  & 1.44e-3 & --   & 1.03e-3 & --   & 4.05e-4 & --   & 3.51e-6 & --   & 2.29e-5 & --   & 2.75e-6 & --   \\
160 & 3.93e-4 & 1.87 & 2.82e-4 & 1.87 & 1.11e-4 & 1.87 & 9.94e-7 & 1.82 & 6.24e-6 & 1.87 & 7.45e-7 & 1.88 \\
320 & 1.11e-4 & 1.83 & 7.94e-5 & 1.83 & 3.12e-5 & 1.83 & 2.68e-7 & 1.89 & 1.75e-6 & 1.83 & 2.08e-7 & 1.84 \\
640 & 2.93e-5 & 1.92 & 2.10e-5 & 1.92 & 8.27e-6 & 1.92 & 7.01e-8 & 1.93 & 4.65e-7 & 1.91 & 5.47e-8 & 1.93 \\
\midrule
\multicolumn{11}{c}{} \\[-2.0ex]
\multicolumn{13}{c}{IRP-CWENO3 \, \& \,  $\mathcal{D}$-method with $\mathcal{J}_{j+1/2}^{\rm 4th}$}\\[0.5ex]
\toprule
80  & 1.59e-3 & --   & 1.14e-3 & --   & 4.48e-4 & --   & 4.01e-6 & --   & 2.84e-5 & --   & 3.24e-6 & --   \\
160 & 1.93e-4 & 3.04 & 1.39e-4 & 3.04 & 5.45e-5 & 3.04 & 4.81e-7 & 3.06 & 3.39e-6 & 3.07 & 3.87e-7 & 3.06 \\
320 & 2.49e-5 & 2.95 & 1.79e-5 & 2.95 & 7.04e-6 & 2.95 & 5.99e-8 & 3.01 & 4.19e-7 & 3.02 & 4.75e-8 & 3.03 \\
640 & 3.56e-6 & 2.81 & 2.56e-6 & 2.81 & 1.00e-6 & 2.81 & 7.78e-9 & 2.95 & 5.33e-8 & 2.97 & 5.95e-9 & 3.00 \\
\midrule
\multicolumn{11}{c}{} \\[-2.0ex]
\multicolumn{13}{c}{IRP-CWENO3 \, \& \,  $\mathcal{B}$-method with $\mathcal{E}_{j+1/2}^{\rm 4th}$}\\[0.5ex]
\toprule
80  & 1.60e-3 & --   & 1.15e-3 & --   & 4.52e-4 & --   & 4.04e-6 & --   & 2.87e-5 & --   & 3.28e-6 & --   \\
160 & 1.94e-4 & 3.05 & 1.39e-4 & 3.05 & 5.47e-5 & 3.05 & 4.84e-7 & 3.06 & 3.41e-6 & 3.07 & 3.90e-7 & 3.07 \\
320 & 2.50e-5 & 2.96 & 1.79e-5 & 2.96 & 7.05e-6 & 2.96 & 6.01e-8 & 3.01 & 4.20e-7 & 3.02 & 4.77e-8 & 3.03 \\
640 & 3.56e-6 & 2.81 & 2.56e-6 & 2.81 & 1.01e-6 & 2.81 & 7.79e-9 & 2.95 & 5.34e-8 & 2.98 & 5.96e-9 & 3.00 \\
\bottomrule
\end{tabular}}}
\end{table}

\begin{figure}[t]
	\centering
	\includegraphics[width=\textwidth]{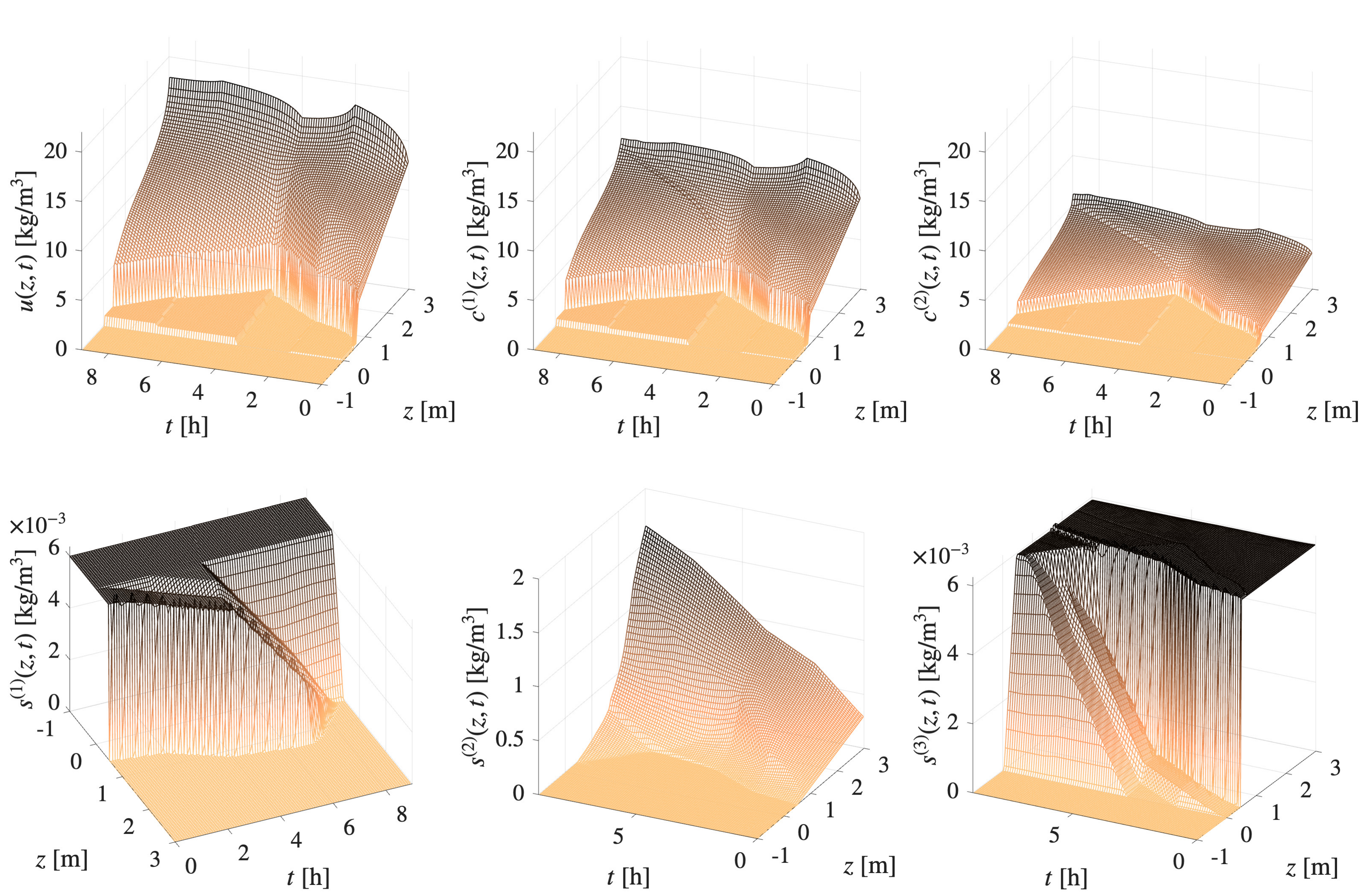}
	\caption{Example 3: numerical solution with $N = 640$ produced with IRP-CWENO3 scheme at $t = 9\, \mathrm{h}$. The solution has been projected onto a coarse grid.\label{fig:fig6}}
\end{figure}

\begin{table}[t]
\caption{Example 3: invariant-region-preserving test including diffusion for the MUSCL and IRP-CWENO3 schemes with the $\mathcal{D}$-method and $\mathcal{B}$-method from Section~\ref{sec:numscheme}, respectively.
The minimum and maximum values related to set $\mathbcal{W}$ at the specified time $t$, denoted by underlining and overlining, respectively, are defined in \eqref{eq:IRP:quantities}. \label{tab:inv-region:diff:MUSCL:IRP-CWENO3}}
\setlength{\cmidrulewidth}{0.7pt}
\centering
{\addtolength{\tabcolsep}{1pt}{\small
\begin{tabular}{r|ccccc|ccccc}
\multicolumn{11}{c}{MUSCL}\\[0.7ex]
\toprule
\multicolumn{1}{c|}{$t=9\, \mathrm{h}$} &
\multicolumn{5}{c|}{$\mathcal{D}$-method with $\mathcal{J}_{j+1/2}^{\rm 2nd}$} &
\multicolumn{5}{c}{ $\mathcal{B}$-method with $\mathcal{E}_{j+1/2}^{\rm 2nd}$}\\[0.5ex]
\cmidrule(lr){1-1} \cmidrule(lr){2-6}\cmidrule(lr){7-11}
$N$ &
$\underline{u}$& $\overline{u}$  & $\underline{\bB{c}}$ & $\underline{\bB{s}}$ & $\overline{|\,u-\bB{1}^{\tt t}\bB{c}\,|}$  &
$\underline{u}$& $\overline{u}$  & $\underline{\bB{c}}$ & $\underline{\bB{s}}$ & $\overline{|\,u-\bB{1}^{\tt t}\bB{c}\,|}$  \\
\cmidrule(lr){1-1} \cmidrule(lr){2-6} \cmidrule(lr){7-11}
40  &  0.0 & 18.50 & 0.0 & 0.0 & 1.37e-12    &   0.0 & 18.49 & 0.0 & 0.0 & 1.44e-12\\
80  &  0.0 & 18.88 & 0.0 & 0.0 & 1.35e-12    &   0.0 & 18.88 & 0.0 & 0.0 & 9.63e-13\\
160 &  0.0 & 19.04 & 0.0 & 0.0 & 2.28e-12    &   0.0 & 19.04 & 0.0 & 0.0 & 9.63e-13\\
320 &  0.0 & 19.11 & 0.0 & 0.0 & 6.45e-12    &   0.0 & 19.11 & 0.0 & 0.0 & 1.04e-12\\
\midrule
\multicolumn{11}{c}{} \\[-1.0ex]
\multicolumn{11}{c}{IRP-CWENO3}\\[0.7ex]
\toprule
\multicolumn{1}{c|}{$t=9\, \mathrm{h}$} &
\multicolumn{5}{c|}{$\mathcal{D}$-method with $\mathcal{J}_{j+1/2}^{\rm 4th}$ } &
\multicolumn{5}{c}{ $\mathcal{B}$-method with $\mathcal{E}_{j+1/2}^{\rm 4th}$}\\[0.5ex]
\cmidrule(lr){1-1} \cmidrule(lr){2-6}\cmidrule(lr){7-11}
$N$ &
$\underline{u}$& $\overline{u}$  & $\underline{\bB{c}}$ & $\underline{\bB{s}}$ & $\overline{|\,u-\bB{1}^{\tt t}\bB{c}\,|}$  &
$\underline{u}$& $\overline{u}$  & $\underline{\bB{c}}$ & $\underline{\bB{s}}$ & $\overline{|\,u-\bB{1}^{\tt t}\bB{c}\,|}$  \\
\cmidrule(lr){1-1} \cmidrule(lr){2-6}\cmidrule(lr){7-11}
40  & 0.0 & 18.45 & 0.0 & 0.0 & 5.26e-12 & 0.0 & 18.44 & 0.0 & 0.0 & 5.60e-12\\
80  & 0.0 & 18.85 & 0.0 & 0.0 & 8.51e-12 & 0.0 & 18.85 & 0.0 & 0.0 & 8.64e-12\\
160 & 0.0 & 19.03 & 0.0 & 0.0 & 1.15e-11 & 0.0 & 19.03 & 0.0 & 0.0 & 1.28e-11\\
320 & 0.0 & 19.10 & 0.0 & 0.0 & 2.40e-11 & 0.0 & 19.10 & 0.0 & 0.0 & 3.15e-11\\
\midrule
\end{tabular}}
}
\end{table}

\section{Conclusions}\label{sec:conclusions}

In this work, we have developed high-order finite volume numerical schemes for the approximation of a model of reactive sedimentation, which consists of a system of convection-diffusion-reaction equations.
Two component-wise polynomial reconstructions have been proposed, a second-order MUSCL and a third-order central WENO scheme (CWENO3), which in both cases, the polynomial reconstructions for the solid components are designed such that the sum of each solid component gives the total concentration of solids.
Positivity of each concentration component, and boundedness of the total concentration of solids are established directly from the definition of the MUSCL scheme, while limiters were implemented for the CWENO3 scheme. In this way, we adapted the standard Zhang and Shu \cite{zhang2010maximum,zhang2011maximum} limiters to achieve that the CWENO3 polynomials lie in the set of physically relevant solutions $\mathbcal{W}$ \eqref{eq:def:invariant:region:set:W}.

As a first result, we show that under an appropriate CFL condition \eqref{eq:CFL:convective}, each component of the numerical solution computed with the related first-order upwind-based scheme without compression effects (no diffusion) and forward Euler in time satisfies an invariant-region-preserving property in $\mathbcal{W}$.
Some ingredients used to obtain this result were the decreasing and Lipschitz continuity condition of the hindered-settling velocity $\vhs$ and the hypothesis on the reaction terms given in Assumption~\ref{asumption:reactionterms}. Unlike previous works \cite{SDIMA_MOL,SDAMM_SBR2,Burger2023}, these invariant-region properties for the first-order scheme were achieved without using a monotonicity argument.
The main theoretical results are the invariant-region-preserving property in $\mathbcal{W}$ for the high-order schemes with the MUSCL and CWENO3 schemes, respectively, with forward Euler time approximation, and also with the strong-stability preserving (SSP) second- and third-order TVD Runge-Kutta time discretizations. The drawback, in the case of the higher-order time discretization, is a more restrictive CFL condition.
For the approximation of the diffusion-related terms, we proposed two approaches, one discretizing the standard compression function $\mathcal{D}$ \cite{Burger2005} which corresponds to a degenerate diffusion function for the total concentration $u$, and a second approach for based on the factor function $\mathcal{B}$. In both cases, the diffusion-related flux $\mathcal{E}$ \eqref{eq:def:E} was included in the approximation of the total flux  $\mathcal{F}$ \eqref{eq:def:F} of the equation for the total concentration $u$, which was based on an upwind scheme combined with the corresponding polynomial reconstructions. This particular approximation of $\mathcal{F}$ \eqref{eq:def:F} including the diffusion-related flux $\mathcal{E}$, and the proposed approximation for the equation of the vector of solid components $\bB{c}$, was devised such that the sum of the solid components equals $u$. The numerical results corroborate the theoretical findings, in particular the invariant-region property, and also the order of accuracy in the case without diffusion. Notably, the numerical examples show that the proposed approaches to discretize the diffusion-related terms preserve the desired order of convergence, and also preserve the invariant-region property.

Regarding future extensions of this work, the main line of research is to establish the theoretical orders of accuracy of the proposed schemes for the diffusion terms, and to show that these schemes also preserve the invariant-region property. An analysis of the stability of the proposed numerical schemes also deserves to be the subject of future research. Further studies of the reactive sedimentation model in multidimensional context, either in axisymmetric domains or three dimensional ones, are of research interest. However, multidimensional models often require solving additional Navier-Stokes equations for the bulk velocity, which becomes a vector, and the so-called excess pore-pressure. Therefore, the bulk velocity is no longer a given function, leading to the need of additional polynomial reconstructions for this bulk velocity. Finally, as the model is written in terms of an arbitrary number of solid and substrate components, this work is suitable for the simulation of more general activated sludge processes involving several concentrations of solid particles and substrates, such as the ASM1 \cite{Henze2000} or general eutrophication processes \cite{Blauw2008, Gargallo2017, Markelov2019, Pauer2000, Tipping2016}.

\section*{Acknowledgments}

This work was partially supported by ANID-Chile through Fondecyt project No.~3230553. LMV is supported by ANID - Chile through of Centro de Modelamiento Matemático (CMM), project FB210005 of BASAL funds for Centers of Excellence. JBC is supported by the National Agency for Research and Development, ANID-Chile through the Scholarship Program, Becas Doctorado Nacional 2022, 21221387.

\bibliographystyle{plain}

\begin{thebibliography}{10}

\bibitem{barajas2025invariant}
J.~Barajas-Calonge, R.~B{\"u}rger, P.~Mulet, and L.M. Villada.
\newblock {Invariant-region-preserving WENO schemes for one-dimensional
  multispecies kinematic flow models}.
\newblock {\em J. Comput. Phys.}, 537:114081, 2025.

\bibitem{barajas2026second}
J.~Barajas-Calonge, R.~B{\"u}rger, P.~Mulet, and L.M. Villada.
\newblock {A Second-Order Invariant-Region-Preserving Scheme for a
  Transport-Flow Model of Polydisperse Sedimentation}.
\newblock {\em J. Sci. Comput.}, 108(1):29, 2026.

\bibitem{Blauw2008}
A.N. Blauw, H.F.J. Los, M.~Bokhorst, and P.L.A. Erftemeijer.
\newblock Gem: a generic ecological model for estuaries and coastal waters.
\newblock {\em Hydrobiologia}, 618(1), 2008.


\bibitem{SDIMA_MOL}
R.~B\"urger, J.~Careaga, and S.~Diehl.
\newblock A method-of-lines formulation for a model of reactive settling in
  tanks with varying cross-sectional area.
\newblock {\em IMA J. Appl. Math.}, 86(3):514--546, 2021.

\bibitem{SDcace_reactive}
R.~B{\"u}rger, J.~Careaga, S.~Diehl, C.~Mej\'{i}as, I.~Nopens, E.~Torfs, and
  P.~A. Vanrolleghem.
\newblock Simulations of reactive settling of activated sludge with a reduced
  biokinetic model.
\newblock {\em Computers Chem. Eng.}, 92:216--229, 2016.

\bibitem{SDAMM_SBR1}
R.~B\"{u}rger, J.~Careaga, S.~Diehl, and R.~Pineda.
\newblock A moving-boundary model of reactive settling in wastewater treatment.
  {P}art~1: Governing equations.
\newblock {\em Appl. Math. Modelling}, 106:390--401, 2022.

\bibitem{SDAMM_SBR2}
R.~B\"{u}rger, J.~Careaga, S.~Diehl, and R.~Pineda.
\newblock A moving-boundary model of reactive settling in wastewater treatment.
  {P}art~2: {N}umerical scheme.
\newblock {\em Appl. Math. Modelling}, 111:247--269, 2022.

\bibitem{Burger2023}
R.~B\"{u}rger, J.~Careaga, S.~Diehl, and R.~Pineda.
\newblock Numerical schemes for a moving-boundary convection-diffusion-reaction
  model of sequencing batch reactors.
\newblock {\em ESAIM: Math. Model. Numer. Anal.}, 57(5):2931--2976, 2023.

\bibitem{SDm2an_reactive}
R.~B\"{u}rger, S.~Diehl, and C.~Mej{\'{\i}}as.
\newblock A difference scheme for a degenerating convection-diffusion-reaction
  system modelling continuous sedimentation.
\newblock {\em {ESAIM}: Math. Modelling Num. Anal.}, 52(2):365--392, 2018.

\bibitem{Burger2005}
R.~B\"{u}rger, K.H. Karlsen, and J.D. Towers.
\newblock A model of continuous sedimentation of flocculated suspensions in
  clarifier-thickener units.
\newblock {\em SIAM J. Appl. Math.}, 65(3):882--940, 2005.

\bibitem{Chancelier1994}
J.-Ph. Chancelier, M.~Cohen~De Lara, and F.~Pacard.
\newblock {Analysis of a Conservation PDE With Discontinuous Flux: A Model of
  Settler}.
\newblock {\em SIAM J. Appl. Math.}, 54(4):954--995, 1994.

\bibitem{DeClercq2008}
J.~De Clercq, I.~Nopens, J.~Defrancq, and P.A. Vanrolleghem.
\newblock {Extending and calibrating a mechanistic hindered and compression
  settling model for activated sludge using in-depth batch experiments}.
\newblock {\em Water Res.}, 42(3):781--791, 2008.

\bibitem{Cravero2017}
I.~Cravero, G.~Puppo, M.~Semplice, and G.~Visconti.
\newblock Cweno: Uniformly accurate reconstructions for balance laws.
\newblock {\em Math. Comput.}, 87(312):1689--1719, 2017.

\bibitem{Cruz2012}
M.N.~Cruz Bournazou, H.~Arellano‐Garcia, G.~Wozny, G.~Lyberatos, and
  C.~Kravaris.
\newblock {ASM3 extended for two‐step nitrification–denitrification: a
  model reduction for sequencing batch reactors}.
\newblock {\em J. Chem. Technol. Biotechnol.}, 87(7):887--896, 2012.

\bibitem{Diehl1996}
S.~Diehl.
\newblock {A conservation Law with Point Source and Discontinuous Flux Function
  Modelling Continuous Sedimentation}.
\newblock {\em SIAM J. Appl. Math.}, 56(2):388--419, 1996.

\bibitem{FloresAlsina2012}
X.~Flores-Alsina, K.V. Gernaey, and U.~Jeppsson.
\newblock Benchmarking biological nutrient removal in wastewater treatment
  plants: influence of mathematical model assumptions.
\newblock {\em Water Sci. Technol.}, 65(8):1496–1505, 2012.

\bibitem{Gargallo2017}
S.~Gargallo, M.~Martín, N.~Oliver, and C.~Hernández-Crespo.
\newblock {Sedimentation and resuspension modelling in free water surface
  constructed wetlands}.
\newblock {\em Ecol. Eng.}, 98:318–329, 2017.

\bibitem{gottlieb1998total}
S.~Gottlieb and C.-W. Shu.
\newblock {Total variation diminishing Runge-Kutta schemes}.
\newblock {\em Math. Comput.}, 67(221):73--85, 1998.

\bibitem{gottlieb2001strong}
S.~Gottlieb, C.-W. Shu, and E.~Tadmor.
\newblock Strong stability-preserving high-order time discretization methods.
\newblock {\em SIAM review}, 43(1):89--112, 2001.

\bibitem{Henze1987}
M.~Henze, C.P.L. Grady, W.~Gujer, G.V.R. Marais, and T.~Matsuo.
\newblock {Activated Sludge Model No. 1}.
\newblock Sci. tech. rep. no. 1., IAWQ, London, UK, 1987.
\newblock No DOI assigned.

\bibitem{Henze2000}
M.~Henze, W.~Gujer, T.~Mino, and M.C.M. van Loosdrecht.
\newblock {\em Activated Sludge Models ASM1, ASM2, ASM2d and ASM3}.
\newblock IWA Publishing, London, UK, 2000.
\newblock No DOI assigned.

\bibitem{Jeppsson1993}
U.~Jeppsson and G.~Olsson.
\newblock Reduced order models for on-line parameter identification of the
  activated sludge process.
\newblock {\em Water Sci. Technol.}, 28(11-12):173--183, 1993.

\bibitem{Jiang1996}
G.-S. Jiang and C.-W. Shu.
\newblock Efficient implementation of weighted {ENO} schemes.
\newblock {\em J. Comput. Phys.}, 126(1):202--228, 1996.

\bibitem{kolgan1972application}
V.P. Kolgan.
\newblock {Application of the minimum-derivative principle in the construction
  of finite-difference schemes for numerical analysis of discontinuous
  solutions in gas dynamics}.
\newblock {\em Uch. Zap. TsAGI [Sci. Notes Central Inst. Aerodyn]},
  3(6):68--77, 1972.

\bibitem{kuzmin2024property}
D.~Kuzmin and H.~Hajduk.
\newblock {\em {Property-preserving numerical schemes for conservation laws}}.
\newblock World Scientific, 2024.

\bibitem{Kynch1952}
G.J. Kynch.
\newblock A theory of sedimentation.
\newblock {\em Trans. Faraday Soc.}, 48:166, 1952.

\bibitem{Julien1998}
P. Lessard S.~Julien, J.P.~Babary.
\newblock Theoretical and practical identifiability of a reduced order model in
  an activated sludge process doing nitrification and denitrification.
\newblock {\em Water Sci. Technol.}, 37(12), 1998.

\bibitem{LeVeque1990Green}
R.J. LeVeque.
\newblock {\em Numerical Methods for Conservation Laws}.
\newblock Birkh\"{a}user Basel, 1990.

\bibitem{Levy2000}
D.~Levy, G.~Puppo, and G.~Russo.
\newblock Compact central {WENO} schemes for multidimensional conservation
  laws.
\newblock {\em {SIAM} J.~Sci.~Comput.}, 22(2):656--672, 2000.

\bibitem{Li2014}
B.~Li and M.K. Stenstrom.
\newblock {Research advances and challenges in one-dimensional modeling of
  secondary settling Tanks – A critical review}.
\newblock {\em Water Res.}, 65:40--63, 2014.

\bibitem{Li2021}
J.~Li, C.-W. Shu, and J.~Qiu.
\newblock Multi-resolution {HWENO} schemes for hyperbolic conservation laws.
\newblock {\em J.~Comput.~Phys.}, 446:110653, 2021.

\bibitem{Markelov2019}
I.~Markelov, R.‐M. Couture, R.~Fischer, S.~Haande, and P.~Van Cappellen.
\newblock {Coupling Water Column and Sediment Biogeochemical Dynamics: Modeling
  Internal Phosphorus Loading, Climate Change Responses, and Mitigation
  Measures in Lake Vansjø, Norway}.
\newblock {\em J. Geophys. Res. Biogeosci.}, 124(12):3847--3866, 2019.

\bibitem{Ni2008}
B.-J. Ni and H.-Q. Yu.
\newblock An approach for modeling two-step denitrification in activated sludge
  systems.
\newblock {\em Chem. Eng. Sci.}, 63(6):1449--1459, 2008.

\bibitem{Liu1994}
S.~Osher X.-D.~Liu and T.~Chan.
\newblock Weighted essentially non-oscillatory schemes.
\newblock {\em J. Comput. Phys.}, 115(1):200--212, 1994.

\bibitem{Pauer2000}
J.~Pauer.
\newblock Nitrification in the water column and sediment of a hypereutrophic
  lake and adjoining river system.
\newblock {\em Water Res.}, 34(4):1247--1254, 2000.

\bibitem{Shu1998}
C.-W. Shu.
\newblock {\em Essentially non-oscillatory and weighted essentially
  non-oscillatory schemes for hyperbolic conservation laws}, pages 325--432.
\newblock Springer Berlin Heidelberg, 1998.

\bibitem{Shu2009}
C.-W. Shu.
\newblock High order weighted essentially nonoscillatory schemes for convection
  dominated problems.
\newblock {\em {SIAM} Review}, 51(1):82--126, 2009.

\bibitem{Shu1988}
C.-W. Shu and S.~Osher.
\newblock Efficient implementation of essentially non-oscillatory
  shock-capturing schemes.
\newblock {\em J. Comput. Phys.}, 77(2):439--471, 1988.

\bibitem{Tipping2016}
E.~Tipping, J.F. Boyle, D.N. Schillereff, B.M. Spears, and G.~Phillips.
\newblock {Macronutrient processing by temperate lakes: A dynamic model for
  long-term, large-scale application}.
\newblock {\em Sci. Total Environ.}, 572:1573--1585, 2016.

\bibitem{van1979towards}
B.~Van Leer.
\newblock {Towards the ultimate conservative difference scheme. V. A
  second-order sequel to Godunov's method}.
\newblock {\em J. Comput. Phys.}, 32(1):101--136, 1979.

\bibitem{vanHaandel1981}
A.~van Haandel, G.~Ekama, and G.~Marais.
\newblock The activated sludge process—3 single sludge denitrification.
\newblock {\em Water Res.}, 15(10):1135–1152, 1981.

\bibitem{xing2013positivity}
Y.~Xing and X.~Zhang.
\newblock {Positivity-preserving well-balanced discontinuous Galerkin methods
  for the shallow water equations on unstructured triangular meshes}.
\newblock {\em J. Sci. Comput.}, 57(1):19--41, 2013.

\bibitem{xing2010positivity}
Y.~Xing, X.~Zhang, and C.-W. Shu.
\newblock {Positivity-preserving high order well-balanced discontinuous
  Galerkin methods for the shallow water equations}.
\newblock {\em Adv. Water Resour.}, 33(12):1476--1493, 2010.

\bibitem{zhang2010maximum}
X.~Zhang and C.-W. Shu.
\newblock {On maximum-principle-satisfying high order schemes for scalar
  conservation laws}.
\newblock {\em J. Comput. Phys.}, 229(9):3091--3120, 2010.

\bibitem{zhang2010positivity}
X.~Zhang and C.-W. Shu.
\newblock {On positivity-preserving high order discontinuous Galerkin schemes
  for compressible Euler equations on rectangular meshes}.
\newblock {\em J. Comput. Phys.}, 229(23):8918--8934, 2010.

\bibitem{zhang2011maximum}
X.~Zhang and C.-W. Shu.
\newblock {Maximum-principle-satisfying and positivity-preserving high-order
  schemes for conservation laws: survey and new developments}.
\newblock {\em Proc. R. Soc. A: Math. Phys. Eng. Sci.}, 467(2134):2752--2776,
  2011.

\bibitem{zhang2011positivity}
X.~Zhang and C.-W. Shu.
\newblock {Positivity-preserving high order discontinuous Galerkin schemes for
  compressible Euler equations with source terms}.
\newblock {\em J. Comput. Phys.}, 230(4):1238--1248, 2011.

\bibitem{Zhu2018}
J.~Zhu and C.-W. Shu.
\newblock A new type of multi-resolution {WENO} schemes with increasingly
  higher order of accuracy.
\newblock {\em J.~Comput.~Phys.}, 375:659--683, 2018.

\end{thebibliography}

\end{document}